\documentclass[11pt, oneside]{amsart}
 \usepackage[text={5.58in,8.5in},centering,letterpaper,dvips]{geometry}
\usepackage{graphicx}
\usepackage{amsfonts}
\usepackage[dvipsnames]{xcolor}
\usepackage{epsf}
\usepackage{amssymb}
\usepackage{amsmath}
\usepackage{amscd}
\usepackage{pdfpages}
\usepackage{fancyhdr}
\usepackage{setspace}
\usepackage{soul} 
\usepackage[all]{xy} 
\usepackage{verbatim}
\usepackage{enumerate}
\usepackage[colorlinks=true, urlcolor=NavyBlue, linkcolor=NavyBlue, citecolor=NavyBlue]{hyperref}

\theoremstyle{theorem}
\newtheorem{theorem}{Theorem}[section]
\newtheorem{proposition}[theorem]{Proposition}
\newtheorem{lemma}[theorem]{Lemma}

\newtheorem{corollary}[theorem]{Corollary}

\newtheorem*{GPRC}{Generalized Property R Conjecture}
\newtheorem*{SGPRC}{Stable Generalized Property R Conjecture}
\newtheorem*{WGPRC}{Weak Generalized Property R Conjecture}

\newtheorem*{ELT}{Equivariant Loop and Sphere Theorems}

\makeatletter
\newtheorem*{rep@theorem}{\rep@title}
\newcommand{\newreptheorem}[2]{%
\newenvironment{rep#1}[1]{%
 \def\rep@title{#2 \ref{##1}}%
 \begin{rep@theorem}}%
 {\end{rep@theorem}}}
\makeatother

\newreptheorem{theorem}{Theorem}
\newreptheorem{lemma}{Lemma}
\newreptheorem{question}{Question}
\newreptheorem{corollary}{Corollary}
\newreptheorem{proposition}{Proposition}

\theoremstyle{definition}

\newtheorem{remark}[theorem]{Remark}

\newcommand{\Z}{\mathbb{Z}}
\newcommand{\C}{\mathbb{C}}
\newcommand{\N}{\mathbb{N}}
\newcommand{\Q}{\mathbb{Q}}
\newcommand{\R}{\mathbb{R}}

\newcommand{\id}{\text{id}}
\newcommand{\Int}{\text{Int}}

\newcommand{\Aa}{\mathcal A}

\newcommand{\Dd}{\mathcal D}

\newcommand{\Kk}{\mathcal K}
\newcommand{\Ll}{\mathcal L}

\newcommand{\Rr}{\mathcal R}
\newcommand{\Ss}{\mathcal S}
\newcommand{\Tt}{\mathcal T}

\newcommand{\wh}[1]{\widehat{#1}}

\newcommand{\A}{\alpha}
\newcommand{\n}{\beta}
\newcommand{\g}{\gamma}
\newcommand{\pd}{\partial}
\newcommand{\lk}{\ell k}
\newcommand{\emp}{\emptyset}
\newcommand{\X}{\times}
\newcommand{\wt}{\widetilde}

\makeatletter
\def\@seccntformat#1{%
  \protect\textup{\protect\@secnumfont
    \ifnum\pdfstrcmp{subsection}{#1}=0 \bfseries\fi
    \csname the#1\endcsname
    \protect\@secnumpunct
  }%
}  
\makeatother

\begin{document}

\rhead{\thepage}
\lhead{\author}
\thispagestyle{empty}


\raggedbottom
\pagenumbering{arabic}
\setcounter{section}{0}


\title{Generalized square knots and homotopy 4--spheres}
\date{\today}

\author{Jeffrey Meier}
\address{Department of Mathematics, University of Georgia, 
Athens, GA 30602}
\email{jeffrey.meier@uga.edu}
\urladdr{http://jeffreymeier.org} 

\author{Alexander Zupan}
\address{Department of Mathematics, University of Nebraska-Lincoln, Lincoln, NE 68588}
\email{zupan@unl.edu}
\urladdr{http://www.math.unl.edu/~azupan2}

\begin{abstract}
	The purpose of this paper is to study geometrically simply-connected homotopy 4--spheres by analyzing $n$--component links with a Dehn surgery realizing $\#^n(S^1\times S^2)$.  We call such links \emph{$n$R-links}.  Our main result is that a homotopy 4--sphere that can be built without 1--handles and with only two 2--handles is diffeomorphic to the standard 4--sphere in the special case that one of the 2--handles is attached along a knot of the form $Q_{p,q} = T_{p,q}\#T_{-p,q}$, which we call a \emph{generalized square knot}.  This theorem subsumes prior results of Akbulut and Gompf.
	
	Along the way, we use thin position techniques from Heegaard theory to give a characterization of 2R-links in which one component is a fibered knot, showing that the second component can be converted via trivial handle additions and handleslides to a derivative link contained in the fiber surface.  We invoke a theorem of Casson and Gordon and the Equivariant Loop Theorem to classify handlebody-extensions for the closed monodromy of a generalized square knot $Q_{p,q}$.  As a consequence, we produce large families, for all even $n$, of $n$R-links that are potential counterexamples to the Generalized Property R Conjecture.  We also obtain related classification statements for fibered, homotopy-ribbon disks bounded by generalized square knots.
	
\end{abstract}

\maketitle

\section{Introduction}\label{sec:intro}

The Smooth 4--Dimensional Poincar\'e Conjecture (S4PC) asserts that if $X$ is a \emph{homotopy 4--sphere}, a closed, smooth 4--manifold homotopy equivalent to the standard 4--sphere $S^4$, then $X$ is diffeomorphic to $S^4$.  The topological version of the S4PC was established by Freedman~\cite{Fre_The-topology-of-four-dimensional_82}, and the S4PC is the final unsettled case of the Generalized Poincar\'e Conjecture.  In 1987, David Gabai resolved the famous Property R Conjecture~\cite{Gab_Foliations-III_87}, showing that the unknot is the only knot in $S^3$ that admits a Dehn surgery yielding $S^1\times S^2$. This result can be viewed as initial progress toward a positive resolution of the S4PC, since it follows that a homotopy 4--sphere built with no 1--handles and a single 2--handle must be diffeomorphic to $S^4$.  In this paper, we extend this classification to a broader family of handle decompositions.  We refer to the knot $Q_{p,q} = T_{p,q}\#T_{-p,q}$ as a \emph{generalized square knot}.

\begin{theorem}
	\label{thmx:PC}
	Suppose that $X$ is a homotopy 4--sphere that can be built with no 1--handles and two 2--handles such that the attaching sphere of one of the 2--handles is a generalized square knot $Q_{p,q}$.  Then $X$ is diffeomorphic to $S^4$.
\end{theorem}

At first glance, this class may appear somewhat restricted; however, it includes a number of historically important examples of homotopy 4--spheres.  The first such example was the Akbulut-Kirby sphere $\Sigma_0$, which was introduced by Cappell and Shaneson in 1976~\cite{CapSha_Some-new-four-manifolds_76}, studied in detail by Akbulut and Kirby in 1985~\cite{AkbKir_A-potential-smooth_85}, and shown to be standard by Gompf in 1991~\cite{Gom_Killing-the-Akbulut-Kirby_91}.  Subsequently, Gompf drew handlebody diagrams for an infinite family $\{\Sigma_m\}$ of Cappell-Shaneson homotopy spheres in 1991~\cite{Gom_On-Cappell-Shaneson-4-spheres_91}.  This family remained one of the most prominent classes of potential counterexamples to the S4PC (see~\cite{manmachine}) until Akbulut showed that each $\Sigma_m$ is standard in his celebrated 2010 paper~\cite{Akb_Cappell-Shaneson-homotopy_10}.  Another infinite family $H(n,k)$ generalizing $\Sigma_0$ was introduced and standardized by Gompf~\cite{Gom_Killing-the-Akbulut-Kirby_91} (cf. Figure~14 of~\cite{GomSchTho_Fibered-knots_10}).  Each of these examples satisfies the hypotheses of Theorem~\ref{thmx:PC}, so the present approach subsumes the proofs that these manifolds are standard.  Moreover, the methods here are qualitatively different than the other approaches; whereas past results involved techniques to simplify specific handle decompositions, our work is a more flexible characterization of a substantially larger collection of homotopy 4-spheres.

A 4--manifold that can be built without 1--handles is called \emph{geometrically simply-connected}.  If $X$ is a geometrically simply-connected 4--manifold that can be built with a single 0--handle, $n$ 2--handles, $n$ 3--handles, and a single 4--handle, then $\chi(X) = 2$ and $X$ is a homotopy 4--sphere. Since the attaching map of the 3--handles is unique up to isotopy~\cite{LauPoe_A-note-on-4-dimensional_72}, the manifold $X$ is completely characterized by the attaching spheres of the 2--handles, an $n$--component link $L$ in $S^3$ with a Dehn surgery to the manifold $\#^n(S^1 \X S^2)$, which we denote by $Y_n$. (The framings and linking numbers for $L$ are determined by this Dehn surgery and must all be zero.)  We call an $n$--component link $L$ with the property that 0--surgery on $L$ yields $Y_n$ an \emph{R-link} (or an \emph{$n$R-link} when we wish to emphasize the number of components).  Conversely, every R-link $L$ determines a handle decomposition of a 4--manifold we denote $X_L$, and the above arguments imply that $X_L$ is a geometrically simply-connected homotopy 4--sphere.

In this vein, Gabai's result establishes that the unknot is the only 1R-link. This simple structure quickly disappears for $n>1$, since handleslides of $L$ preserve the result of Dehn surgery.  The Generalized Property R Conjecture (GPRC) asserts that, modulo handleslides, the only R-link is the unlink.

\begin{GPRC}
	Every R-link is handleslide-equivalent to an unlink.
\end{GPRC}

If $U$ is an unlink, then the induced handle decomposition of $X_U$ contains canceling 2--handle/3--handle pairs, implying that $X_U$ can be built with only a 0--handle and 4--handle, so $X_U$ is diffeomorphic to the standard $S^4$.  The same is true for any link $L$ handleslide-equivalent to $U$, and thus, the GRPC implies the S4PC for geometrically simply-connected 4--manifolds.  In this case, we way that $L$ has \emph{Property R}.  There are other, weaker versions of the GPRC, which also have the same implication. We denote the split union of two links $L_1$ and $L_2$ by $L_1\sqcup L_2$.

\begin{SGPRC}
	For every R-link $L$, there is a 0--framed unlink $U$ such that $L \sqcup U$ is handleslide-equivalent to an unlink.
\end{SGPRC}

A \emph{Hopf pair} is a Hopf link where one component is 0--framed, while the other is decorated with a dot and encodes a 4--dimensional 1--handle in the standard way. (See~\cite{GomSti_4-manifolds-and-Kirby_99} for details regarding handlebody calculus for 4--manifolds.)

\begin{WGPRC}
	For every R-link $L$, there is a 0--framed unlink $U$ and a split collection of Hopf pairs $V$ such that $L \sqcup U \sqcup V$ is handleslide-equivalent to an unlink and a split collection of Hopf pairs.
\end{WGPRC}

As above, if $L$ satisfies the weak/stable GPRC, we say that $L$ has \emph{Weak/Stable Property~R}.  If $L$ has Stable Property R, then the handle decomposition of $X_L$ can be converted to the standard handle decomposition of $S^4$ after adding some canceling 2--handle/3--handle pairs (corresponding to the unlink $U$).  If $L$ has Weak Property R, the handle decomposition of $X_L$ can be made standard after adding both canceling 1--handle/2--handle pairs and canceling 2--handle/3--handle pairs.  It follows from Cerf Theory that the Weak GPRC is equivalent to the S4PC for geometrically simply-connected 4--manifolds.

Following \cite{GomSchTho_Fibered-knots_10}, we say that a given knot $K$ in $S^3$ has (\emph{weak}/\emph{stable}) \emph{Property $n$R} if for every $n$R-link $L$ having $K$ as a constituent knot, $L$ has (Weak/Stable) Property $n$R.  Using this language, we can give a slightly stronger restatement of  Theorem~\ref{thmx:PC}.

\begin{theorem}
\label{thmx:weak}
	Every generalized square knot $Q_{p,q}$ has Weak Property 2R; moreover, any 2R-link containing $Q_{p,q}$ an be simplified after adding at most two Hopf pairs.
\end{theorem}

As mentioned above, this proves the S4PC for a class of geometrically simply-connected homotopy 4--spheres, including those standardized by Gompf in 1991~\cite{Gom_Killing-the-Akbulut-Kirby_91} and Akbulut in 2010~\cite{Akb_Cappell-Shaneson-homotopy_10}.  Notably, our approach differs dramatically from previous work; in particular, no (explicit) use of a fishtail neighborhood is made here. See Subsection~\ref{subsec:classical} for details.

\begin{corollary}
\label{corox:classical}
	The Cappell-Shaneson homotopy 4--spheres $\Sigma_m$ and the Gompf homotopy 4--spheres $H(n,k)$ are standard.
\end{corollary}

The main theorem is also interesting from the perspective of the GPRC and the Stable GPRC, since the consensus appears to be that neither of these two conjectures is likely to be true.  In \cite{GomSchTho_Fibered-knots_10}, Gompf, Scharlemann, and Thompson produced a family $\{L_n\}$ of potential counterexamples to the GPRC (building on work of Akbulut and Kirby \cite{AkbKir_A-potential-smooth_85} and Gompf~\cite{Gom_Killing-the-Akbulut-Kirby_91}), in which each $L_n$ is a 2--component R-link with a square knot component.  If $L_n$ has Property R, then the trivial group presentation
$$P_n = \langle x,y \, | \, xyx=yxy, x^n= y^{n+1} \rangle$$
satisfies the Andrews-Curtis Conjecture~\cite{AndCur_Free-groups_65}, which is widely believed \emph{\textbf{not}} to be the case when $n \geq 3$.  See \cite{GomSchTho_Fibered-knots_10} for further details about the Andrews-Curtis Conjecture.

The family $\{L_n\}$ of 2R-links, which have the property that one component is the square knot $Q_{3,2}$, was further studied and characterized by Scharlemann in \cite{Sch_Proposed-Property_16}.  We expand on Scharlemann's characterization to produce, for each generalized square knot $Q_{p,q}$, an infinite family of R-links having $(p-1)(q-1)$ components, most of which appear to be potential counterexamples to the GPRC.  These are the first potential counterexamples having more than two components.

\begin{proposition}
\label{propx:nR}
	Fix a generalized square knot $Q_{p,q}$. For $n = (p-1)(q-1)$ and for any $c/d\in\Q$ with $c$ even, there is an $n$R-link $L_{c/d}^{p,q}$ contained in a fiber for $Q_{p,q}$.
\end{proposition}

In Section~\ref{sec:tris}, we revisit a program by which to disprove the GPRC and Stable GPRC using the theory of 4--manifold trisections introduced by Gay and Kirby \cite{GayKir_Trisecting-4-manifolds_16}.  We show how to associate a natural trisection to the homotopy 4--sphere $X_L$ corresponding to an R-link $L$, and we describe explicit trisection diagrams for these trisections in the case of the 4--manifolds $X_{L^{p,q}_{c/d}}$ associated to the R-links of Theorem~\ref{propx:nR}.  An R-link $L$ satisfies the Stable GPRC precisely when these natural trisections have a certain stable property.

The relevant characteristics of a generalized square knot are that they are ribbon, fibered, and have periodic monodromy. In the course of proving Theorem~\ref{thmx:PC}, we also prove the next theorem, which may be of independent interest. By the \emph{closed monodromy} of a fibered knot $K$ in $S^3$, we mean the monodromy of the associated closed surface-bundle obtained as 0--surgery on $K$.

\begin{theorem}
\label{thmx:CGequiv}
	If $L = Q \cup J$ is a 2R-link and $Q$ is nontrivial and fibered, then there is an unlink $U$ such that $Q \cup J\sqcup U$ is handleslide-equivalent to $Q \cup L^+$, such that
	\begin{enumerate}
		\item $L^+$ is $n$--component link with $n = g(Q)$,
		\item $L^+$ is contained in a fiber $F$ of $Q$, and
		\item the closed monodromy of $Q$ extends over the handlebody determined by $L^+$.
	\end{enumerate}
\end{theorem}

The proof of Theorem~\ref{thmx:CGequiv} revolves around the theory of Heegaard splittings of 3--manifolds and thin position arguments initiated by Scharlemann and Thompson \cite{ST}.  This theorem could potentially be used to prove that all fibered, homotopy-ribbon knots have Weak Property 2R. 

The link $L^+$ in Theorem~\ref{thmx:CGequiv} has a special name; we call it a \emph{Casson-Gordon derivative}, in reference to the seminal work of Casson and Gordon characterizing the monodromies for fibered, homotopy-ribbon knots \cite{CasGor_A-loop-theorem_83}:  A fibered knot $K$ is homotopy-ribbon in a homotopy 4-ball if and only if the closed monodromy of $K$ extends across a handlebody. Moreover, such an extension encodes a fibered, homotopy-ribbon disk-knot bounded by $K$. (By a \emph{disk-knot} we mean a properly embedded disk $D$ in homotopy 4--ball $B$.) Thus, the following classification of fibered, homotopy-ribbon disk-knots bounded by generalized square knots is closely related to Theorem~\ref{thmx:PC}.

\begin{theorem}
\label{thmx:disks}
	There is a family $\{(B_{c/d},R_{c/d})\}$ of fibered, homotopy-ribbon disk-knots  for $(S^3,Q_{p,q})$, indexed by $c/d\in\Q$ with $c$ even, such that
\begin{enumerate}
	\item $(B_0,R_0)$ is the product ribbon disk $(B^3,T_{p,q}^\circ) \X I$;
	\item The members of $\{(B_{c/d},R_{c/d})\}$ are pairwise non-diffeomorphic rel-$\partial$;
	\item For any fibered, homotopy-ribbon disk-knot $(B,R)$ for $(S^3,Q_{p,q})$, we have $(B,R)\in \{(B_{c/d},R_{c/d})\}$; and
	\item The members of $\{(B_{c/d},R_{c/d})\}$ have diffeomorphic exterior.
\end{enumerate}
\end{theorem}

Finally, we return to the notion of extending a mapping class across a handlebody.  Long showed that there exists a fibered knot whose closed (pseudo-Anosov) monodromy admits extensions over two distinct handlebodies~\cite{Lon_On-pseudo-Anosov-maps_90}.  In general, for a knot with pseudo-Anosov monodromy, only finitely many extensions are possible~\cite{CasLon_1985_Algorithmic-compression}.  We give the following analogue of Long's result for generalized square knots.  By the theorem of Casson and Gordon, each CG-derivative $L_{c/d}^{p,q}$ gives rise to an extension $\phi^{p,q}_{c/d}$ of the closed monodromy $\wh\varphi^{p,q}$ of the generalized square knot $Q_{p,q}$.

\begin{theorem}
\label{thmx:extensions}
	Every handlebody-extension of $\wh\varphi^{p,q}$ is isotopic to $\phi_{c/d}^{p,q}$, for some $c/d\in\Q$ with $c$ even, and each $\phi_{c/d}^{p,q}$ represents an extension of $\wh\varphi^{p,q}$ over a distinct handlebody for each choice of $c/d\in\Q$ with $c$ even.
\end{theorem}

The common element of many of these theorems is the rational number $c/d$:  For a fixed $p$ and $q$, the exterior $B_{c/d}\setminus\nu(R_{c/d})$ is given by the handlebody-bundle $H\times_{\phi_{c/d}}S^1$, and the link $L_{c/d}$ bounds a cut system for $H$ in this extension.

\subsection*{Organization}

In Section~\ref{sec:prelim}, we state general preliminary material and give detailed discussions of disk-knots, R-links, and fibered, homotopy-ribbon knots in the context of the theorem of Casson and Gordon.
In Section~\ref{sec:stable2R}, we turn our attention to the theory of Heegaard splittings of 3--manifolds and apply thin position arguments to prove Theorem~\ref{thmx:CGequiv}.
In Section~\ref{sec:square}, we give a detailed account of generalized square knots, including a careful analysis of the fibrations of their exteriors and of their 0--surgeries.
In Section~\ref{sec:L0}, we give a detailed analysis of the simplest Casson-Gordon derivative for a generalized square knot and show that this link has Property R.
In Section~\ref{sec:autom}, we describe a pair of automorphisms of the Seifert fibered space obtained as zero-surgery on a generalized square knot that are given by twisting along vertical tori.  These automorphisms are the key ingredient in the final part of the proof of our main results.
In Section~\ref{sec:std}, we give proofs of Theorems~\ref{thmx:PC} and~\ref{thmx:weak} by considering certain handle decompositions of the Casson-Gordon homotopy 4--spheres corresponding to extensions of the closed monodromy of generalized square knots that are well adapted to the automorphisms referenced above.
In Section~\ref{sec:class_CG}, we turn our attention to a final analysis of monodromy extensions and disk-knots and prove Theorems~\ref{thmx:disks} and~\ref{thmx:extensions}.
 In Section~\ref{sec:tris}, we give trisections for Casson-Gordon homotopy 4--spheres and  discuss connections between the theory of trisections, the GPRC, and the Slice-Ribbon Conjecture arising from considerations of R-links and fibered, homotopy-ribbon knots.  

\subsection*{Acknowledgements}

The authors would like to thank the following people for their interest in this project and for helpful conversations: Mark Brittenham, Christopher Davis, Bob Gompf, Cameron Gordon,  Kyle Larson, Tye Lidman, Tom Mark, Maggie Miller, Marty Scharlemann, and Abby Thompson.

The first author was supported by NSF grants DMS-1400543 and DMS-1758087, and the second author was supported by NSF grant DMS-1664578 and NSF-EPSCoR grant OIA-1557417.

\section{Preliminaries}
\label{sec:prelim}

We begin with some standard declarations.  All manifolds are smooth and orientable unless specified.  If $Y \subset X$, we let $\nu(Y)$ denote an open regular neighborhood of $Y$ in $X$, and for ease of notation, we let $X \setminus Y = X - \nu(Y)$. The term \emph{$n$--dimensional genus $g$ handlebody} refers to the compact orientable $n$--manifold constructed by attaching $g$ $n$--dimensional 1--handles to an $n$--dimensional 0--handle.  We use the word \emph{handlebody} to mean a 3--dimensional handlebody; otherwise, we will specify dimension.  Let~$L$ be a framed link in $S^3$, with components $L_1$ and $L_2$ (and possibly others).  A \emph{handleslide} of $L_1$ over $L_2$ is the process by which $L$ is replaced with $L' = (L - L_1) \cup L_1'$, where $L_1'$ is the framed knot obtained by connecting $L_1$ to $L_2$ with a band. (See Section~5 of~\cite{GomSti_4-manifolds-and-Kirby_99} for complete details.)  If a link $L'$ can be obtained from $L$ by a finite sequence of handleslides, we say $L$ and $L'$ are \emph{handleslide-equivalent}.  If $U$ and $U'$ are unlinks and $L \sqcup U$ is handleslide-equivalent to $L' \sqcup U'$, we say $L$ and $L'$ are \emph{stably equivalent}.  Note that two stably equivalent R-links $L$ and $L'$ give rise to diffeomorphic 4--manifolds $X_L$ and $X_{L'}$.  A \emph{curve} contained in a surface $\Sigma$ is a free homotopy class of a simple loop that does not bound a disk in $\Sigma$ and is not parallel to a component of $\pd \Sigma$.

\subsection{Slice knots and links}
\label{subsec:Slice}
\ 

Throughout this section, let $B$ be a homotopy 4--ball; i.e., $B$ is a smooth, contractible 4--manifold with $\partial B \cong S^3$.  By~\cite{Fre_The-topology-of-four-dimensional_82}, $B$ is homeomorphic to $B^4$, the standard smooth 4--ball; it is unknown in general whether $B$ and $B^4$ are diffeomorphic.  A collection $\Dd$ of smooth, properly embedded disks in $B$ is called a \emph{disk-link}, or a \emph{disk-knot} if $\Dd$ is a single disk. A disk-link is called \emph{homotopy-ribbon} if the natural inclusion map $(S^3,\partial \Dd)\hookrightarrow(B,\Dd)$ induces a surjection $\pi_1(S^3\setminus\partial \Dd)\twoheadrightarrow\pi_1(B\setminus \Dd)$.  A disk-link $\Dd$ in $B^4$ is called \emph{ribbon} if $\Dd$ can be isotoped to have no local maxima with respect to the radial height function on $B^4$.

A link $L\subset S^3$ is called \emph{slice in $B$} (resp., \emph{homotopy-ribbon in $B$}) if $(S^3,L) = \partial(B,\Dd)$ for a disk-link (resp., homotopy-ribbon disk-link) $\Dd$ in some homotopy 4-ball $B$.  If $L$ is slice in $B^4$ (resp., homotopy-ribbon in $B^4$), we simply call $L$ \emph{slice} (resp., \emph{homotopy-ribbon}).  Finally, if $L$ bounds a ribbon disk-link in $B^4$, we say that $L$ is \emph{ribbon}.  These collections of links are related as follows:
$$\{\text{ribbon links}\}\subset\{\text{homotopy-ribbon links}\}\subset\{\text{slice links}\}.$$
Moreover, it is unknown whether any of the above set inclusions are set equalities. The notion of homotopy-ribbon links was introduced in~\cite{CasGor_A-loop-theorem_83}, while the notions of slice knots and ribbon knots date back to Fox~\cite{Fox_A-quick-trip_62,Fox_Some-problems_62}, who posited the famous Slice-Ribbon Conjecture, which asserts that every slice knot is ribbon.


For a link $L \subset S^3$, we will set the convention that $Y_L$ denotes the 3--manifold obtained by zero-framed Dehn surgery on each component of $L$.  In addition, define the \emph{exterior} of $L$ to be $E_L = S^3 \setminus L$.  Similarly, if $\Dd$ is a disk link in $B$, we define the \emph{exterior} of $\Dd$ to be $E_{\Dd} = B \setminus \Dd$.

\begin{lemma}
If $L$ is the boundary of a disk link $\Dd \subset B$, then $\pd E_{\Dd} = Y_L$.
\end{lemma}

\begin{proof}
The boundary of the exterior $E_\Dd$ admits the following decomposition:
$$\partial E_\Dd = E_L \cup (\overline{\partial\nu(\Dd) \cap \Int(B)}).$$

The second factor is diffeomorphic to $c$ disjoint copies of $S^1\times D^2$.  Thus, $\partial E_\Dd$ is the result of some Dehn surgery on $L$.  Note that $H_1(E_L) = H_1(E_\Dd) = \Z^c$, and the map on $H_1$ induced by the inclusion $E_L \hookrightarrow E(\Dd)$ is an isormophism.  Note, however, that this inclusion factors as $E_L \hookrightarrow \pd E_\Dd \hookrightarrow E_\Dd$, and thus $H_1(\pd E_\Dd) = \Z^c$ as well.  It follows that the framing of the Dehn surgery on $L$ yielding $\pd E_\Dd$ is the 0--framing, so that $\pd E_\Dd = Y_L$.
\end{proof}

Recall that an $n$--component link $L$ in $S^3$ is an \emph{R-link} if zero-framed surgery on $L$ gives $Y_n=\#^n(S^1\times S^2)$.

\begin{proposition}\label{lem:cR=>hrib}
	Every R-link is homotopy-ribbon in a homotopy 4--ball.
\end{proposition}

\begin{proof}
	Suppose $L\subset S^3$ is an R-link, and let $B$ be the 4--manifold obtained by attaching zero-framed 2--handles to the components of $L$ and capping off the resulting surgery manifold, which is $Y_n$ by hypothesis, with $n$ 3--handles and a 4--handle (i.e. $B$ is $X_L$ without its 0--handle).  Let $\Dd$ denote the cores of the 2--handles.  Then $\partial (B,\Dd) = (S^3,L)$, so it remains to show that $B$ is a homotopy 4--ball and that $\Dd$ is a homotopy-ribbon disk knot.
	
	The first claim follows from the fact that $B$ is built from $S^3$ without 1--handles, so it is simply-connected and $\chi(B) = 1$.  This implies, by theorems of Whitehead and Hurewicz, that $B$ is homotopic to a point (Corollary 4.33 of~\cite{Hat_Algebraic-topology_02}).
	
	To verify the second claim, observe that $Y_L$ is obtained by Dehn filling $E_L$, and thus the inclusion $E_L \hookrightarrow Y_L$ induces a surjection $\pi_1(E_L)\twoheadrightarrow\pi_1(Y_L)$.  In addition, $Y_L = \#^n (S^1 \X S^2) = \pd (E_{\Dd})$, where $E_{\Dd}$ is a 4--dimensional handlebody of genus~$n$, since it is composed of $n$ 3--handles and a 4--handle. Hence, $\pi_1(Y_L) = \pi_1(E_{\Dd})$, the free group on $n$ letters, and the inclusion $Y_L \hookrightarrow E_{\Dd}$ induces an isomorphism of fundamental groups.  It follows that $\pi_1(E_L)$ surjects onto $\pi_1(E_{\Dd})$, and $L$ is homotopy-ribbon in $B$.
\end{proof}

Note that the proof shows something even stronger:  Every $n$--component R-link $L$ is the boundary of a homotopy-ribbon disk-knot whose complement has free fundamental group of rank $n$.

\subsection{Fibered, homotopy-ribbon knots}
\label{subsec:FhRKs}
\ 

Let $X$ be a compact manifold, and let $\Phi\colon X\to X$ be a diffeomorphism.  The \emph{mapping torus} is the identification space
$$X\times_\Phi S^1 = (X\times I)/\sim,$$
where $I = [0,1]$ and $\sim$ is the equivalence relation $(x,1)\sim(\Phi(x),0)$ for all $x\in X$.  Note that in the case that $\partial X\not=\emptyset$, the boundary of a mapping torus is a mapping torus:
$$\partial(X\times_\Phi S^1) \cong \partial X\times_{\Phi\vert_{\partial X}}S^1.$$
The map $\Phi$ is called the \emph{monodromy}, and, for each $\theta\in S^1 = \{e^{2\pi i \theta} \in \C\}$, the submanifold $X\times_\Phi\{\theta\} = (X \times \{\theta\})/\sim \,\,\,\subset (X \X I)/\sim$ is called a \emph{fiber}. Recall that a knot $K\subset S^3$ is called \emph{fibered} if the knot exterior is the mapping torus
$$E_K \cong F\times_\varphi S^1,$$
with $\varphi\vert_{\partial F} = \id$.

Suppose that $K$ is a fibered knot, with 0--framed filling on $E_K$ denoted $Y_K$, as above.  Then $Y_K = \left(F \X_{\varphi} S^1\right) \cup (D^2 \X S^1)$, where $\pd F \X_{\varphi} \{\theta\} = \pd D^2 \X \{\theta\}$ for all $\theta \in S^1$, and thus this gluing has the effect of capping off each fiber with a disk.  Moreover, using $\varphi\vert_{\partial F} = \id$, we can (uniquely) extend $\varphi$ as the identity over this disk to a diffeomorphism $\wh\varphi\colon \wh F\to \wh F$, where $\wh F$ is the closed surface $F\cup D^2$.  It follows that $Y_K$ is a closed surface bundle $\wh F\times_{\wh\varphi}S^1$.  We call $\wh\varphi$ the \emph{closed monodromy} of $K$.

We say that a diffeomorphism $\wh\varphi$ of a closed surface $\wh F$ admits an \emph{extension} if there is a handlebody $H$ with $\partial H = \wh F$ and a diffeomorphism $\Phi\colon H\to H$ such that $\wh\varphi = \Phi\vert_{\wh F}$. Note that we are restricting our attention exclusively to the case where the monodromy extends over a handlebody, as opposed to the more general cases where it might extend over a compression body or a more general type of 3--manifold. An elegant characterization of fibered, homotopy-ribbon knots was given by Casson and Gordon.

\begin{theorem}\cite{CasGor_A-loop-theorem_83}\label{thm:CGEX}
	A fibered knot $K$ in $S^3$ is homotopy-ribbon in a homotopy 4-ball if and only if the closed monodromy for $K$ admits an extension.
\end{theorem}

The characterization also relates an extension of the monodromy of a homotopy-ribbon knot to the topology of a homotopy-ribbon disk exterior via the following corollary.

\begin{corollary}\cite{CasGor_A-loop-theorem_83}\label{coro:CGEX2}
	Suppose $K$ is a fibered knot in $S^3$ with monodromy~$\varphi$.  If $K$ is homotopy-ribbon in some homotopy 4--ball $B$, then there is an extension $\Phi$ of $\wh \varphi$ and a disk $R_\Phi$ in a homotopy 4--ball $B_{\Phi}$ such that that
\[ B_{\Phi} \setminus R_\Phi \cong H \X_{\Phi} S^1.\]
\end{corollary}

The next lemma and discussion following it outlines a connection between the Casson-Gordon Theorem and the notion of R-links introduced above.  The lemma is well-known, but we offer a proof that is motivated by the techniques used later in this paper.

\begin{lemma}\label{handlebundle}
	Suppose that $\wh F$ is a genus $g$ surface bounding a handlebody $H$, and let~$\Phi:H \rightarrow H$ be a diffeomorphism.  Then $H \X_{\Phi} S^1$ has a handle decomposition with a 0--handle, $g+1$ 1--handles, and $g$ 2--handles.
\end{lemma}

\begin{proof}
	We will show that $H \X_{\Phi} S^1$ can be constructed by gluing $g$ 4--dimensional 2--handles to $(\wh F \X_{\Phi\vert_{\wh F}} S^1) \X I$ followed by attaching $(g+1)$ 3--handles and a 4--handle.  Inverting this decomposition gives the desired result.  Let $L$ be a collection of $g$ pairwise disjoint curves in $\wh F$ that bound a collection $\Dd$ of disks in $H$.  Then $\wh F$ can be capped off with $g$ 3--dimensional 2--handles along $L$ and one 3--dimensional 3--handle to obtain $H$, and thus a collar $\wh F \X I$ can be capped off with $g$ 4--dimensional 2--handles along $L$ and one 4--dimensional 3--handle to obtain $H \X I$.  Consider a collar neighborhood of $H \X_{\Phi} \{0\} = H \X_{\Phi} \{1\} \subset H \X_{\Phi} S^1$, whose complement is $H \X_{\Phi} [\epsilon,1-\epsilon] \cong H \X I$.  Since $H \X I$ is a 4--dimensional genus $g$ handlebody, it can be built with $g$ 3--handles and a 4--handle.  Thus, $H \X_{\Phi} S^1$ can be obtained by attaching $g$ 2--handles to $(\wh F \X_{\Phi\vert_{\wh F}} S^1) \X I$ along $L$ followed by attaching $(g+1)$ 3--handles and a 4--handle.
\end{proof}



Let $F$ be a genus $g$ Seifert surface for a knot $K \subset S^3$.  A $g$--component link $L = L_1\cup\cdots\cup L_g$ contained in $F$ is called a \emph{derivative} for $K$ if the classes $[L_i]$ are independent in $H_1(F)$ and if $\lk(L_i,L_j) = 0$ for all $i,j$, where $\lk(L_i,L_i)$ is calculated with a pushoff of $L_i$ in $F$.  In light of the previous lemma, suppose that $K$ is a fibered knot with Seifert surface $F$ and monodromy $\varphi$.  Let $L \subset F$ be a derivative for $K$, and let $H$ be the (abstract) handlebody determined by $L$.  We call $L$ a \emph{CG-derivative} (short for \emph{Casson-Gordon derivative}) if the closed monodromy $\wh \varphi$ admits an extension to $H$.  CG-derivatives are central to this paper, as indicated by Theorem \ref{thmx:CGequiv} and the next proposition.

\begin{proposition}\label{CGR}
Suppose $K$ is a fibered knot with CG-derivative $L$.  Then both $L$ and $K \cup L$ are R-links.
\end{proposition}
\begin{proof}
As above, let $F$ be a genus $g$ fiber for $K$, let $H$ be the handlebody determined by $L$, let $\varphi$ be the monodromy of $K$, and let $\Phi$ be an extension of $\wh \varphi$ to $H$.  We construct a compact 4--manifold $B_{\Phi}$ by the following process:  First, attach a 0--framed 2--handle to $S^3 \X I$ along $K \X \{1\}$.  The resulting 4--manifold has two boundary components, one diffeomorphic to $S^3$ and the other diffeomorphic to $Y_K$, the result of 0--surgery on $K$.  Since $Y_K = \wh F \X_{\wh \varphi} S^1$, we can cap off this boundary component with $H \X_{\Phi} S^1$ to get a compact 4--manifold we call $B_{\Phi}$, where $\pd B_{\Phi} \cong S^3$.  By Lemma~\ref{handlebundle}, $B_{\Phi}$ has a handle decomposition relative to its boundary with $g+1$ 2--handles and $g+1$ 3--handles, and thus the attaching link for the 2--handles, namely $K \cup L$, is an R-link in $S^3$.

To see that $L$ is also an R-link by itself, we note that $F \setminus L$ is a connected planar surface with $2g+1$ boundary components, one of which corresponds to $K$.  As such, there is a sequence of handleslides of $K$ over the components of $L$ that takes $K$ to $K'$, where $K'$ bounds a disk in $F \setminus L$.  Thus, after handleslides, the 2--handle that attaches to $K'$ in the handle decomposition of $B_{\Phi}$ cancels a 3--handle, and so $B_{\Phi}$ can be built with $g$ 2--handles and $g$ 3--handles, where $L$ is the R-link that serves at the attaching link for the 2--handles.
\end{proof}

We call the manifold $B_\Phi$ a \emph{Casson-Gordon homotopy 4--ball}, or \emph{CG-ball}, for short.  Since $\partial B_\Phi = S^3$, we can cap off $B_\Phi$ with a standard $B^4$ to obtain a homotopy four-sphere $X_\Phi$, which we call a \emph{Casson-Gordon homotopy 4--sphere}, or \emph{CG-sphere}, for short.

Let $R_\Phi$ denote the core of the 2--handle that is attached along $K$ in the handle decomposition of $B_\Phi$ described above.  We have seen (cf. Corollary~\ref{coro:CGEX2}) that $R_\Phi$ is a fibered, homotopy-ribbon disk for for $K$ in $B_\Phi$.  We call $R_\Phi$ a \emph{Casson-Gordon disk}, or \emph{CG-disk} for short.  We refer to $(B_\Phi, R_\Phi)$ as a \emph{CG-pair}.

Finally, Casson and Gordon also provided a useful criterion to decide whether a given derivative is a CG-derivative, using only the the action $\wh\varphi_*$ of the closed monodromy $\wh\varphi$ on $\pi_1(\wh F)$.  For any derivative $L$ for $K$ in $F$, let $N$ denote the normal subgroup of $\pi_1(\wh F)$ generated by the homotopy classes of the components of $L$.  Observe that if $L$ is a CG-derivative, then $\wh \varphi_*(N) = N$, since there is an extension $\Phi$ of $\wh \varphi$ and thus $\wh \varphi$ preserves the kernel of the map $\pi_1(\wh F)\twoheadrightarrow \pi_1(H)$ induced by inclusion, which is equal to $N$.  Casson and Gordon strengthened this connection with the following converse.

\begin{proposition}\cite{CasGor_A-loop-theorem_83}
\label{cggen}
	Let $L \subset F$, where $F$ is a Seifert surface for a fibered knot $K$ with closed monodromy $\wh \varphi$, such that $F \setminus L$ is a connected planar surface, and let $N$ be the normal subgroup of $\pi_1(\wh F)$ generated by the homotopy classes of the components of $L$.  If $\wh \varphi_*(N) = N$, then $L$ is a CG-derivative.
\end{proposition}

\section{Stable equivalence classes of 2R-links}
\label{sec:stable2R}

In this section, we prove Theorem \ref{thmx:CGequiv}, which asserts that if $L = Q \cup J$ is a 2R-link and $Q$ is fibered, then $L$ is stably equivalent to the link $Q \cup L^+$, where $L^+$ is a CG-derivative for $Q$.  The machinery used in the proof of this theorem includes a decomposition called a Heegaard double (cf. \cite{GomSchTho_Fibered-knots_10}), along with ideas from thin position of Heegaard splittings.  The notation of this section is basically self-contained; we will use letters and symbols here that denote unrelated objects elsewhere.

Let $S$ be a closed surface with one or two components; in the two-component case, suppose neither component is a 2--sphere.  Consider the product $S \X I$, letting $S^+ = S \X \{1\}$ and $S^- = S \X \{0\}$.  Let $\Delta^+$ be a pair of disks contained in $S^+$, let $\Delta^-$ be a pair of disks in $S^-$, and let $\Delta$ be the four disks $\Delta^+ \cup \Delta^-$.  We require that if $S$ is disconnected, then $\Delta^{\pm}$ contains one disk in each component of $S^{\pm}$.  Attach 1--handles $H^{\pm}$ to $S^{\pm}$ along $\Delta^{\pm}$.  We let $\Sigma^{\pm}$ denote the resulting two boundary components of $(S \X I) \cup H^- \cup H^+$, noting that $\Sigma^+$ and $\Sigma^-$ are connected, even if $S$ is disconnected.  Finally, suppose $h\colon\Sigma^+ \rightarrow \Sigma^-$ is a diffeomorphism.  Then we can build a 3--manifold $Y$ by gluing $\Sigma^+$ to $\Sigma^-$ via $h$, and we call such a decomposition $(Y; S, \Delta,h)$ a \emph{Heegaard double}, observing that $S$, $\Delta$, and $h$ uniquely determine $Y$.  We let $c^{\pm}$ denote the boundary of the co-core $D^{\pm}$ of the 1--handle $H^{\pm}$, so that $c^{\pm}$ bounds a compressing disk for $\Sigma^{\pm}$.  Note that $c^{\pm}$ is non-separating if and only if $S$ is connected, and thus either both $c^+$ and $c^-$ are separating or both are non-separating in $\Sigma^{\pm}$.  In addition, requiring that $S$ does not have a 2--sphere component in the disconnected case guarantees that $c^{\pm}$ is an essential curve in $\Sigma^{\pm}$. See Figure~\ref{fig:HeegDouble}.

The definition of a Heegaard double can be generalized to allow $H^{\pm}$ to represent multiple 1--handles, but all Heegaard doubles in the present article will be of the type described above, where $\chi(S) = \chi(\Sigma^{\pm})+2$. The observant reader will note that this definition is not the same as that of \cite{GomSchTho_Fibered-knots_10}; however, we can obtain their version from ours by cutting $S \X I$ open along $S \X \{1/2\}$.  This yields two compression bodies, $C^+$ and $C^-$, where $\pd_- C^+$ and $\pd_- C^-$ are identified via the identity map and $h\colon \pd_+ C^+ \rightarrow \pd_+ C^-$ is the other gluing map.  If $(Y;S,\Delta,h)$ is a Heegaard double, we let $Y^* = C^- \cup_h C^+$, so that $Y^* = Y \setminus (S \X \{1/2\})$.  Note that the decomposition $Y^* = C^- \cup_h C^+$ is a Heegaard splitting.  This Heegaard splitting is called \emph{reducible} if there is an essential curve $c \in \pd_+ C^+$ such that $c$ bounds a disk in $C^+$ and $h(c)$ bounds a disk in $C^-$. 

\begin{figure}[h!]
	\centering
	\includegraphics[width=.2\textwidth]{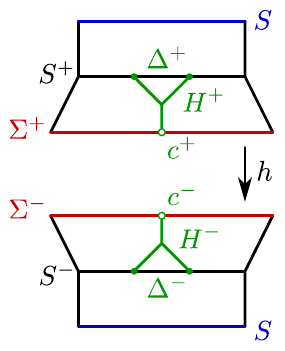}
	\caption{A schematic of a Heegaard double $(Y;S,\Delta,h)$, where $Y$ is split into two compression bodies $C^\pm$ with $\Sigma^\pm = \partial_+C^\pm$ and $\partial_-C^- = \partial_-C^+ = S$.}
	\label{fig:HeegDouble}
\end{figure}

The next lemma also appears as Proposition 4.2 in~\cite{GomSchTho_Fibered-knots_10}.

\begin{lemma}
\label{samecurve}
	Suppose $(Y;S,\Delta,h)$ is a Heegaard double such that $h(c^+)$ is isotopic to $c^-$ in $\Sigma^-$.
	\begin{enumerate}
		\item If $c^{\pm}$ is non-separating in $\Sigma^{\pm}$, then $Y \cong (S^1 \X S^2) \# Y'$, where $Y'$ is a fibered 3--manifold with fiber $S$.
		\item If $c^{\pm}$ is separating in $\Sigma^{\pm}$, then either $Y \cong Y' \# Y''$ or $Y \cong (S^1 \X S^2) \# Y'$, where $Y'$ and $Y''$ are fibered 3--manifolds with fibers given by the two components of $S$. 
	\end{enumerate}
\end{lemma}

\begin{proof}
	The assumption $h(c^+) = c^-$ implies that the Heegaard splitting $Y^* = C^- \cup_h C^+$ is reducible.  If $c^-$ is non-separating, this Heegaard splitting can be expressed as the connected sum of a genus one splitting of $S^1 \X S^2$ and the splitting given by compressing $C^{\pm}$ along $D^{\pm}$ and gluing the resulting pieces together along the compressed positive boundary surfaces.  But $C^{\pm}$ compressed along $D^{\pm}$ is the trivial compression body $\pd_- C^{\pm} \X I$, and thus $Y^* = (S^1 \X S^2) \# ((\pd_- C^+ \X I) \cup (\pd_- C^- \X I))$.  To recover $Y$ from $Y^*$, we glue the two boundary components of $\pd_- C^{\pm} \X I$, yielding the desired result.

For the second statement, suppose that $h(c^+) = c^-$ is separating, so that the disks $D^+$ and $D^-$ glue together to bound a reducing sphere $P^*$ for the Heegaard splitting $Y^* = C^- \cup_h C^+$.  Cutting $Y^*$ open along $P^*$ and capping off the resulting 2--sphere boundary components with 3--balls has the same effect as compressing $C^+$ along $D^+$ to get $(C^+)'\sqcup(C^+)''$ and $C^-$ along $D^-$ to get $(C^-)'\sqcup(C^-)''$, where $(C^{\pm})'\sqcup(C^\pm)'' = \pd_- C^{\pm} \X I \cong S \X I$.  In this case, $S$ has two components, $S'$ and $S''$.  Then $Y^* = (S' \X I) \#_{P^*} (S'' \X I)$, and to recover $Y$ from $Y^*$ we glue the boundary components of $Y^*$.  There are two possibilities here:  In the first case, boundary components of $S' \X I$ are identified and boundary components of $S'' \X I$ are identified, in which case $Y = Y' \# Y''$.  In the second case, $S' \X \{1\}$ is glued to $S'' \X \{0\}$ and $S'' \X \{1\}$ is glued to $S' \X \{0\}$.   Here the reducing sphere $P^*$ for $Y^*$ is a non-separating sphere for $Y$, and cutting $Y$ open along $P^*$ and capping off with 3--balls yields a fibered manifold with fiber $S'$ (equivalently, $S''$).  Thus, $Y \cong (S^1 \X S^2) \# Y'$, completing the proof.
\end{proof}

As an example of the type of splitting arising in Lemma \ref{samecurve}, let $S = S^2$, let $\Delta^{\pm} \subset S^{\pm}$ be parallel copies of two disks in $S$, and let $h$ be the identity map on the torus, giving rise to a Heegaard double $(Y;S,\Delta,h)$.  By (1) of Lemma \ref{samecurve}, $Y \cong (S^1 \X S^2) \# Y'$, where $Y'$ fibers over the 2--sphere.   It follows that $Y' = S^1 \X S^2$, and so $Y \cong Y_2 = \#^2 (S^1 \X S^2)$.  We call this the \emph{standard Heegaard double} of the manifold $Y_2$; it will feature prominently in arguments below.

We say that the Heegaard splitting $Y^* = C^- \cup_h C^+$ is \emph{weakly reducible} if there exist essential curves $c \in \pd_+ C^+$ bounding a compressing disk for $C^+$ and $c' \in \pd_+ C^-$ bounding a compressing disk for $C^-$ with the property that $h(c)$ and $c'$ are isotopic to disjoint curves in $\Sigma^-$.  The next lemma is a classical result from the theory of Heegaard splittings.

\begin{lemma}
\label{weakred}
\cite{SchTho_Surgery-on-a-knot_09}
	If $(Y;S,\Delta,h)$ is a Heegaard double and $S \X \{1/2\}$ is compressible in $Y$,  then the Heegaard splitting $Y^* = C^- \cup_h C^+$ is weakly reducible.
\end{lemma}

We will use Lemma \ref{weakred} in our analysis of Heegaard doubles on the manifold $Y_2$; note that every incompressible surface in $Y_2$ is a 2--sphere; thus, if $(Y_2; S,\Delta,h)$ is a Heegaard double of $Y_2$ and $g(S) > 0$, then the Heegaard splitting $Y_2^* = C^- \cup_h C^+$ is weakly reducible.  In the next lemmas, we analyze what weak reducibility tells us about the curves $c^{\pm}$, since compressing disks for $C^-$ and $C^+$ are not necessarily unique.  The first lemma is Lemma~4.6 from~\cite{GomSchTho_Fibered-knots_10}.

\begin{lemma}
\label{uniquedisk}
\cite{GomSchTho_Fibered-knots_10}
	Suppose a curve $c \subset \pd_+ C^{\pm}$ bounds a compressing disk for $C^{\pm}$.  If $c^{\pm}$ is separating, then $c$ is isotopic to $c^{\pm}$ in $\pd_+ C^{\pm}$.  If $c^{\pm}$ is non-separating, then either $c$ is isotopic to $c^{\pm}$ in $\pd_+ C^{\pm}$, or $c$ is separating and cuts of a genus one subsurface of $\pd_+ C^{\pm}$ containing $c^{\pm}$.
\end{lemma}

The next lemma shows how weak reducibility can be leveraged in the present setting of Heegaard doubles.

\begin{lemma}
\label{possible}
	Let $(Y;S,\Delta,h)$ be a Heegaard double, and suppose the Heegaard splitting $Y^* = C^- \cup_h C^+$ is weakly reducible.  Then one of the following holds:
	\begin{enumerate}
		\item $Y$ is a fibered 3--manifold with fiber $S$,
		\item $Y = Y' \# Y''$, where $Y'$ is $S^1 \times S^2$ or a lens space and $Y''$ is a fibered 3--manifold with fiber a component of $S$,
		\item $Y = Y' \# Y''$, where $Y'$ and $Y''$ are fibered 3--manifolds with fibers given by the two components of $S$, or
		\item $h(c^+)$ and $c^-$ are non-isotopic and can be isotoped to be disjoint in $\Sigma^-$.
	\end{enumerate}
\end{lemma}

\begin{proof}
	Suppose that $c,c'$ are curves bounding compressing disks in $C^+$ and $C^-$, respectively, and $h(c) \cap c' = \emp$.  There are several cases to consider.  First, suppose that $c^\pm$ is separating in $\pd_+ C^{\pm}$.  By Lemma \ref{uniquedisk}, it follows that $c$ is isotopic to $c^+$ in $\Sigma^+$ and $c'$ is isotopic to $c^-$ in $\Sigma^-$.  If $h(c)$ is isotopic to $c'$ in $\Sigma^-$, then by Lemma \ref{samecurve}, conclusion (2) or (3) holds.  Otherwise, conclusion~(4) holds.

On the other hand, suppose that $c^{\pm}$ is non-separating in $\Sigma^{\pm}$.  For the first sub-case, suppose that $c$ is isotopic to $c^+$ in $\Sigma^+$ and $c'$ is isotopic to $c^-$ in $\Sigma^-$.  If $h(c)$ is also isotopic to $c'$, then by Lemma \ref{samecurve}, conclusion (2) is true.  Otherwise, conclusion (4) holds.  For the second sub-case, suppose without loss of generality that $c$ is not isotopic to $c^+$ in $\Sigma^+$, so that by Lemma~\ref{uniquedisk}, $c$ cuts off a genus one subsurface $T \subset \Sigma^+$ containing $c^+$.  Isotope $c^-$ in $\Sigma^-$ so that it intersects $\partial h(T)$ minimally.  If $c^- \subset h(T)$, then $\pd h(T) = h(c)$ bounds a disk in $C^-$.  In this case, the Heegaard splitting $Y^* = C^- \cup_h C^+$ is reducible, and the reducing sphere given by $c$ and $h(c)$ cuts off a genus one summand from the Heegaard splitting of $Y^*$, and thus $Y^* = Y' \# (S \X I)$ as in the proof of Lemma \ref{samecurve}, where $Y'$ is either $S^3$, $S^1 \X S^2$, or a lens space.  It follows that $Y$ is one of the 3--manifolds described in~(1) or~(2).

For the final sub-case, suppose that $c^-$ is not contained entirely within the subsurface $h(T)$.  If $c'$ is isotopic to $c^-$ in $\Sigma^-$, then the assumption $h(c) \cap c' = \emp$ implies $h(T) \cap c^- = \emp$, and conclusion (4) is satisfied.  Otherwise, by Lemma \ref{uniquedisk}, $c'$ cuts of a genus one subsurface $T' \subset \Sigma^-$ containing $c^-$.  Isotope $c^-$ and $T'$ in $\Sigma^-$ so that they meet $\partial h(T)$  minimally.  Since $c^-$ is not contained $h(T)$, the assumption $h(c) \cap c' = \emp$ implies $h(T) \cap T'  =\emp$.  Thus, after isotopy, $h(c^+) \cap c^- = \emp$, and conclusion (4) holds once again.
\end{proof}


Lemma \ref{possible} has an important consequence, which we record as the following lemma.

\begin{lemma}
\label{y2class}
	Every Heegaard double of $Y_2$ is either the standard Heegaard double, or $h(c^+)$ and $c^-$ are non-isotopic and can be isotoped to be disjoint in $\Sigma^-$
\end{lemma}

\begin{proof}
Let $(Y_2;S,\Delta,h)$ be a Heegaard double.  If $S$ is a 2--sphere, then $\Sigma^{\pm}$ is a torus and the only possible gluing map $h$ yielding $Y_2$ is the identity map, so this is the standard Heegaard double.  Otherwise, $g(S) > 0$ (in either the connected or disconnected case).  By Lemma~\ref{weakred}, the Heegaard splitting $Y_2^* = C^- \cup_h C^+$ is weakly reducible, and by Lemma~\ref{possible}, it must be true that conclusion (4) is satisfied.
\end{proof}

We now undertake a deeper analysis of what can happen in the case that $h(c^+)$ and $c^-$ non-isotopic and can be isotoped to be disjoint in $\Sigma^-$.  In this case, we can simplify the Heegaard double in a process called \emph{untelescoping}.  In order to define untelescoping, we require several new definitions.

Suppose $Y$ is a compact 3--manifold, and let $Y'$ be the result of attaching a 1--handle $H$ along a pair of disks $\Delta \subset \pd Y$.  We call the newly constructed boundary surface of $Y'$ the surface \emph{induced by the 1--handle attachment}.  On the other hand, let $c'$ be an essential curve in a boundary component $S$ of $Y$, and let $Y''$ be the result of attaching a 2--handle $H'$ along $c'$.  We call the newly constructed boundary surface $S''$ the surface \emph{induced by the 2--handle attachment}.  Note $S''$ has one component if $c$ is non-separating and two components if $c$ is separating.  In either case, $H' \cap S''$ is two embedded disks, which we call \emph{scars}.

Attaching 1--handles and 2--handles are dual processes, which we make rigorous in the following standard lemma.  The proof is left to the reader.

\begin{lemma}\label{handleswitch}
Let $S$ be a surface containing an essential curve $c$, let $Y$ be the result of attaching a 2--handle to $S \X I$ along $c \subset S \X \{0\}$, and let $S'$ be the surface induced by the 2--handle attachment containing scars $\Delta$.  Let $Y'$ be the 3--manifold obtained by attaching a 1--handle to $S' \X I$ along the pair of disks $\Delta \X \{1\} \subset S' \X \{1\}$.  Then there is a diffeomorphism $f\colon Y \rightarrow Y'$ such that $f|_{\pd Y}$ is the identity map.  In this case, the surface in $\pd Y'$ induced by the 1--handle attachment is $S \X \{1\}$ and the boundary of a co-core of the 1--handle is the curve $c$.
\end{lemma}

In light of Lemma \ref{handleswitch}, we give the process of reinterpreting a 2--handle attachment as a 1--handle attachment a name:  \emph{Pushing a 2--handle through the product $S \X I$}.

A Heegaard double $(Y;S,\Delta,h)$ decomposes $Y$ as the union of $S \X I$, $H^-$, and $H^+$, where the general structure is set up to suggest attaching $H^{\pm}$ as a 1--handle to $S^{\pm}$ along $\Delta^{\pm}$, followed by gluing the resulting surfaces via the homeomorphism $h$.  Under the assumption $h(c^+) \cap c^- = \emp$, we can rearrange the order of these gluings to get a new Heegaard double, said to be related to the original by \emph{untelescoping}.  We describe this process in the next proposition.  It may aid the reader's intuition to examine the schematic of the case in which $S_1$ is connected, shown in Figure~\ref{fig:untele}, before (or after) reading the proof of Proposition~\ref{untel}.

\begin{proposition}\label{untel}
	Suppose that $(Y;S_1,\Delta_1,h_1)$ is a Heegaard double such that $h_1(c_1^+)$ and $c_1^-$ are non-isotopic and can be isotoped to be disjoint in $\Sigma_1^-$.  Then there is another Heegaard double $(Y;S_2,\Delta_2,h_2)$ such that $\chi(S_2) > \chi(S_1)$.

In addition, if the resulting surface $S_2$ is connected, then
\begin{enumerate}
\item[(a)] $S_1$ is connected,
\item[(b)] $\Sigma^+_2 = S_1^+ = (\Sigma_1^+ \setminus c_1^+) \cup \Delta_1^+$, and
\item[(c)] $h_2|_{\Sigma_2^+ \setminus \Delta_1^+} = h_1|_{\Sigma_1^+ \setminus c_1^+}$.
\end{enumerate}
\end{proposition}

\begin{proof}

First, isotope $c_1^-$ in $\Sigma_1^-$ so that the resulting curve, call it $c_2^-$, satisfies $h_1(c_1^+) \cap c_2^- = \emp$.  Instead of attaching the 1--handle $H_1^+$ to $\Delta_1^+$ and gluing the resulting boundary $\Sigma_1^+$ to $\Sigma_1^-$ using $h_1$, we attach $H_1^+$ to $\Sigma_1^-$ as a 2--handle, denoted $H_*$, along the curve $h_1(c_1^+)$ in $\Sigma_1^-$.  Let $\Delta'_1$ be the scars of the 2--handle attachment.  Since $\Sigma_1^+ \setminus c_1^+ = S_1^+ \setminus \Delta_1^+$, this induces a new gluing map $h_*$ taking $S_1^+$ to the closed surface $\Sigma_*=(\Sigma_1^- \setminus h(c_1^+)) \cup \Delta'_1$, where $h_*(\Delta_1^+) = \Delta_1'$ and $h_*|_{S_1^+ \setminus \Delta_1^+} = h_1|_{\Sigma_1^+ \setminus c_1^+}$.

There are two cases to consider.  Suppose first that $S_1$ is connected, so that $c_1^{\pm}$ is non-separating in $\Sigma_1^{\pm}$.  Recall that $c_1^-$ is isotopic to $c_2^-$ in $\Sigma_1^-$, where $h_1(c_1^+) \cap c_2^- = \emp$.  This isotopy induces an isotopy from the co-core $D_1^-$ of $H_1^-$ bounded by $c_1^-$ to a disk $D_2^-$ bounded by $c_2^-$, such that compressing $(S_1 \X I) \cup H_1^-$ along $D_2^-$ yields a 3--manifold $S_* \X I$ diffeomorphic to $S_1 \X I$, and such that $S_*^+ = S_* \X \{1\}$ coincides with $S_1^+$.  Let $\Delta_2^-$ be the pair of disks in $S_*^- = S_* \X \{0\}$ such that attaching a 1--handle $H_2^-$ to $S_* \X I$ along $\Delta_2^-$ yields the same submanifold of $Y$ as attaching $H_1^-$ to $S_1 \X I$ along $\Delta_1^-$.  Here the disk $D_2^-$ is the co-core of $H_2^-$.

Next, observe that the attaching curve $h_1(c_1^+)$ for $H_*$ is contained in $\Sigma_1^- \setminus c_2^- = S_*^- \setminus \Delta_2^-$.  Thus, $H_*$ is attached to a curve in $S_*^-$, and by Lemma~\ref{handleswitch} we can push the 2--handle $H_*$ through the product $S_* \X I$.  In other words, $H_* \cup (S_* \X I)$ can be replaced with $(S_2 \X I) \cup H_2^+$, where $S_2^- = S_2 \X \{0\}$ is the surface
\[ S_2^- = \Sigma_1^- \setminus (h_1(c_1^+) \cup c_2^-) \cup (\Delta_1' \cup \Delta_2^-).\]

In addition, $H^+_2$ is a 1--handle attached to $S_2^+ = S_2 \X \{1\}$ along the disks $\Delta_2^+ = \Delta_1' \X \{1\}$. Note that the boundary component of $(S_2 \X I) \cup H_2^+$ induced by the 1--handle attachment is the surface $S_1^+$, and the other boundary component is $S_2^-$.  Let $\Sigma_2^-$ denote the surface induced by attaching $H^-_2$ to $S^-_2$ along $\Delta^-_2$; that is, $\Sigma_2^-  = (\Sigma_1^- \setminus h_1(c_1^+)) \cup \Delta_1'$, which is the surface $\Sigma_*$ defined at the beginning of the proof.  Let $\Sigma_2^+=S_1^+$, and let $h_2 = h_1'$, so $h_2$ takes $\Sigma_2^+$ to $\Sigma_2^-$.  It follows that $(Y;S_2,\Delta_2,h_2)$ is a Heegaard double, and since $S_2^-$ is obtained by attaching a 2--handle to $S_*^-$, we have $\chi(S_2) > \chi(S_*) = \chi(S_1)$.  Note that in this case, conditions (a), (b), and (c) above are satisfied (whether $S_2$ is connected or not).  If $S_2$ is disconnected, the assumption that $h_1(c_1^+)$ and $c_1^-$ are non-isotopic guarantees that $S_2$ does not have a 2--sphere component.

In the second, more complicated case, suppose that $S_1$ is not connected, so that $c_1^{\pm}$ is separating in $\Sigma_1^{\pm}$.  As above, we isotope $c_1^-$ onto $c_2^-$ disjoint from $h_1(c_1^+)$, inducing isotopies of disk $D_1^-$ to disk $D_2^-$, let $S_*^-$ be obtained by compressing $\Sigma_1^-$ along $D_2^-$, and let $H_2^-$ be the corresponding 1--handle such that attaching $H_2^-$ to disks $\Delta_2^-$ in $S_*^-$ yields $\Sigma_1^-$.  Then the 2--handle $H_*$ is attached to $h_1(c_1^+) \subset \Sigma_1^- \setminus c_2^- = S_*^- \setminus \Delta_2^-$.

Let $S_*'$ and $S_*''$ denote the components of $S_*$, with $(S_*')^{\pm},(S_*'')^{\pm} \subset S_*^{\pm}$, chosen so that $h_1(c_1^+) \subset (S_*')^-$.  As above, we push $H_*$ through the product $S_*' \X I$.  By Lemma~\ref{handleswitch}, we can replace $H_* \cup (S_*' \X I)$ with $(S_2 \X I) \cup H_2^+$, where $S_2$ is given by
$$S_2 = \Sigma_1\setminus(h_1(c_1^+)\cup c_2^-)\cup(\Delta_1'\cup\Delta_2^-),$$
the disks $\Delta_2^+$ are given by $\Delta_2^+ = \Delta_1' \X \{1\} \subset S_2 \X \{1\}$, and $H_2^+$ is a 1--handle attached to $S_2 \X \{1\}$ along $\Delta_2^+$.  The boundary component of $(S_2 \X I) \cup H_2^+$ induced by the 1--handle attachment is the surface $(S'_*)^+$, and the other boundary components are $S_2^- = (S_*')^- \setminus h_1(c_1^+) \cup \Delta_1'$.  Since $h_1(c_1^+)$ is not isotopic to $c_1^-$ in $\Sigma_1^-$ and separates $\Sigma_1^-$, it follows that $h_1(c_1^+)$ is separating in $(S_*')^-$, cutting $(S_*')^-$ into two components of positive genus.  This implies that $S_2$ is disconnected, with components $S_2'$ and $S_2''$.  Since each surface $(S_*)^-$ and $(S_*'')^-$ contains one attaching disk in $\Delta_2^-$ for $H_2^-$, we have that either $(S_2')^-$ or $(S_2'')^-$ contains an attaching disk in $\Delta_2^-$, and so we choose $S_2'$ so that $(S_2')^-$ contains this disk.

\begin{figure}[h!]
	\centering
	\includegraphics[width=\textwidth]{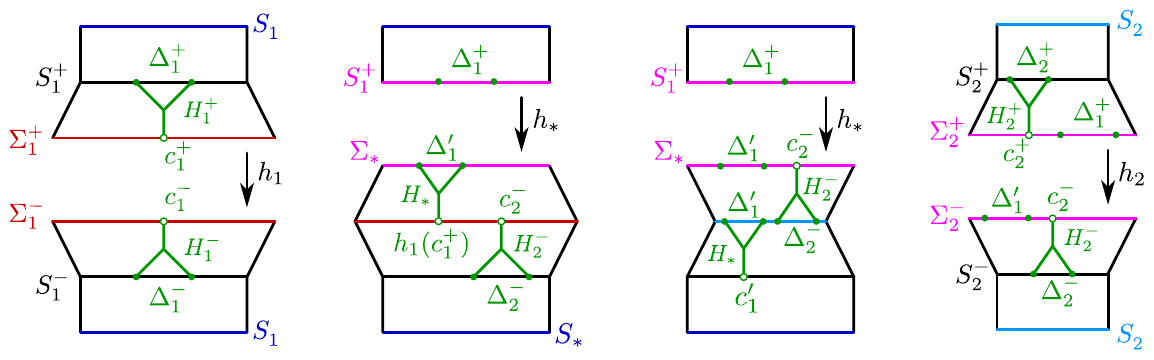}
	\caption{A sequence of schematics describing the process of untelescoping a Heegaard double in the case that $S_1$ is connected.}
	\label{fig:untele}
\end{figure}

Note that by construction $S_1^+ = S_*^+$, and thus $S_1^+$ is disconnected with components $(S_*')^+$ and $(S_*'')^+$. In this case, the gluing map $h_*$ described at the beginning of the proof takes the disconnected surface $S_1^+ = (S_*')^+ \cup (S_*'')^+$ to $(\Sigma_1^- \setminus h_1(c_1^+)) \cup \Delta_1'$, and we may separate $h_*$ into two maps $h'_*$ and $h''_*$ on $(S_*')^+$ and $(S_*'')^+$.  Observe that the image of $h_*$, the two components of $(\Sigma_1^- \setminus h_1(c_1^+)) \cup \Delta_1'$, are obtained by attaching $H_2^-$ to $(S_*'')^- \cup (S_2')^- \cup (S_2'')^-$ along $\Delta_2^-$, where the disks of $\Delta_2^-$ are contained in $(S_*'')^-$ and $(S_2')^-$.  Hence, the image of $h_*$ consists of $(S_2'')^-$ and $\Sigma_2^-$, where $\Sigma_2^-$ is induced by attaching $H_2^-$ to $(S_*'')^- \cup (S_2')^-$.  It follows that $g(\Sigma_2^-) = g(S_*'') + g(S_2') > g(S_*'')$, forcing $h'_*$ to map $(S_*')^+$ onto $\Sigma_2^-$ and $h''_*$ to map $(S_*'')^+$ to $(S_2'')^-$.  Since the result of attaching $S_*'' \X I$ to $S_2 \X I$ by gluing $(S_*'')^+$ to $(S_2'')^-$ is homeomorphic to $S_2 \X I$, we have a new decomposition of $Y$ obtained by attaching $H_2^-$ to $S_2 \X \{0\}$, attaching $H_1^+$ to $S_2 \X \{1\}$, and gluing the resulting boundary components $\Sigma_2^+ = (S_*^+)'$ and $\Sigma_2^-$.  This can be represented by a Heegaard double $(Y;S_2,\Delta_2,h_2)$, and by construction, $\chi(S_2) > \chi(S_1)$.  To complete the proof, we note that in this case, $S_2$ is never connected, and so the additional hypotheses are not satisfied.
\end{proof}
	
As stated above, we call the process of reducing the Heegaard double $(Y;S_1,\Delta_1,h_1)$ to $(Y;S_2,\Delta_2,h_2)$ \emph{untelescoping}.  Note that the situation for Heegaard doubles is somewhat different than for classical Heegaard splittings, since untelescoping a Heegaard double produces another Heegaard double.  Returning to the manifold $Y_2 = \#^2(S^1 \X S^2)$, we have the following lemma.

\begin{lemma}\label{totalreduction}
Any Heegaard double $(Y_2;S_1,\Delta_1,h_1)$ of $Y_2$ can be repeatedly untelescoped until it becomes isotopic to the standard Heegaard double of $Y_2$.  In addition, the curves $c_1^{\pm}$ are non-separating in $\Sigma_1^{\pm}$.
\end{lemma}

\begin{proof}
If $(Y_2;S_1,\Delta_1,h_1)$ is not the standard Heegaard double, then by Lemma \ref{y2class}, it can be untelescoped, increasing $\chi(S_1)$.  After finitely many untelescoping operations, Lemma \ref{y2class} implies that the result $(Y_2;S_n,\Delta_n ,h_n)$ is the standard Heegaard double.  In the standard Heegaard double, $S_n$ is connected, and thus by repeated applications of Lemma \ref{untel}, the surface $S_1$ is connected as well.   We conclude $c_1^\pm$ is non-separating in $\Sigma_1^\pm$.
\end{proof}

Now we turn to the specific case of a 2R-link $L = Q \cup J$, where $Q$ is fibered.  Recall the notation and language set up in Section \ref{sec:intro}.  We let $F$ denote a fiber of $Q$ in $S^3$, $Y_Q$ the result of 0--surgery on $Q$ with closed fiber $\wh F$ and closed monodromy $\wh \varphi$.  Recall also that two links $L$ and $L'$ are stably equivalent if there are unlinks $U$ and $U'$ such that $L \sqcup U$ is handleslide-equivalent to $L' \sqcup U'$.

\begin{lemma}
\label{surgerydouble}
	Suppose $L = Q \cup J$ is a 2R-link, where $Q$ is non-trivial and fibered knot with fiber $F$.  In $Y_Q$, the framed knot $J$ can be isotoped to lie in a closed fiber $\wh F$ with the surface framing, and $J \subset \wh F$ naturally induces a Heegaard double $(Y_2;S,\Delta,h)$, where 
	\begin{enumerate}
		\item $\Sigma^+$ and $\Sigma^-$ are copies of $\wh F$;
		\item $S$ is the result of gluing disks $\Delta^*$ to the boundary components of $\wh F \setminus J$;
		\item $\Delta = \Delta^+ \cup \Delta^-$, where $\Delta^+ = \Delta^*\times\{1\} \subset S^+$ and $\Delta^- = \Delta^*\times\{0\}\subset S^-$; and 
		\item $h\colon \Sigma^+\to\Sigma^-$ is the closed monodromy $\wh \varphi$.  
	\end{enumerate} 
\end{lemma}

\begin{proof}
Since $J$ is disjoint from $Q$, we may view $J$ as a knot in $Y_Q$, which we will also denote $J$ in an abuse of notation.  By Corollary~4.3 of~\cite{SchTho_Surgery-on-a-knot_09}, the knot $J$ is isotopic in $Y_Q$ into a closed fiber $\wh F$ for $Y_Q$, where the surface slope of $J$ in $\wh F$ is the 0--framing. We now describe the process dubbed the \emph{surgery principle} in Lemma 4.1 of~\cite{GomSchTho_Fibered-knots_10}.
	
Let $\wh F$ and $\wh F^*$ denote two copies of the closed fiber in $Y_Q$ such that $Y_Q$ is the union of $\wh F\times I$ and $\wh F^*\times I$, where $\wh F^*\times\{0\}$ is identified with $\wh F\times \{1\}$ using the diffeomorphism $\wh\varphi\colon \wh F^*\times\{1\}\to\wh F\times\{0\}$, and $\wh F\times\{1\}$ is identified with $\wh F^*\times\{0\}$ using the identity.  Let $\Sigma^- = \wh F\times\{0\}$, and let $\Sigma^+ = \wh F^*\times\{1\}$.

Suppose now that $J$ has been isotoped into $\wh F \X \{1\}$, and let $J^*$ be a copy of $J$ in $\wh F^* \X \{0\}$, which is identified with $J$ in $Y_Q$.  We obtain $Y_2$, the result of 0--surgery on $J$, by attaching a 2--handle $H$ to $\wh F \X \{1\}$ along a copy of $J$ and another 2--handle $H^*$ to $J^* \subset \wh F^* \X \{0\}$.  Then, letting $S$ and $S^*$ be the surfaces induced by the 2--handle attachments, with scars $\Delta'$ and $\Delta^*$, respectively, we glue $S$ to $S^*$ with the identity map and glue $\wh F^* \X \{1\}$ to $\wh F \X \{0\}$ with $\wh \varphi$.  Pushing the 2--handle $H$ across the product $\wh F \X I$ and pushing $H^*$ across the product $\wh F^* \X I$ yields the desired Heegaard double, $(Y_2;S,\Delta,h)$, where $h = \wh \varphi$ and $\Delta$ consists of $\Delta' \X \{0\}$ and $\Delta^* \X \{1\}$.
\end{proof}

We remark that the surfaces $\Sigma^+$ and $\Sigma^-$ of the induced Heegaard double may be viewed as parallel copies of the fiber $\wh F$ in $Y_Q$, in which the compressing curves $c^+$ and $c^-$ are parallel copies of $J \subset \wh F$.  The natural next step is to untelescope this induced double, and with careful bookkeeping, we can prove the main theorem from this section.

\begin{reptheorem}{thmx:CGequiv}
	Suppose $L = Q \cup J$ is a 2R-link, where $Q$ is non-trivial and fibered.  Then $L$ is stably equivalent $Q \cup L^+$, where $L^+$ is a CG-derivative of $Q$.
\end{reptheorem}

\begin{proof}

Let $(Y_2;S_1,\Delta_1,h_1)$ be the Heegaard double described in Lemma~\ref{surgerydouble}.  We will use the same notation as in that lemma, so that $(Y_2; S_1,\Delta_1,h_1)$ is induced by isotoping $J$ to lie in a closed fiber $\wh F$ of $Y_Q$.  In addition, $c_1^+$ and $c_1^-$ are parallel copies of $J$ contained in the surfaces $\Sigma_1^+$ and $\Sigma_1^-$, which can be viewed as parallel copies of the fiber $\wh F$ in $Y_Q$.  By Lemma~\ref{totalreduction}, this Heegaard double can be untelescoped until it becomes the standard Heegaard double of $Y_2$.  Each surface $S_n$ is connected, which means that each untelescoping operation reduces the genus of the surface $S_1$ by one and this process requires a total of $g-1$ untelescoping operations, where $g = g(\wh F)$.  Note that $(Y_2;S_1,\Delta_1,h_1)$ is not standard since $g \geq 2$.  We will let $(Y_2;S_n,\Delta_n,h_n)$ be the result of untelescoping $(Y_2;S_1,\Delta_1,h_1)$ a total of $n-1$ times, for $1 \leq n \leq g$, so that $(Y_2;S_g,\Delta_g,h_g)$ is the standard Heegaard double of $Y_2$.

Consider the surfaces $\Sigma_1^+$ and $\Sigma_1^-$, which (as noted above) may be considered to be parallel copies of $\wh F$ in $Y_Q = \wh F\times[0,1]/\sim$, where $(x,1)\sim(h_1(x),0)$ for $h_1\colon \Sigma^+\to\Sigma^-$. Let $\pi\colon \wh F \X I \to \Sigma_1^+$ be the projection map induced by the product structure.  The map $\pi$ is a mechanism we use to keep track of a fixed copy, $\Sigma_1^+$, of $\wh F$, as opposed to working with two copies, $\Sigma^-$ and $\Sigma^+$, in parallel.  For the remainder of the proof, we will interpret $\wh \varphi$ as a map from $\Sigma_1^+$ to itself, so that $h_1\colon\Sigma_1^+ \rightarrow \Sigma_1^-$ satisfies $\pi \circ h_1 = \wh \varphi$.

By Proposition~\ref{untel}, $\Sigma_2^+ = \Sigma_1^+ \setminus c_1^+ \cup \Delta_1^+$, and thus $c_2^+$ may be chosen so that $c_2^+ \subset \Sigma_1^+ \setminus c_1^+$.  By induction, we have $\Sigma_n^+ = \Sigma_1^+ \setminus (c_1^+ \cup \dots \cup c_{n-1}^+) \cup (\Delta_1^+ \cup \dots \cup \Delta_{n-1}^+)$, so that $c_n^+$ may be chosen so that $c_n^+ \subset \Sigma_1^+ \setminus (c_1^+ \cup \dots \cup c_{n-1}^+)$.  For a set of choices $c_1^+,\dots,c_g^+$, we let $L^+$ be the $g$--component link given by $L^+ = c_1^+ \cup \dots \cup c_g^+$, noting that $L^+$ is a $g$--component link cutting the fiber surface $\Sigma_1^+$ into a planar surface.  Give $L^+$ the surface framing.

\textbf{Claim 1:} For some choice of the curves $c_1^+,\dots,c_g^+$, the link $L^+$ is a Casson-Gordon derivative for $Q$.

We remark that the claim is, in fact, true for all choices of $c_i^+$; however, we need only this weaker statement for the proof of the theorem.


\emph{Proof of Claim 1:} Since $c_1^+,\dots,c_n^+$ are pairwise disjoint in $\Sigma_1^+$, repeated applications of Proposition~\ref{untel} yield that $h_n(c_n^+) = h_1(c_n^+)$.  Following the proof of Proposition~\ref{untel}, there exist disks $\Delta_n'$ contained in $\pd H_n^+$ such that $\Sigma_{n+1}^- = \Sigma_n^- \setminus h_1(c_n^+) \cup \Delta_n' = \Sigma_1^- \setminus (h_1(c_1^+) \cup \dots \cup h_1(c_n^+)) \cup(\Delta_1' \cup \dots \cup \Delta_n')$.  Observe that $c_{n+1}^-$ is necessarily contained in $\Sigma_{n+1}^-$; we will assume that $c_{n+1}^-$ has been chosen to be a curve disjoint from $h_1(c_1^+),\dots,h_1(c_n^+)$ and isotopic to $c_n^-$ in $\Sigma_n^-$.

In order to prove the claim, we establish the following statement:  For every $n$ such that $1 \leq n \leq g$, the two sets of pairwise disjoint curves
\[ L_n^+ = \{c_1^+,\dots,c_n^+\} \qquad \text{ and } \qquad  L_n^- = \{\pi(c_{n}^-),\wh \varphi(c_1^+),\dots,\wh\varphi(c_{n-1}^+)\}\]
define the same compression body as curves in $\Sigma_1^+$.

We induct on $n$.  The case $n = 1$ follows from the fact that $\pi(c_1^-) = c_1^+$.  Suppose $n=2$.  By the proof of Proposition~\ref{untel}, $c_2^-$ is isotopic to $c_1^-$ and disjoint from $h_1(c_1^+)$ in $\Sigma_1^-$, and we obtain $\Sigma_2^-$ by cutting along $h_1(c_1^+)$ and capping off with disks $\Delta_1'$.  There is a curve $c_1' \subset \Sigma_1^-$ such that the disjoint pair $(c_2^-,h_1(c_1^+))$ is isotopic to the disjoint pair $(c_1^-,c_1')$ in $\Sigma_1^-$.  Since $c_1' \cap c_1^- = \emp$, there is a annulus $A_1' \subset S_1 \X I$ such that $\pd(S_1 \X I) = c_1' \cup \pi(c_1')$; thus, we can let $c_2^+ = \pi(c_1')$.  Since $(c_2^-,h_1(c_1^+))$ is isotopic to $(c_1^-,c_1')$ in $\Sigma_1^-$, we have $(\pi(c_2^-),\pi(h_1(c_1^+)))$ is isotopic to $(\pi(c_1^-),\pi(c_1'))$ in $\Sigma_1^+$.  The former pair is $L^-_2=(\pi(c_2^-),\wh\varphi(c_1^+))$, while the latter is $L_2^+=(c_1^+,c_2^+)$.  Since $L_2^-$ is isotopic to $L_2^+$ in $\Sigma_1^+$, they define the same compression body.

Now suppose by way of induction that $L_n^+$ and $L_n^-$ define the same compression body.  This implies that the curves in $L_n^-$ can be changed into the curves in $L_n^+$ by a sequence of isotopies and handleslides in $\Sigma_1^+$.  As above, the curve $c_{n+1}^-$ is isotopic to $c_n^-$ and disjoint from $h_n(c_n^+) = h_1(c_n^+)$ in $\Sigma_n^-$, where $\Sigma_n^- = \Sigma_1^- \setminus (h_1(c_1^+) \cup \dots \cup h_1(c_{n-1}^+)) \cup (\Delta_1' \cup \dots \cup \Delta_{n-1}')$.  It follows that $c_{n+1}^-$ is isotopic to $c_n^-$ in $\Sigma_1$ modulo handleslides over the curves of $\{h_1(c_1^+),\ldots,h_1(c_{n-1}^+)\}$.  Since $c_{n+1}^- \cap h_1(c_n^+) = \emp$, there is a curve $c_n' \subset \Sigma_1^- \setminus (h_1(c_1^+) \cup \dots \cup h_1(c_{n-1}^+))$ such that the disjoint pair $(c_{n+1}^-,h_1(c_n^+))$ is isotopic to the disjoint pair $(c_n^-,c_n')$ in $\Sigma_1^+$, modulo handleslides over the curves of $\{h_1(c_1^+),\ldots,h_1(c_{n-1}^+)\}$.  By applying the projection $\pi$,  we have that the pair $(\pi(c_{n+1}^-),\wh\varphi(c_n^+))$ is isotopic to the pair $(\pi(c_n^-),\pi(c_n'))$ in $\Sigma_1^+$, modulo handleslides over the curves of $\{\wh\varphi(c_1^+),\ldots,\wh\varphi(c_{n-1}^+)\}$.

Now, observe that
\[(\pi(c_n^-) \cup \pi(c_n')) \cup (\wh\varphi(c_1^+) \cup \ldots \cup \wh\varphi(c_{n-1}^+)) = L_n^- \cup \pi(c_n').\]
There exists a curve $c_{n+1}^+ \subset \Sigma_1^+ \setminus L_n^+$ such that the sequence of isotopies and handleslides taking $L_n^-$ to $L_n^+$ give rise to a sequence of isotopies and handleslides taking $L_n^- \cup \pi(c_n')$ to $L_n^+ \cup c_{n+1}^+$.  Note that $S_n^{\pm} \subset \Sigma_n^{\pm} \setminus L_n^{\pm}$, and thus isotopies and handleslides over the curves in $L_n^-$ describe an isotopy from $h_1(c_n^+)$ to $c_{n+1}^+$ in the product $S_n \X I$, verifying that the curve $c_{n+1}^+$ obtained via this process is isotopic in $S_n^+ = \Sigma_{n+1}^+$ to a co-core of $H_{n+1}^+$.  We conclude that $L_{n+1}^+$ and $L_{n+1}^-$ define the same compression body, and by induction this holds for all $n$.

To complete the proof of the claim, we note that $(Y_2;S_g,\Delta_g,h_g)$ is the standard Heegaard double, which implies that $h_g(c_g^+) = h_1(c_g^+)$ is isotopic to $c_g^-$ in $\Sigma_g^-$; equivalently, $h_1(c_g^+)$ is isotopic to $c_g^-$ in $\Sigma_1^-$ modulo handleslides over $h_1(c_1^+),\dots,h_1(c_{g-1}^+)$.  It follows that the curves $\wh \varphi(L_g^+)$ define the same handlebody as $L_g^-$, which defines the same handlebody as $L_g^+$ by the above argument.  We conclude that $\wh \varphi$ extends over the handlebody determined by $L^+ = L_g^+$.

\textbf{Claim 2:} $L = Q \cup J$ is stably equivalent to $Q \cup L^+$.

\emph{Proof of Claim 2:} Let $U$ denote a split, $(g-1)$--component 0--framed unlink in $E(L)$.  We can isotope $U$ in $E(L)$ so that $U \subset F$ and $U$ bounds a collection of $g-1$ disjoint disks in $F$.  By a sequence of handleslides, one for each component of $U$, we may change $U$ to a $(g-1)$--component link in which each component is a parallel copy of $J$ in $F$.  Since $U \subset E(L)$, we may view $U$ as a link in $Y_2$, with each component of $U$ surface-framed and isotopic to $c_1^+$ in $\Sigma_1^+$, as $c_1^+$ is parallel to $J$ in $Y_Q$.  In other words, each component of $U$ is a parallel push-off in $\Sigma_1^+$ of the boundary $c_1^+$ of the co-core $D_1^+$ of the 1--handle $H_1^+$.  As we untelescope $g-1$ times, we isotope the 1--handle $H_1^+$, and for each iteration, we leave behind a component of $U$ as one of the pairwise disjoint curves $c_2^+,\dots,c_g^+$.

It follows that $U$ is isotopic to the link $c_2^+ \cup \dots \cup c_g^+$ in $Y_2$, which implies that $U$ is isotopic to $c_2^+ \cup \dots \cup c_g^+$ modulo handleslides over $J$ in $Y_Q$, and thus $J \cup U$ is handleslide equivalent to $L^+$ in $Y_Q$.  Finally, it follows that $Q \cup J \cup U$ is handleslide equivalent to $Q \cup L^+$ in $S^3$, completing the proof of the theorem.
\end{proof}

\section{Curves on the fiber of a generalized square knot}
\label{sec:square}

In the previous section, we showed in Theorem~\ref{thmx:CGequiv} that in order to understand the possible stable equivalence classes of a 2R-link $L = Q \cup J$ with $Q$ fibered, it suffices to understand Casson-Gordon derivatives for $Q$.  In this section, we build on the approach and techniques of Scharlemann~\cite{Sch_Proposed-Property_16} to develop the background we will need for the classification of Casson-Gordon derivatives for generalized square knots, which we give at the end of the section.

We begin by describing detailed pictures of the monodromies of torus knots and the closed monodromies of generalized square knots.  Next, we show how this closed monodromy generates the group of deck transformations for a branched covering of the capped off fiber surface of a generalized square knot over a 2--sphere.  By lifting distinct curves from the 2--sphere to the fiber surface, we give a list of CG-derivatives for each generalized square knot, and by invoking the Equivariant Loop Theorem, we show that this list is complete. As a consequence, we construct many R-links that are potential counterexamples to the Stable Generalized Property R Conjecture, as in Proposition~\ref{propx:nR}, which we prove in this section.

\subsection{Fibering generalized square knots}
\label{subsec:square}
\ 

Recall that the generalized square knot $Q_{p,q}$ is defined to be $Q_{p,q} = T_{p,q}\#T_{-p,q}$, where $T_{p,q}$ denotes the $(p,q)$--torus knot with $0<q<p$.  For the rest of this section, in order to ease notation we fix the parameters $p$ and $q$, letting $K^\pm = T_{\pm p,q}$ and $Q = Q_{p,q} = K^+\# K^-$. Let $F^\pm$ denote fixed minimal genus Seifert surfaces for $K^\pm$, and let $F = F^+\natural F^-$ denote the corresponding Seifert surface for $Q$, where $\natural$ denotes the natural boundary-connected summation of Seifert surfaces yielding a Seifert surface for $Q$.  It is well-known that~$g(F^\pm) = \frac{1}{2}(p-1)(q-1)$, so $g(F) = (p-1)(q-1)$.  As before, we use $E_{K^{\pm}}$ and $E_Q$ to represent the exteriors of $K^{\pm}$ and $Q$, respectively, and we let $Y_Q$ denote the result of 0--framed Dehn surgery on $Q$, with $\wh F$ the closed fiber in $Y_Q$.  In addition, we let $\varphi^{\pm}$ denote the monodromy for $E_{K^{\pm}}$, $\varphi$ the monodromy for $E_Q$, and $\wh \varphi$ the monodromy for $Y_Q$.

In this subsection, we give an explicit description of the surface bundle structures on $E_{K^\pm}$, $E_Q$, and $Y_Q$.  To begin, we construct the Seifert surface $F^+$ for the torus knot $K^+ = T_{p,q}$, where $K^+$ is contained in a Heegaard torus $T$ cutting $S^3$ into solid tori $V$ and $V'$.  Let $D_1,\dots,D_p$ be disjoint meridian disks for $V$, and let $D'_1,\dots,D_q'$ be disjoint meridian disks for $V'$, so that $\{D_i\}$ and $\{D_j'\}$ meet in $pq$ points $\{x_{i,j}\}$, with $x_{i,j} = D_i \cap D_j'$.  Replace each point of intersection $x_{i,j}$ with a band $B_{i,j}$ containing a negative quarter twist, so that the union $F^+ = \{D_i\} \cup \{D_j'\} \cup \{B_{i,j}\}$ is a Seifert surface for $K^+$.  See Figure~\ref{fig:torus_band}.

\begin{figure}[h!]
	\centering
	\includegraphics[width=.4\textwidth]{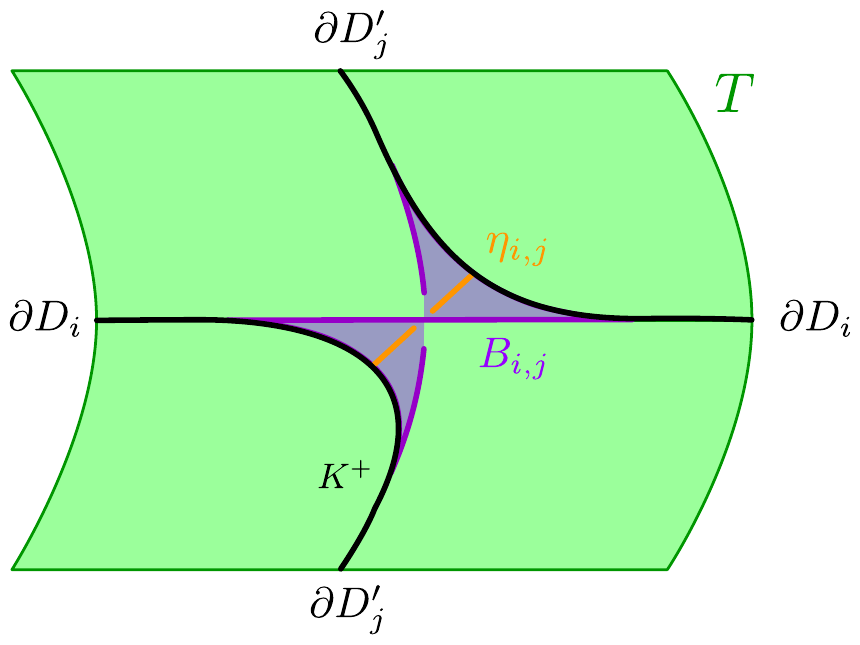}
	\caption{A local picture of the Heegaard torus $T$ near where the boundaries of the disks $D_i$ and $D_j'$ intersect, where the intersection point $x_{i,j}$ has been replaced with the band $B_{i,j}$.}
	\label{fig:torus_band}
\end{figure}

The monodromy $\varphi^+$ corresponding to the fibration of $E_{K^+}$ is well-understood: It can be visualized as a simultaneous cork-screwing of the disks $\{D_i\}$ and $\{D_j'\}$ within the solid tori $V$ and $V'$.  Specifically, $\varphi^+$ cyclically permutes both sets of disks, as well as the bands.  Thus, $\varphi^+$ has order $pq$, and we may assume that the disks are labeled so that $\varphi^+(D_i) = D_{i+1}$, $\varphi^+(D_j) = D_{j+1}$, and $\varphi^+(B_{i,j}) = B_{i+1,j+1}$, with indices $i$ and $j$ considered modulo $p$ and $q$, respectively.  See the left graphic of Figure~\ref{fig:TK_fib}, where we have represented the core of $V$ by the $z$--axis and the core of $V'$ by the unit circle in the plane $\{z=0\}$ in our illustration of the case of $(p,q)=(4,3)$.

In order to better understand the action of $\varphi^+$ on $F^+$, we build an alternative picture as in \cite{Sch_Proposed-Property_16}.  Let $\Gamma^+$ be a graph embedded in $F^+$, where $\Gamma^+$ has a vertex $v_i$ in the center of each disk $D_i$ and a vertex $v_j'$ in the center of each disk $D_j'$, for a total of $p+q$ vertices.  In addition, $\Gamma^+$ has $pq$ edges, labeled $e_{i,j}$, connecting $v_i$ to $v_j'$ and passing through the core of the band $B_{i,j}$.  As such, we may suppose without loss of generality that that $\varphi^+(\Gamma^+) = \Gamma^+$, where $\varphi^+(v_i) = v_{i+1}$, $\varphi^+(v_j') = v_{j+1}'$, and $\varphi^+(e_{i,j}) = e_{i+1,j+1}$, with indices considered modulo $p$ and $q$ as above.  See the left panel of Figure~\ref{fig:TK_fib} for the case of $(p,q)=(4,3)$.

Knowing $\varphi^+(\Gamma^+)$, we now consider the action of $\varphi^+$ on $F^+ \setminus \Gamma^+$.  Cutting $F^+$ along $\Gamma^+$ yields an annulus $A^+$, where one boundary component of $A^+$ is the knot $K^+$ and the other boundary component is a $2pq$--gon coming from $\Gamma^+$.  Each edge $e_{i,j}$ of $\Gamma^+$ gives rise to two edges $e_{i,j}^{\pm}$ in $\pd A^+$, labeled as in the center panel of Figure~\ref{fig:TK_fib}.  Moving clockwise around $\pd A^+$, we see that edges alternate between $+$ and $-$, the edge $e_{i,j}^+$ is adjacent to $e_{i,j+1}^-$, and the edge $e_{i,j}^-$ is adjacent to $e_{i+1,j}^+$.  Moreover, the monodromy $\varphi^+$ preserves the orientation of the edges, and thus $\varphi^+$ acts on the $2pq$--gon by a $2\pi / pq$ clockwise rotation.  As in \cite{Sch_Proposed-Property_16}, we assume that $\varphi^+$ also induces a $2\pi / pq$ rotation of the knot $K^+$.  With this setup, we see a departure from the usual convention that $\varphi^+|_{\pd E_{K^+}}$ is the identity, since the knot itself is rotated along with the fiber surface. We make this choice because is it compatible with the Seifert fibered structure on $E_{K^+}$ in that $\varphi^+$ preserves fibers (see Lemma~\ref{Seif}).   Furthermore, this assumption does not alter our eventual description of the closed monodromy $\wh \varphi$ for $Y_Q$.

\begin{figure}[h!]
	\centering
	\includegraphics[width=\textwidth]{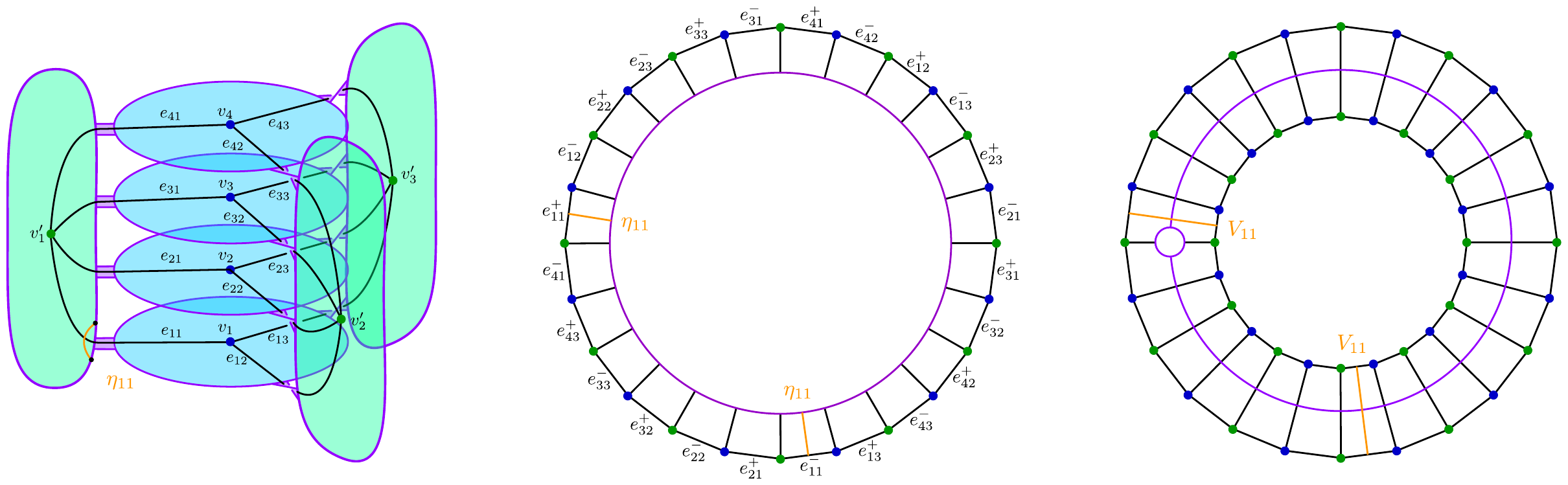}
	\caption{(Left) The surface fiber $F^+$ for the torus knot $K^+$, shown with spine graph $\Gamma^+$.  (Center) The annulus $A^+$ obtained by cutting open $F^+$ along $\Gamma^+$. (Right) The (punctured) annulus $A$ whose edge identifications yield the surface fiber $F$ for $Q$ . Shown in orange is the arc $\eta_{1,1}$ on $F^+$ and the curve $V_{1,1}$ on $F$.}
	\label{fig:TK_fib}
\end{figure}

Since the mapping class group of $A^+$ is $\Z$ and we understand $\varphi^+|_{\pd A^+}$, the map $\varphi^+$ is completely determined up to some number $k$ of Dehn twists about the core of the annulus $A^+$.  Consider the co-core arc $\eta_{i,j}$ of a band $B_{i,j}$.  The arc $\eta_{i,j}$ meets $\Gamma^+$ once, crossing the edge $e_{i,j}$, so that $\eta_{i,j} \cap A^+$ consists of two disjoint arcs, connecting $K$ to $e_{i,j}^+$ and $K$ to $e_{i,j}^-$.  The number of twists $k$ in $\varphi^+$ is equal to zero if and only if $\eta_{i,j} \cap \varphi^+(\eta_{i,j}) = \emp$, and we see that since $\varphi^+(B_{i,j}) = B_{i+1,j+1}$, the map $\varphi^+$ moves $\eta_{i,j}$ completely off of itself.  We conclude that $k = 0$, and the monodromy $\varphi^+$ is isotopic to a $2\pi/pq$ clockwise rotation of the annulus $A^+$. 

The last piece of information we need in order to completely understand $\varphi^+$ is the identification of $(+)$--edges and $(-)$--edges in $\pd A^+$ that recovers the Seifert surface $F^+$.  If we label the sides of the $2pq$--gon component of $\pd A^+$ in clockwise order from 0 to $2pq-1$, where the edge $e_{1,1}^+$ has label zero, we see that $(+)$--edges have even labels and $(-)$--edges have odd labels.  In addition, every edge $e_{1,j}^+$ is labeled $2ap$ for some integer $a$, and every edge $e_{i,1}^+$ is labeled $2bq$ for some integer $q$.  Since the edge $e_{1,1}^-$ is adjacent to both $e_{2,1}^+$ and $e_{1,q}^+$, its label $l$ is equal to $2ap + 1$ and $2bq - 1$.  Equivalently, we have that $ap + 1 = bq$, and thus the $(+)$--edge labeled 0 is identified to the $(-)$--edge labeled $2ap + 1$, where $ap \equiv -1 \pmod q$.  More generally, every $(+)$--edge labeled $l$ is identified to the $(-)$--edge labeled $l + 2ap + 1 \pmod{2pq}$, completing the picture.

\begin{remark}\label{rmk:sch_fiber}
	Upon first glance, the reader might notice that the picture described here is different than the picture described in~\cite{GomSchTho_Fibered-knots_10}~and~\cite{Sch_Proposed-Property_16}, where $(p,q) = (3,2)$.  However, these two descriptions can be seen to be identical after the following observation:  In the case that $q = 2$, we have that $p \equiv 1 \pmod 2$, and thus the $(+)$--edge labeled $l$ is identified with the $(-)$--edge labeled $l + 2p + 1$ in the $4p$--gon boundary component of $A^+$.  In addition, the $(+)$--edge labeled $l + 2p$ is identified with the $(-)$--edge $l + 4p + 1 \equiv l + 1 \pmod {4p}$.  Thus, the consecutive pair of $(\pm)$--edges labeled $l-1$ and $l$ are glued to the consecutive pair of $(\pm)$--edges labeled $l + 2p$ and $l+ 2p + 1$, and our description may be simplified.  In this case, the $4p$--gon boundary component may be viewed as a $2p$--gon in which opposite edges are identified.  Moreover, the monodromy remains a $2\pi/ pq = \pi/p$ clockwise rotation, and we see that the descriptions here and in \cite{Sch_Proposed-Property_16} are identical.  The distinction stems from the fact that when $q = 2$, the vertices $v_j' \in \Gamma^+$ have valence two, and the co-cores  $\eta_{i,1}$ and $\eta_{i,2}$ of the bands $B_{i,1}$ and $B_{i,2}$ are isotopic in $F^+$.
\end{remark}

We may now proceed to understand the monodromy of $Q = K^+ \# K^-$.  The monodromy of $K^-=T_{-p,q}$ can be described by reflecting the annulus $A^+$ through its boundary component coming from $K^+$ to get another annulus $A^-$, corresponding to a Seifert surface $F^-$ for $K^-$ containing an analogous graph $\Gamma^-$.  As such, this monodromy can be represented by a clockwise rotation of $A^-$, and it follows that the once-punctured surface fiber $F$ for $Q$ comes from gluing $A^-$ to $A^+$ along a portion of $K^{\pm}$ to obtain a punctured annulus $A$ with the given edge identifications, mimicking a similar step in \cite{Sch_Proposed-Property_16}.  The result is displayed in the right panel of Figure~\ref{fig:TK_fib}.  The knot $Q$ interferes with the periodicity of the monodromy -- rotation of $A$ moves the puncture -- so the rotation of $A$ must be followed by an isotopy taking $Q$ back to its starting position.  Once this is done, we have recovered the monodromy $\varphi$ corresponding to the surface bundle $E_Q$; see Figure~\ref{fig:mono_drag}.

\begin{figure}[h!]
	\centering
	\includegraphics[width=.7\textwidth]{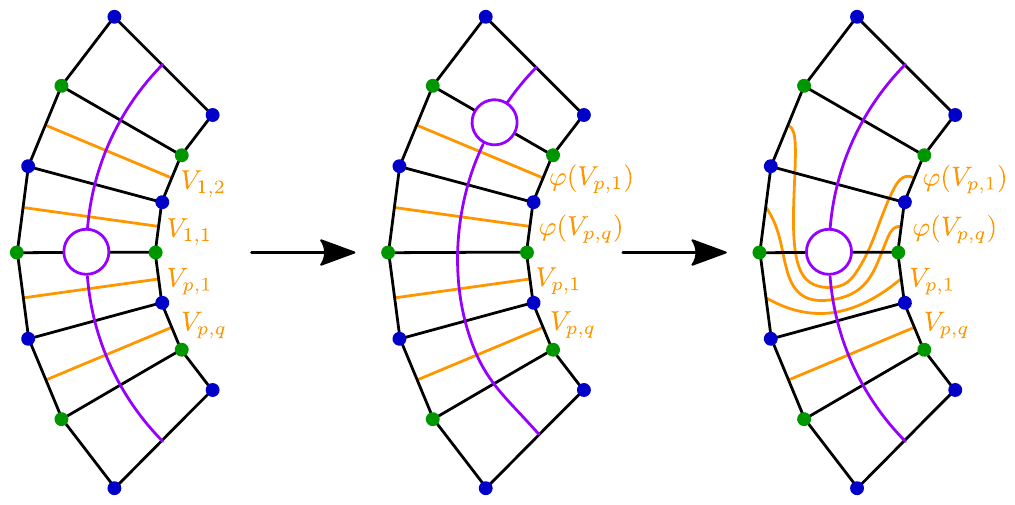}
	\caption{The local model of the monodromy $\varphi$ of $Q$ near the puncture of $A$, featuring the necessary action of dragging the puncture back to its initial position after the $1/pq$ clockwise rotation.}
	\label{fig:mono_drag}
\end{figure}

\subsection{The Seifert fibered structure of $Y_Q$}
\label{subsec:cover}
\ 

Consider $Y_Q$, the result of 0--surgery on $Q$ in $S^3$.  Using our work above, $Y_Q$ is a fibered 3--manifold with periodic monodromy $\wh\varphi$ of order $pq$ and (closed) surface fiber $\wh F$.  Moreover, $\wh F$ can be obtained by performing the above edge identifications on the annulus $\wh A = A^+ \cup A^-$, which has two $2pq$--gon boundary components, in which case $\wh\varphi$ is represented by an honest (clockwise) $2\pi /pq$ rotation of $\wh A$.  (Alternatively, $\wh A$ is obtained by filling in the puncture of $A$, which corresponds to the 0--framed Dehn surgery.)

\begin{lemma}\label{Seif}
The manifold $Y_Q$ is Seifert fibered with base space a 2--sphere $S$ with four exceptional fibers of orders $p$, $q$, $p$, and $q$.
\end{lemma}

\begin{proof}
Let $\Lambda_\infty$ denote the core of the annulus $\wh A$.  Since $\wh\varphi$ maps $\Lambda_{\infty}$ to itself, preserving orientation, it follows that there is a torus $W_{\infty} = \Lambda_{\infty} \X_{\wh \varphi} S^1 \subset Y_Q$.  Cutting $Y_Q$ open along $W_{\infty}$ yields the 3--manifolds, call them $Y^+$ and $Y^-$, fibering over $\wh F \setminus \Lambda_{\infty} = F^+ \cup F^-$.  As the restriction of $\wh \varphi$ to $F^{\pm}$ is $\varphi^{\pm}$, we have that $Y^{\pm}$ is homeomorphic to $E_{K^{\pm}}$.  It follows that $Y_Q$ can be obtained by gluing $E_{K^+}$ to $E_{K^-}$ along their respective boundary tori.

It is well-known that each of $E_{K^+}$ and $E_{K^-}$ is Seifert fibered over a disk with two exceptional fibers of orders $p$ and $q$.  Moreover, the monodromies $\varphi^{\pm}$ act on the Seifert fibers, which are the orbits of points in $F^{\pm}$.  Since these monodromies agree on $\pd F^{\pm}$, it follows that $E_{K^+}$ is glued to $E_{K^-}$ along Seifert fibers, and therefore $Y_Q$ has a Seifert fibered structured over the glued base spaces; namely, over a 2--sphere with four exceptional fibers of orders $p$, $q$, $p$, and $q$.
\end{proof}  

Henceforth, we will let $S(p,q,p,q)$ denote the base space of $Y_Q$, sometimes abbreviating this with just $S$.    A surface in a Seifert fibered space is called \emph{vertical} if it is a union of fibers or \emph{horizontal} if it is transverse to every fiber it meets.  It is well-known that every essential surface in a Seifert fibered space is either vertical or horizontal, and closed vertical surfaces are tori~\cite{Hat__Notes-on-Basic}.  Let $\overline{\rho}\colon Y_Q \rightarrow S(p,q,p,q)$ be the natural projection map that associates each fiber in $Y_Q$ to its corresponding point in $S(p,q,p,q)$, and let $\rho\colon\wh F \rightarrow S(p,q,p,q)$ be the restriction of $\overline{\rho}$ to $\wh F$.

\begin{lemma}
\label{lem:deck}
	The map $\rho\colon \wh F \to S$ is a branched covering of order $pq$, where $S$ is identified with a 2--sphere with four cone points of order $p,q,p,$ and $q$.  The corresponding group of deck transformations is given by $G = \langle\wh\varphi\rangle$, so $\rho\circ\wh\varphi = \rho$.
\end{lemma}

\begin{proof}
Since $\wh F$ is not a vertical torus, it must be a horizontal surface in $Y_Q$, from which it follows that the restriction of $\overline{\rho}$ to $\wh \rho$ is a branched covering map (see~\cite{Sco_83_The-geometries-of-3-manifolds}).  The exceptional fibers meet $\wh F$ in the vertices of the two graphs $\Gamma^{\pm}$, viewed as graphs embedded in $\wh F$ cutting $\wh F$ into $\wh A$.  The regular fibers meet $\wh F$ away from the vertices, where each of these points is contained in fiber that meets $\wh F$ in $pq$ distinct points, so the degree of the cover is $pq$.  The exceptional fibers are precisely the orbits of the vertices of $\Gamma^{\pm}$ under the action of $\wh \varphi$; and each of these orbits meets $\wh F$ either $p$ or $q$ times.  Since $\wh \varphi$ preserves fibers, we have that $\rho \circ \wh \varphi = \rho$.  Finally, as each power of $\wh \varphi$ is a deck transformation and $\langle \wh\varphi\rangle$ contains $pq$ distinct deck transformations, it follows that this is the entire group $G$.
\end{proof}

We refer to $S = S(p,q,p,q)$ as \emph{the pillowcase}, since $S$ can be viewed as the union of two squares along their edges.  The left panel of Figure~\ref{fig:Pillow_to_24-gon} depicts a fundamental domain $R$ of the branched covering map $\rho$.  The center and right panels illustrate the gluings of $R$ induced by $\langle \wh \varphi \rangle$ to form $S$.  In our figures, the cone points are drawn at the corners of $S$. We set the convention that the top corners of the square are the cone points of order $p$ and the bottom corners have order $q$, as in Figure~\ref{fig:Pillow_to_24-gon}. Recall from the proof that $\rho^{-1}$ of a cone point of order $p$ (resp., of order $q$) has a total of $q$ preimages (resp., a total of $p$ preimages) in $\wh F$.

\begin{figure}[h!]
	\centering
	\includegraphics[width=.9\textwidth]{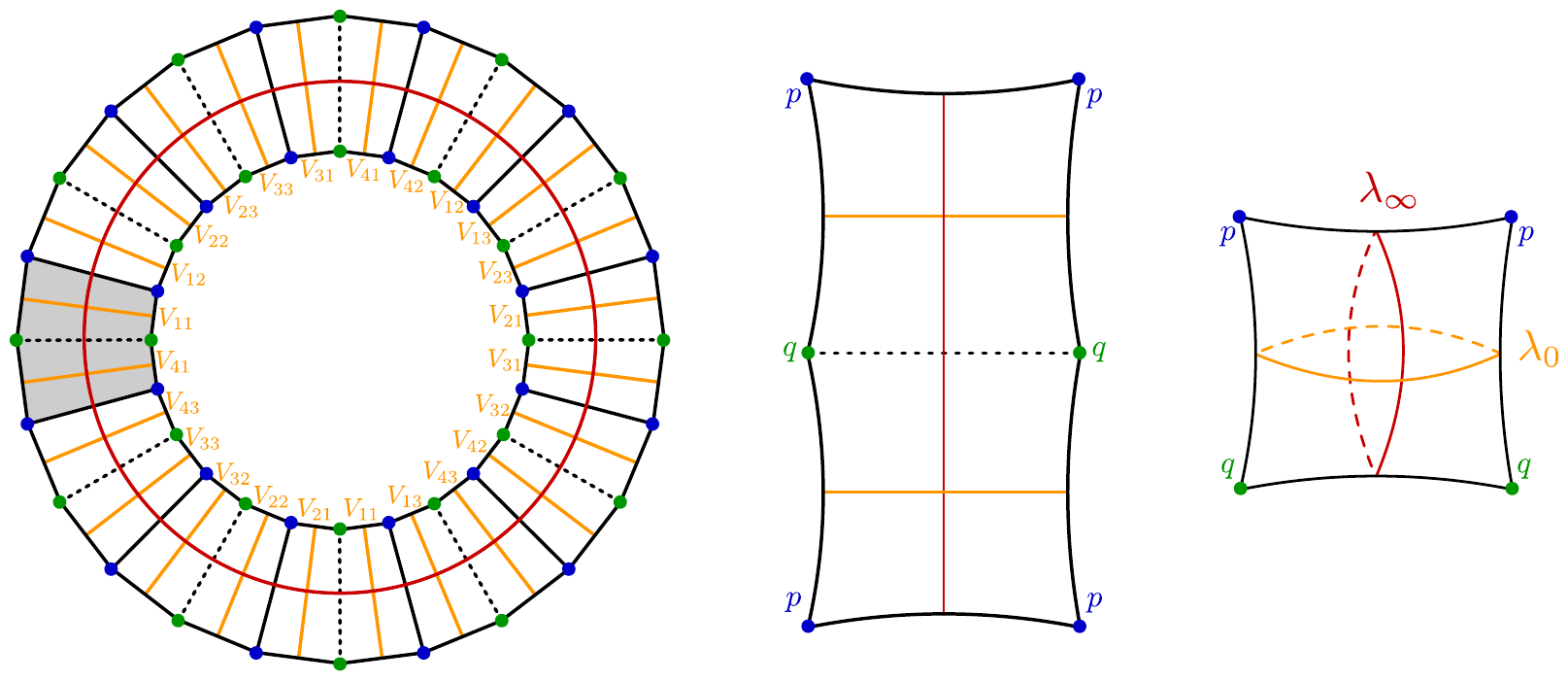}
	\caption{The quotient of $\wh A$ by the action of the monodromy $\wh\varphi$. Shown is the case $(p,q) = (4,3)$. (Left) The annulus $\wh A$ representing the surface fiber $\wh F$, shown with a fundamental domain $\overline S$ for the action of the monodromy $\wh\varphi$ shaded. (Middle) The domain $\overline S$, which can be realized by cutting open the pillowcase.  (Right) The pillowcase $S = S^2(p,q,p,q)$. Shown also are the slopes $\lambda_0$ and $\lambda_\infty$ on $S$, together with their lifts to $\wh F$.}
	\label{fig:Pillow_to_24-gon}
\end{figure}

The next step in this process is to understand the lifting of curves from the pillowcase to $\wh F$.  To begin, let $\lambda_{\infty} = \rho(\Lambda_{\infty})$, where $\Lambda_{\infty}$ is the core of the annulus $\wh A$ described in the proof of Lemma~\ref{Seif}.  Next, note that there is a reflection $\varrho$ of $Y_Q$ through the torus $W_{\infty}$ that swaps $E_{K^+}$ and $E_{K^-}$, in the process transposing the surfaces $F^+$ and $F^-$ and the graphs $\Gamma^{+}$ and $\Gamma^-$.  The reflection $\varrho$ maps Seifert fibers to Seifert fibers; hence it acts on the quotient $S$ as well (as a reflection through the curve $\lambda_{\infty}$).  Let $\lambda_{0}$ be the curve preserved by this reflection, shown at right in Figure~\ref{fig:Pillow_to_24-gon}.

Now, we characterize other essential curves in the pillowcase.  Let $S^*$ be the 4--punctured sphere obtained from $S$ by removing its cone points.  Every curve $\lambda \in S^*$ can be isotoped so that it meets the two unit squares of the pillowcase in parallel arcs with slopes in the extended rational numbers $\Q_{\infty} = \Q\cup\{\infty\}$, where $\infty$ represents the fraction $1/0$.  We call the rational number associated to $\lambda$ the \emph{slope} of $\lambda$.  We let $\lambda_{a/b}$ denote the unique curve in $S^*$ with slope $a/b$, setting the convention that $b \geq 0$.  Note that this definition agrees with our previous descriptions for $\lambda_{\infty}$ and $\lambda_0$.  Since the fractions $\pm 1/1$, $1/0$, and $0/1$ occur frequently, we will use $\pm 1$ in place of $\pm 1/1$, $0$ in place of $0/1$, and $\infty$ in place of $1/0$.

Note that $\Lambda_{\infty} = \rho^{-1}(\lambda_{\infty})$, and let $\Ll_0 = \rho^{-1}(\lambda_0)$.  Recall that $\Lambda_\infty$ is a single curve that separates $\wh F$ into the two surfaces $F^\pm$.  On the other hand, in the example shown in Figure~\ref{fig:Pillow_to_24-gon}, $\Ll_0$ consists of a total of $pq=12$ curves in $\wh F$.  We prove this more generally in the next lemma.  We also show that the lift $\Lambda_1 = \rho^{-1}(\lambda_1)$ is a single curve, just like $\Lambda_{\infty}$.

\begin{lemma}
\label{lem:easy_lifts}
	The lift $\Lambda_1$ is connected, while the lift $\Ll_0$ has of a total of $pq$ connected components.  Moreover, $\wh F \setminus \Ll_0$ is the disjoint union of $q$ copies of the sphere with $p$ boundary components and $p$ copies of the sphere with $q$ boundary components.
\end{lemma}

\begin{proof}
The map $\rho\colon \wh F\to S$ is a cyclic branched covering of order $pq$ and corresponds to a representation $\sigma\colon \pi_1(S^*)\twoheadrightarrow\Z_{pq}$.  For a curve $\lambda\subset S^*$, the cardinality of $\rho^{-1}(\lambda)$ is determined by $\sigma([\lambda])\in\Z_{pq} = \langle t\,|\, t^{pq}\rangle$.
For example, if $\lambda$ is a boundary component of $S^*$ corresponding to a cone point of order $p$ (resp., $q$), then $|\rho^{-1}(\lambda)|$ is $q$ (resp.,~$p$), since $\sigma([\lambda])$ is $t^{ap}$ for some $a=1,\ldots,q-1$ (resp., $t^{bq}$ for some $b=1,\ldots,p-1$).
If $\lambda$ separates $S$ into regions that contain one cone point of order $p$ and one of order $q$ (as in the case of $\lambda_1$), then $\sigma([\lambda])=t^{ap+bq}$, which is a generator.  It follows that $|\Lambda_1| = |\rho^{-1}(\lambda_1)| = 1$.  If $\lambda$ separates the cone points of order $p$ from those of order $q$ (as in the case of $\lambda_0$), then $\sigma([\lambda]) = 1$, since the boundary components of $S^*$ necessarily map to pairs of inverses in $\Z_{pq}$.  In this case, $|\Ll_0| = |\rho^{-1}(\lambda_0)|=pq$.

For the second part of the proof, recall that $\eta^+_{i,j}$ denotes the co-core of the band $B_{i,j}$ in the surface $F^+$, and let $\eta^-_{i,j}$ denote the corresponding co-core in $F^-$.  The reflection $\varrho$ of $Y_Q$ through $W_{\infty}$ sends $\eta^{\pm}_{i,j}$ to $\eta^{\mp}_{i,j}$, and thus the curve $\eta^+_{i,j} \cup \eta^-_{i,j}$ is preserved by $\varrho$ and satisfies $\rho(\eta^+_{i,j} \cup \eta^-_{i,j}) = \lambda_0$.  There are $pq$ curves of this form in $\wh F$, and these curves are permuted by $\wh \varphi$; thus the lift $\Ll_0 = \rho^{-1}(\lambda_0)$ is the union of these $pq$ curves.

For the final part of the proof, let $\eta^{\pm} = \bigcup \eta^{\pm}_{i,j}$.  Note $F^+ \setminus \eta^+$ is $p+q$ disks, where $p$ of these disks each have $q$ boundary arcs in $\eta^+$, and $q$ of these disks each have $p$ boundary arcs in $\eta^+$.  Since $\varrho$ preserves $\Ll_0$, we have that each component of $\wh F \setminus \Ll_0$ is the union of a component of $F^+ \setminus \eta^+$ and its image under $\varrho$, which is a component of $F^- \setminus \eta^-$.  Each of the $p$ disks with $q$ boundary arcs in $\eta^+$ is glued to one of $p$ disks in $F^- \setminus \eta^-$ with $q$ boundary arcs in $\eta^-$ to form a sphere with $q$ boundary components.  Likewise, each of the $q$ disks with $p$ boundary arcs in $\eta^+$ is glued to one of $q$ disks in $F^- \setminus \eta^-$ with $p$ boundary arcs in $\eta^-$ to form a sphere with $p$ boundary components.  The statement of the lemma follows.
\end{proof}

\subsection{Lifting curves and Dehn twists from the pillowcase}
\label{subsec:lifting}
\ 

Let $F^*$ denote $\wh F$ with the vertices of $\Gamma^{\pm}$ removed, so that $\rho \colon F^* \rightarrow S^*$ is a regular covering map of degree $pq$, by Lemma~\ref{lem:deck}, and the group of deck transformations is the cyclic group $G = \Z_{pq}$ generated by $\wh \varphi$.  (In an abuse of notation, we denote the restrictions of $\wh\varphi$ and $\rho$ from $\wh F$ to $F^*$ simply by $\wh\varphi$ and $\rho$, respectively.)  Recall that curves in $S^*$ are parameterized by the extended rational numbers $\Q_{\infty}$.  For two curves $c,c'$ in a surface, the geometric intersection $\iota(c,c')$ is defined to be the minimum of $|c \cap c'|$ up to homotopy.  The next lemma is standard.

\begin{lemma}
\label{inter}
	For any two curves $\lambda_{a/b},\lambda_{c/d} \in S^*$, their intersection number is
	$$\iota(\lambda_{a/b},\lambda_{c/d}) = 2 \cdot |ad-bc|.$$
\end{lemma}

Let $\tau_{a/b}\colon S\to S$ denote a left-handed Dehn twist along $\lambda_{a/b}$.  More precisely, let $\nu_{a/b} = \lambda_{a/b}\times[0,1]$, parameterized by $\psi \in \R/\Z$ and $t \in [0,1]$, and define $\tau_{a/b}$ to be the identity outside of this annulus.  On this annulus, we define
$$\tau_{a/b}(\psi,t) = (\psi-t,t).$$
The action of $\tau^\pm_{a/b}$ on curves in $S$ is as follows; this lemma is also standard.

\begin{lemma}
\label{lem:pillow-twist}
	For any $a/b,c/d\in\Q_{\infty}$, any $n\in\Z$, and $\Delta = |ad-bc|$, we have
	$$\tau^{n}_{a/b}(\lambda_{c/d}) = \lambda_{e/f},$$
	where
	$e = c + 2an\Delta$, $f = d + 2bn\Delta$ if $a$ is odd, and $e = c - 2an\Delta$, $f = d - 2bn\Delta$ if $a$ is even.
\end{lemma}

We have chosen to define $\tau_{a/b}$ as a left-handed Dehn twist so that it preserves the sign of the slope of $\lambda_{c/d}$ when $c$ is odd. For example, $\tau^n_\infty(\lambda_0) = \lambda_{2n/1}$ and $\tau_1^n(\lambda_0) = \lambda_{2n/(2n+1)}$. On the other hand, $\tau^n_0(\lambda_\infty) = \lambda_{-1/2n}$.

In the next lemma, we show that applying sequences of the twists $\tau_\infty$ and $\tau_0$ and their inverses to the curves $\lambda_0$, $\lambda_\infty$, and $\lambda_1$ generates all curves $\lambda_{a/b}$ in $S^*$.

\begin{lemma}
\label{twister}
	Let $a/b \in \Q_{\infty}$. If $a$ is even (resp., odd), then there is a product of the Dehn twists $\tau^{\pm 1}_{\infty}$ and $\tau_0$ taking $\lambda_{a/b}$ to $\lambda_{0}$ (resp., $\lambda_{\infty}$ or $\lambda_{1}$).
\end{lemma}

\begin{proof}
	To begin, we compute
		\begin{eqnarray*}
			\tau^{\pm 1}_{\infty}(\lambda_{a/b}) &=& \lambda_{(a \pm 2b)/b}\\
			\tau_{0}^{\pm 1}(\lambda_{a/b}) &=& \lambda_{a/\left(b \mp 2a\right)}
		\end{eqnarray*}
	Recall that we assume that $b \geq 0$; if the above formula results in $a/b$ with  $b < 0$, we replace $a$ and $b$ with $-a$ and $-b$.  If $b = 0$, then $a/b = \infty$ and we are done.  If $a/b = \pm 1$, then since $\tau_\infty(\lambda_{-1}) = \lambda_{1}$, we are done.  Thus, suppose that $b > 0$ and $|a| \neq b$.  We will induct on the ordered pair $(b,|a|)$ with the dictionary ordering.  Thus,  suppose that there is a series of Dehn twists taking $\lambda_{a'/b'}$ to one of $\lambda_{0}$, $\lambda_{\infty}$, or $\lambda_{1}$ for all $a'/b'$ such that $(b',|a'|) < (b,|a|)$.

	First, suppose that $|a| > b$.  If $a > b$, then $2a > 2b > 0$ and thus $a > 2b - a > -a$.  It follows that $|a -2b| < a$, and we have $\tau^{-1}_{\infty}(\lambda_{a/b}) = \lambda_{a-2b/b}$, so the claim holds by induction.  If $a < -b$, then $2a < -2b < 0$ and thus $a < -2b - a < -a$, so that $|a+2b| < |a|$.  In this case, $\tau_{\infty}(\lambda_{a/b}) = \lambda_{(a+2b)/b}$ and the claim holds by induction.  On the other hand, suppose that $|a| < b$.  If $0<a<b$, then $-b < b - 2a < b$, so that the claim holds for $\tau_0(\lambda_{a/b})$ by induction in this case too.  Otherwise, $-b < a < 0$, so that $-b < b+2a < b$, and we apply the inductive hypothesis to $\tau_0^{-1}(\lambda_{a/b})$.

We conclude that there exists a sequence of Dehn twists taking $\lambda_{a/b}$ to one of $\lambda_{0}$, $\lambda_{\infty}$, or $\lambda_{1}$.  Finally, observe that each twist preserves the parity of the numerator.  Thus, if $a$ is even, these twists take $\lambda_{a/b}$ to $\lambda_{0}$.  Otherwise, $a$ is odd and the result of the twists is $\lambda_{\infty}$ or $\lambda_{1}$.
\end{proof}

Now, we define homeomorphisms $\wt \tau_0,\wt \tau_{\infty}\colon\wh F \rightarrow \wh F$, which lift the Dehn twists $\tau_0$ and $\tau_{\infty}$.  Recalling that $\Ll_0$ contains $pq$ curves, let $\wt \tau_0$ be the product of a single left-handed Dehn twist performed on each of these curves. (The order is not important since these Dehn twists commute.)  The homeomorphism $\wt \tau_{\infty}$ is slightly more complicated.  Recalling that $\wh F \setminus \Lambda_{\infty} = F^+ \cup F^-$, define $\wt \tau_{\infty}$ to be the identity on $F^-$, the inverse monodromy map $(\varphi^+)^{-1}$ on $F^+$, and a $1/pq$ left-handed Dehn twist in an annular neighborhood of $\Lambda_{\infty}$.  In coordinates, we parameterize the neighborhood $\Lambda_{\infty} \X I$ as $\{(\psi,t)\,|\, \psi\in\R/\Z, t\in[0,1]\}$, where $\pd F^- = \Lambda_{\infty} \X \{0\}$ and $\pd F^+ = \Lambda_{\infty} \X \{1\}$.  On $\Lambda_{\infty} \X I$, the twist is defined as
$$\wt \tau_{\infty}(\psi,t) = \left(\psi -\frac{t}{pq}, t\right).$$
Observe that $\wt \tau_{\infty}$ is well-defined, it restricts to the identity map on $\pd F^-$ and restricts to a $1/pq$ counterclockwise rotation on $\pd F^+$; hence it is a homeomorphism of $\wh F$.  We prove the claimed lifting properties with the next lemma.

\begin{lemma}\label{twist-lift}
The homeomorphism $\wt \tau_0$ is a lift of $\tau_0$, and the homeomorphism $\wt \tau_\infty$ is a lift of $\tau_{\infty}$.
\end{lemma}

\begin{proof}
First, we prove that $\rho \circ \wt \tau_0 = \tau_0 \circ \rho$.  Outside of a regular neighborhood of $\Ll_0$, the multi-twist $\wt \tau_0$ is the identity map, and the same is true for $\tau_0$ outside a regular neighborhood of $\lambda_0$.  The restriction of $\rho$ to each component of $\Ll_0$ is a homeomoprhism to $\lambda_0$, which extends to a homeomorphism of an annular neighborhood of each component of $\Ll_0$.  It follows that $\rho \circ \wt \tau_0 = \tau_0 \circ \rho$ in each of these annular neighborhoods, and thus it holds for the entire surface $\wh F$. 

For the second claim, we show that $\rho \circ \wt \tau_{\infty} = \tau_\infty \circ \rho$, proceeding as in the first case.  Outside of a regular neighborhood of $\Lambda_{\infty}$, the map $\wt \tau_{\infty}$ is either the identity or the map $(\varphi^+)^{-1}$, which is the restriction of $(\wh \varphi)^{-1}$ to $F^+$.  Since $\rho \circ \wh \varphi = \rho$, it follows that $\rho \circ \wt \tau_{\infty}  = \rho$ away from $\Lambda_{\infty}$.  Similarly, $\tau_\infty$ is the identity away from $\lambda_{\infty}$, thus $\rho \circ \wt \tau_{\infty} = \tau_\infty \circ \rho$ away from $\Lambda_{\infty}$.  The restriction $\rho\vert_{\nu(\Lambda_{\infty})}\colon \nu(\Lambda_{\infty})\to \nu(\lambda_{\infty})$ is an interval thickening of the canonical $pq$-to-one covering of $S^1$ to $S^1$.  In coordinates, we have
	$$\rho\vert_{\nu(\Lambda_{\infty})}(\psi,t) = (pq\psi,t).$$
	Thus,
	$$
		\rho\vert_{\nu(\Lambda_{\infty})}\circ\wt \tau_{\infty}(\psi,t)
		= \rho\vert_{\nu(\Lambda_{\infty})}\left(\psi-\frac{t}{pq},t\right)
		= (pq\psi - t,t),
	$$
	while
	$$\tau_{\infty}\circ\rho\vert_{\nu(\Lambda_{\infty})}(\psi,t) = \tau_{\infty}(pq\psi,t) = (pq\psi - t,t).$$
It follows that $\rho \circ \wt \tau_{\infty} = \tau_\infty \circ \rho$ on all of $\wh F$, as desired.
\end{proof}


Combining the previous two lemmas, we can show that given any lift $\rho^{-1}(\lambda_{a/b})$, there is a homeomorphism of $\wh F$ that takes this lift to one of $\rho^{-1}(\lambda_0)$, $\rho^{-1}(\lambda_1)$, or $\rho^{-1}(\lambda_\infty)$, depending on the parity of $a$.

\begin{lemma}
\label{lem:autom_lift}
	Given any $a/b \in \Q_{\infty}$, there is a homeomorphism $\wt f\colon \wh F \rightarrow \wh F$ such that $\wt f(\rho^{-1}(\lambda_{a/b}))$ is either $\rho^{-1}(\lambda_0)$ if $a$ is even, or one of $\rho^{-1}(\lambda_1)$ or $\rho^{-1}(\lambda_\infty)$ if $a$ is odd.
\end{lemma}

\begin{proof}
	Let $a/b \in \Q_{\infty}$.  By Lemma~\ref{twister}, there exists a homeomorphism $f\colon S \rightarrow S$, obtained as the product of Dehn twists $\tau^{\pm 1}_{\infty}$ and $\tau_0^{\pm 1}$, such that $f(\lambda_{a/b})$ is either $\lambda_0$ if $a$ is even, or one of $\lambda_1$ or $\lambda_{\infty}$ if $a$ is odd.  By Lemma~\ref{twist-lift}, the homeomorphism $f$ lifts to a homeomorphism $\wt f\colon \wh F \rightarrow \wh F$.  Thus, $\wt f$ maps the lift $\rho^{-1}(\lambda_{a/b})$ to one of the three lifts $\rho^{-1}(\lambda_0)$, $\rho^{-1}(\lambda_1)$, or $\rho^{-1}(\lambda_\infty)$, as desired.
\end{proof}

It follows easily that $\rho^{-1}(\lambda_{a/b})$ contains either one or $pq$ distinct curves, depending on the parity of the numerator $a$.

\begin{proposition}
\label{prop:lift_parity}
	\ 
	\begin{enumerate}
		\item If $a/b \in \Q_{\infty}$ and $a$ is odd, then $\rho^{-1}(\lambda_{a/b})$ is a single separating curve in $\wh F$.
		\item If $c/d\in\Q_\infty$ with $c$ even, then $\rho^{-1}(\lambda_{c/d})$ consists of $pq$ pairwise disjoint curves that are permuted by $\wh \varphi$ and are pairwise non-homotopic in $F^*$.
		\begin{enumerate}
			\item If $q \geq 3$, these curves remain pairwise non-homotopic in $\wh F$.
			\item If $q = 2$, then $\rho^{-1}(\lambda_{c/d})$ contains two curves in each of $p$ distinct homotopy classes of curves in $\wh F$, and $\wh \varphi^p$ swaps a pair of homotopic curves with opposite orientations.
		\end{enumerate}
	\end{enumerate}
\end{proposition}

\begin{proof}
	Suppose $a/b\in\Q_\infty$ with $a$ odd. By Lemma~\ref{lem:autom_lift}, there is a homeomorphism $\wt f$ of $\wh F$ taking $\rho^{-1}(\lambda_{a/b})$ to $\Lambda_\infty$ or $\Lambda_1$, each of which is connected by Lemma~\ref{lem:easy_lifts}, so $\rho^{-1}(\lambda_{a/b})$ is connected, as desired.
	
	Suppose $c/d\in\Q$ with $c$ even. By Lemma~\ref{lem:autom_lift}, there is a homeomorphism $\wt f$ of $\wh F$ taking $\rho^{-1}(\lambda_{c/d})$ to $\Ll_0$.  Thus, it suffices to prove part (2) for $\Ll_0$.  By Lemma~\ref{lem:easy_lifts}, we have that $\Ll_0$ is a separating collection of $pq$ curves in $\wh F$, and $\wh F \setminus \Ll_0$ consists of $p$ spheres with $q$ boundary components and $q$ spheres with $p$ boundary components.  It follows that curves of $\Ll_0$ are non-homotopic in $\wh F$ if and only if $q > 2$.  Otherwise, $q=2$ and $\wh F \setminus \Ll_0$ contains $p$ annuli; hence the curves of $\Ll_0$ are parallel in pairs.  The restriction of $\rho$ is a degree two branched cover from each annulus to the disk component of $S \setminus \lambda_0$ containing the two cone points of order 2; the subgroup $\langle\wh\varphi^p\rangle$ of $\langle \wh \varphi \rangle$ has order two.  Thus, $\wh\varphi^p$ is an involution of each annulus, swapping the boundary components with reversed orientations.
\end{proof}

Moving forward, we distinguish these two cases by letting $a/b \in \Q_{\infty}$ represent an arbitrary fraction with odd numerator and $c/d \in \Q_{\infty}$ represent an arbitrary fraction with even numerator.  In addition, we let $\Lambda_{a/b} = \rho^{-1}(a/b)$ when $a$ is odd, and we let $\Ll_{c/d} = \rho^{-1}(\lambda_{c/d})$ for $c$ even.

In the next lemma, we show that all sets of curves preserved setwise by $\wh \varphi$ must be one of the lifts characterized in this section.  This lemma will be especially important in our classification of Casson-Gordon derivatives in Subsection~\ref{subsec:class}.  We note that it may be the case that two curves in $F^*$ are homotopic in $\wh F$ but not homotopic in $F^*$. (Recall that $F^* = \wh F\setminus\rho^{-1}(\text{cone points})$.) This occurs, for instance, whenever $q =2$; as we saw in Lemma~\ref{lem:easy_lifts}, in this case $\wh F \setminus \Ll_0$ contains $p$ annular components. 

\begin{lemma}
\label{projection}
	Let $\Lambda$ be a collection of pairwise disjoint and non-homotopic curves in $F^*$.  Then, $\wh \varphi(\Lambda) = \Lambda$ if and only if $\Lambda = \rho^{-1}(\lambda_{a/b})$ for some $a/b \in \Q_{\infty}$.
\end{lemma}

\begin{proof}
	Recall that the restriction $\rho\colon F^*\to S^*$ is a cyclic covering map with group of deck transformations generated by $\wh\varphi$, and assume $\wh\varphi(\Lambda) = \Lambda$.  Since $\Lambda$ is an embedded 1--manifold, $\rho(\Lambda)$ is as well.  If any component of $\rho(\Lambda)$ is inessential or if two components are parallel, then the same is true of components of $\lambda$.  Therefore, $\rho(\Lambda)$ is an essential simple closed curve in $S^*$; i.e., $\rho(\Lambda) = \lambda_{a/b}$ for some $a/b\in\Q_\infty$.

	To finish this direction of the proof, we must show that $\rho^{-1}(\rho(\Lambda)) = \Lambda$, which reduces to showing that $\rho^{-1}(\rho(\Lambda)) \subset \Lambda$.  Let $x\in\rho^{-1}(\rho(\Lambda))$ and let $y = \rho(x)$, so $y = \rho(z)$ for $z\in\Lambda$.  Since $\wh\varphi$ generates the cyclic group of deck transformations for the covering $\rho$, we have $\rho^{-1}(y) = \{\wh \varphi^k(x)\,|\,0\leq k\leq pq\}$.  It follows that $\wh\varphi^k(x) = z\in\Lambda$ for some $k$, but since $\wh\varphi(\Lambda) = \Lambda$, we have that $x\in\Lambda$, as desired.
	
	The converse direction is immediate from Lemma~\ref{lem:deck}.
\end{proof}

\begin{remark}
\label{rmk:q2}
	When $q=2$, the branched double cover $\rho_q\colon S(p,p,p,p)\to S(p,2,p,2)$ is an involution, as shown in Figure~\ref{fig:lifts}, and $\rho_p\colon \wh F \rightarrow S(p,p,p,p)$ has a pillowcase as its base space.  Curves in $S(p,p,p,p)$ avoiding the cone points are parametrized in the natural way.  If $c$ is even, the $(\rho_q)^{-1}(\lambda_{c/d})$ is two copies of the curve $\lambda_{c/2d} \subset S(p,p,p,p)$.  If $a$ is odd, then $(\rho_q)^{-1}(\lambda_{a/b}) = \lambda_{a/2b}$.  See Figure~\ref{fig:lifts}.
	
	In the case of $(p,q) = (3,2)$, the authors of~\cite{GomSchTho_Fibered-knots_10} and~\cite{Sch_Proposed-Property_16} work with the pillowcase $S(3,3,3,3)$, and so the slopes in these references are of the form $c/2d$ compare to our $c/d$.  In addition, our slopes have switched signs.  For example, lifts of curves in $S(3,3,3,3)$ of slopes $1/3, 2/5, 3/7$ as defined in~\cite{GomSchTho_Fibered-knots_10} and~\cite{Sch_Proposed-Property_16} correspond to $\Ll_{-2/3},\Ll_{-4/5},\Ll_{-6/7}$ considered as lifts of curves in $S(3,2,3,2)$.
	(When $q=2$, the curves of $\Ll_{c/d}$ occur as $p$ pairs of parallel curves on $\wh F$; in this case, we follow~\cite	{GomSchTho_Fibered-knots_10} and only consider one curve from each pair, as in the right frame of Figure~\ref{fig:lifts}.)
\end{remark}

\begin{figure}[h!]
	\centering
	\includegraphics[width=.9\textwidth]{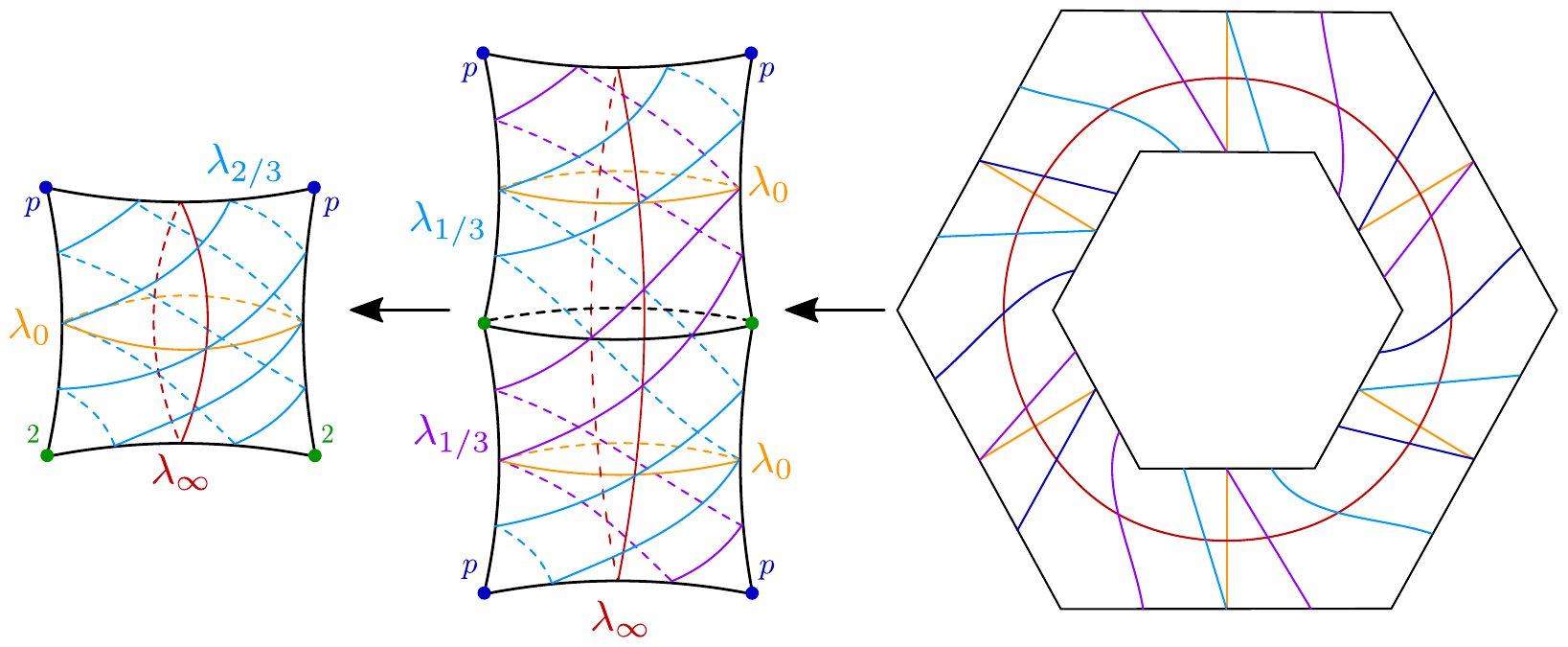}
	\caption{Lifting curves from the pillowcase to the hexulus.  Shown at right are the curves $\Lambda_\infty$ (red), $\Ll_0$ (orange), and $\Ll_{2/3}$ (light blue, dark blue, and violet).}
	\label{fig:lifts}
\end{figure}

\subsection{Classifying the CG-derivatives of $Q_{p,q}$}
\label{subsec:class}

Given any generalized square knot $Q_{p,q}$ and any $c/d\in\Q$ with $c$ even, we have shown how to construct a multi-curve $\Ll_{c/d}$ lying in the closed fiber $\wh F$ for $Q_{p,q}$.  We are now in a position to prove Proposition~\ref{propx:nR}.

\begin{lemma}
\label{lem:cg}
	Every $(p-1)(q-1)$ component sublink $L$ of $\Ll_{c/d}$ that cuts $\wh F$ into a connected planar surface is isotopic to a CG-derivative for $Q_{p,q}$ in $S^3$.
\end{lemma}

\begin{proof}
	Let $N \subset \pi_1(\wh F)$ be the subgroup normally generated by the homotopy classes of curves in $L$, noting that $N$ is also normally generated by all of the curves in $\Ll_{c/d}$.  Since $\wh\varphi$ permutes curves in $\Ll_{c/d}$, it follows that $\wh\varphi_*(N) = N$.  Therefore, Proposition~\ref{cggen} implies that $L$ is a CG-derivative for $Q_{p,q}$.
\end{proof}

Setting $L_{c/d}^{p,q} = L\subset \Ll_{c/d}$, this establishes Proposition~\ref{propx:nR}.  Next, we prove that every CG-derivative for a generalized square knot is equivalent to one of those described in Lemma~\ref{lem:cg}.  In order to understand all CG-derivatives of $Q_{p,q}$ up to handleslide-equivalence, we invoke the Equivariant Loop Theorem (as stated in~\cite{YauMee_The-equivariant-loop_84}).  We also state the Equivariant Sphere Theorem (as stated in~\cite{Dun_An-equivariant-sphere_85}), to be used later to prove Proposition~\ref{prop:rational_branch}. 

\begin{ELT}[\cite{MeeYau_The-classical-Plateau_79,MeeYau_Topology-of-three-dimensional_80,MeeSimYau_Embedded-minimal_82}]
\label{thm:ELT}
	Let $G$ be a finite group acting smoothly on a compact three-dimensional manifold $Y$ such that $Y$ is closed or $\partial Y = F$ and $g(F)=F$ for all $g\in G$.
	\begin{description}
		\item[Loop Theorem] Let $\kappa = \ker(\iota_*)$, where $\iota\colon F\to Y$ is inclusion.  Then there is a collection $\Dd = \{D_i\}_{i=1}^k$ of properly embedded disks in $Y$ with the following properties:
		\begin{enumerate}
			\item $\kappa$ is generated as a normal subgroup of $\pi_1(F)$ by $\{[\partial D_i]\}_{i=1}^k$.
			\item For any $g\in G$ and $1\leq i,j\leq k$, either $g(D_j)\cap D_i = \emptyset$ or $g(D_j) = D_i$.
		\end{enumerate} 
		\item[Sphere Theorem] Let $S\subset Y$ be a two-sphere that does bound a three-ball.  Then there exists such an $S$ such that $g(S) = S$ or $g(S)\cap S = \emptyset$ for all $g\in G$.
	\end{description}
\end{ELT}

We remark that, although the original proofs of the Equivariant Loop Theorem by Meeks and Yau  and the Equivariant Sphere Theorem by Meeks, Simon, and Yau both used analytic techniques, purely topological proofs have since been given by Dunwoody~\cite{Dun_An-equivariant-sphere_85} and Edmonds~\cite{Edm_A-topological-proof_86}.

\begin{proposition}\label{prop:zero}
Suppose that $L^+ \subset S^3$ is a Casson-Gordon derivative for $Q_{p,q}$.  Then there exists $c/d$ with $c$ even such that $Q_{p,q} \cup L^+$ is stably handleslide-equivalent to $Q_{p,q} \cup \Ll_{c/d}$.
\end{proposition}

\begin{proof}
	By the definition of a Casson-Gordon derivative, there exists a handlebody $H$ such that $L^+$ such that the closed monodromy $\wh\varphi\colon \wh F \rightarrow \wh F$ extends to a homeomorphism $\phi\colon H \rightarrow H$, and such that $L^+$ bounds a cut system for $H$.  Since $\wh\varphi^{pq}$ is the identity, $\phi^{pq}$ must also be isotopic to the identity. Since no lesser power of $\wh\varphi$ is the identity, neither is a lesser power of $\phi$. It follows that $\phi$ generates an action of $\Z_{pq}$ on $H$.  By the Equivariant Loop Theorem, there is a finite collection of disks $\Dd=\{D_i\}$ that are properly embedded in $H$ and have the property that the subgroup of $\pi_1(\wh F)$ generated by the curves $\Ll = \pd \Dd$ is equal to the kernel of the map $\iota_*$ induced by the inclusion $\iota\colon\wh F\to H$.  Moreover, for any $1\leq k\leq pq$, we have that either $\phi^k(D_j)\cap(\cup_iD_i)=\emptyset$ or $\phi^k(D_j)\in\{D_i\}$.

Note that since $\phi^{pq}$ is the identity, after deleting parallel disks, the disks in $\Dd$ can be expressed as $\{D_1,\phi(D_1),\dots,\phi^{m-1}(D_1)\}$ for some integer $m$, where $\phi^m(D_1) = D_1$ and $\phi^k(D_1) \neq D_1$ for $k < m$.  In the event that $\phi^m(D_1)$ and $D_1$ have opposite orientations (which will occur when $q=2$), we replace each disk $D$ in $\Dd$ with the ends of an equivariant collar neighborhood $D\X I$ of $D$ in $H$; that is, $D$ is replaced with $D^- = D \X \{0\}$ and $D^+ = D \X \{1\}$.  In this case, $\Dd = \{D_1,\phi(D_1),\dots,\phi^{2m-1}(D_1)\}$ has the property that $\phi^{2m}(D_1) = D_1$ (preserving orientation), and $\phi^k(D_1) \neq D_1$ setwise for any $k < 2m$. Once this is done, we have that $\phi$ cyclically permutes the disks of $\Dd$.

Note that curves in $\Ll = \pd \Dd$ are of the form $\pd (\phi^k(D_1)) = \wh \varphi^k(\pd D_1)$.  We claim that $\Ll \subset \wh F$ does not meet any of the lifts of the cone points of $S$:  Observe that $\Ll$ is invariant under $\wh \varphi$.  If $\Ll$ passes through the lift $x$ of a cone point of order $p$ (resp. $q$), then $\wh\varphi^q$ (resp. $\wh \varphi^p$) induces a $1/p$ (resp. $1/q$) rotation in a neighborhood of $x$.  However, this implies that either $\Ll$ has a transverse self-intersection (in the case $p$ or $q \geq 3$) or that $\wh \varphi^k$ maps a curve in $\Ll$ to itself with opposite orientation (in the case $q = 2$), which has been ruled out by our choice of the disks $\Dd$.  We conclude that $\Ll$ does not meet a lift of a cone point, so that $\Ll \subset F^*$ as in Lemma~\ref{projection}, which asserts that $\Ll  = \Ll_{c/d} = \rho^{-1}(\lambda_{c/d})$ for some $c/d \in\Q_{\infty}$.  Since the kernel of $\iota_*$ is not cyclic, $\Ll$ contains more than one curve, and by Proposition~\ref{prop:lift_parity}, we have that $c$ is even and $|\Ll| = pq$.

Finally, let $L \subset \Ll_{c/d}$ be any collection of curves cutting $\wh F$ into a connected planar surface.  Since both $L^+$ and $L$ are cut systems for the same handlebody, they are handleslide-equivalent in $\wh F$.  Viewing $\wh F$ as a subspace of $Y_{Q_{p,q}}$, we have that $L^+$ is handleslide-equivalent to $L$ in $Y_{Q_{p,q}}$, so that $Q_{p,q} \cup L^+$ is handleslide-equivalent to $Q_{p,q} \cup L$ in $S^3$.  Adding in the rest of the curves in $\Ll_{c/d}$ may be achieved by stable equivalence; hence $Q_{p,q} \cup L$ is stably equivalent to $Q_{p,q} \cup \Ll_{c/d}$.
\end{proof}

\section{The link $Q_{p,q} \cup \Ll_{0}$ has Property R}
\label{sec:L0}

In this section, we give a detailed analysis of the link $\Ll_0$ lying in the fiber $F$ for $Q=Q_{p,q}$ in $S^3$.  First, we prove that $Q \cup \Ll_{c/d}$ is stably equivalent to $Q \cup V$, where $V$ is any one component of $\Ll_{c/d}$.  We then show directly that $Q \cup \Ll_0$ has Property R by showing that $Q \cup V$ is handleslide trivial when $V\subset \Ll_0$.

\begin{lemma}
For any $c/d$ with $c$ even and for any component $V$ of $\Ll_{c/d}$, the link $Q \cup \Ll_{c/d}$ is handleslide-equivalent to $Q \cup V \sqcup U$, where $U$ is a split unlink.
\end{lemma}

\begin{proof}
Since $\wh\varphi$ permutes the curves in $\Ll_{c/d}$, it follows that every component of $\Ll_{c/d}$ is isotopic to $V$ in the 3--manifold $Y_Q$.  Thus, in $Y_Q$, the link $\Ll_{c/d}$ is isotopic to a collection of curves parallel to $V$.  Handleslides in $Y_Q$ convert this collection to $V\sqcup U$, where $U$ is a split unlink.  In $S^3$, this implies that $Q \cup \Ll_{c/d}$ is handleslide-equivalent to $Q \cup V \sqcup U$, as desired.
\end{proof}

The \emph{Farey graph} has vertices corresponding to the extended rational numbers $\Q^{\infty}$, where two rational numbers $p/q$ and $r/s$ are connected by an edge whenever $|ps - qr| = 1$.  A \emph{Farey triangle} is a triple of rational numbers, each connected by an edge.  The Farey graph can also be associated to the 1-skeleton of the curve complex of the torus, as well as the arc complex of the torus with one boundary component.  For further background information on the Farey graph, see~\cite{Hat_Topology-of-Numbers_}.

Now, we turn our attention to understanding the knot types of the components of $\Ll_0$ in $S^3$.  To this end, fix a band $B_{i,j}$ connecting meridional disks $D_i$ of $V$ and $D_j'$ of $V'$, as in Subsection~\ref{subsec:square}. We may suppose $B_{i,j}$ is transverse to the Heegaard torus $T$ containing $K^+$, so that the co-core $\eta^+=\eta_{i,j}^+$ of $B_{i,j}$ is contained in $T$.  As in~\cite{Sch_Proposed-Property_16}, we obtain the curve $V_{i,j} \subset \Ll_0$ in $\wh F$ by gluing $\eta^+ \subset F^+$ to its image $\eta^- \subset F^-$ under the reflection of $\wh F$ across $\Lambda_\infty$ that interchanges $F^+$ and $F^-$. (See Figure~\ref{fig:TK_fib} for an example.) This construction, however, does little to help us determine the knot type of $V_{i,j}$ in $S^3$.  For this purpose, we follow~\cite{Sch_Proposed-Property_16}:  We may homotope $\eta^+$ in $F^+$ (via a homotopy that does not fix $\pd \eta^+$) until its boundary points coincide, yielding a knot.  Since the two points $\pd \eta^+$ cut $K^+ = \pd F^+$ into two arcs, there are two choices for this homotopy; we will let $J_1^+$ and $J_2^+$ denote the resulting knots.  In addition, we let $J_1^-$ and $J_2^-$ denote the corresponding mirror images in $F^-$ obtained from $\eta^-$.

Since components of $\Ll_0$ are constructed by gluing a given co-core to its mirror image, we can mirror the homotopy of $\eta^+$ in $F^-$, so that $V_{i,j} = J_1^+ \# J_1^-$ or $V_{i,j} = J_2^+ \# J_2^-$.  In $Y_Q$ these two knots are isotopic into $\wh F$ and are related by a single handleslide over $Q$, which may be viewed as a homotopy across the disk $D \subset \wh F$.

\begin{lemma}
\label{lem:L0}
	\ 
	\begin{enumerate}
		\item Let $J^+_1$ and $J^+_2$ be defined as above.  As knots in $S^3$, the curves $J^+_1$ and $J^+_2$ are the torus knots $T_{r_1,s_1}$ and $T_{r_2,s_2}$, such that
\begin{enumerate}
\item $0<s_i<r_i$,
\item $|ps_1-qr_1| = |ps_2-qr_2| = |r_1s_2-s_1r_2| = 1$,
\item $0 < r_1,r_2 < p$, and
\item $0 < s_1,s_2 < q$.
\end{enumerate}

		\item After slides over $Q$ in $S^3$, each component $V_{i,j}$ of $\Ll_0$ is either $Q_{r_1,s_1}$ or $Q_{r_2,s_2}$.
		\item After slides over $Q$ in $S^3$, there is a genus two Heegaard surface $\Sigma$ for $S^3$ and a component $V=Q_{r_1,s_1}$ of $\Ll_0$ such that $Q \cup V \subset \Sigma$, and there is a reducing curve $\delta$ for $\Sigma$ cutting $Q$ and $V$ into their respective summands.
	\end{enumerate}
\end{lemma}

\begin{proof}
First, observe that we may crush each band $B_{i,j}$ to its co-core $\eta^+_{i,j}$, so that $F^+$ may be viewed as the union of the disks $D_1,\dots,D_p$ and $D_1',\dots,D_q'$, where disks meet along the co-cores $\eta^+_{i,j}$.  This implies that $K^+ \cup \eta^+$ is an embedded graph in the Heegaard torus $T^+$.  The endpoints of $\eta^+$ cut the knot $K^+ = T_{p,q}$ into arcs $\omega_1$ and $\omega_2$ in $T$, where $J^+_i = \eta^+ \cup \omega_i$, from which it follows that $J^+_i$ is a torus knot $T_{r_i,s_i}$.  Note further that a parallel pushoff of $K^+$ in $T$ meets $J^+_i$ in a single point.  Moreover, $J^+_1$ and $J^+_2$ may be constructed by taking the disjoint arcs $\omega_1$ and $\omega_2$ and connecting them with copies of $\eta^+$ that meet in a single point, as shown in Figure~\ref{fig:farey}, so that $|J^+_1 \cap J^+_2| = 1$.  We conclude that the curves $K^+,J^+_1,J^+_2$ form a Farey triangle in the curve graph of $T^+$.

\begin{figure}[h!]
	\centering
	\includegraphics[width=.5\textwidth]{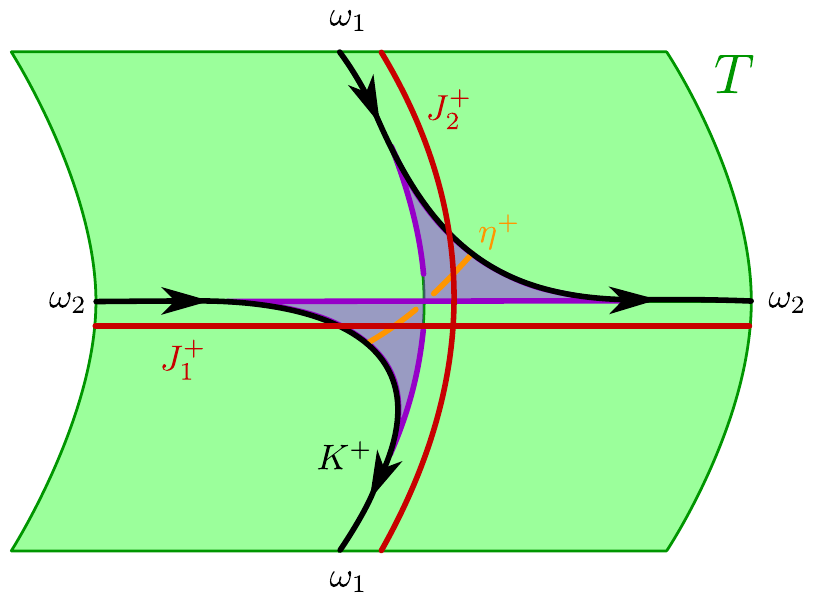}
	\caption{A local picture of the Heegaard torus $T^+$ containing $K^+$, $J_1^+$, and $J_2^+$ near a band of a Seifert surface for $K^+$ whose co-core is $\eta^+$.}
	\label{fig:farey}
\end{figure}

Recall that $D_j'$ is the meridian disk for $V'$ containing $\eta_+$, and a pushoff of $K^+$ meets $D_j'$ transversely in $p$ points of positive sign.  A slight pushoff of $\eta_+$ is disjoint from $D_j'$, and assuming each arc $\omega_i$ meets $D_j'$ transversely in at most $p$ points of positive sign, we have that $r_i \leq p$ for $i = 1,2$, forcing $0 < r_i < p$ since the curves meet pairwise once.  A similar argument using $D_i$ instead of $D_j'$ shows that $0 < s_i < q$.  The second statement of the lemma follows from the fact that $V_{i,j}$ is the connected sum $J^+_1 \# J^-_1$ or $J^+_2 \# J^-_2$.

To see that the final statement is true, we first homotope the arc $\eta^+$ along $K^+$ in $T^+$ so that its boundary points are close, we let $K^- \cup \eta^-  \subset T^-$ be the corresponding mirror images, and we take the connected sum of $T^+$ and $T^-$ along disks that contain the boundary points of $\eta^{\pm}$.  The resulting link is $Q \cup V$, contained in the Heegaard surface $T^+ \# T^-$ with a reducing curve $\delta$ as desired.
\end{proof}

As mentioned above, the Farey graph corresponds to the arc complex of $T^+$, a torus with one boundary component, and every triple $(\A_0,\A_1,\A_2)$ of pairwise disjoint non-homotopic arcs in $T^+$ corresponds to a triangle in the Farey graph. The process of replacing a pair of curves in a triple, say $(\A_0,\A_1)$, with a different pair from the same triple, say $(\A_1,\A_2)$, is called an \emph{arc-slide}.  Any two edges in the Farey graph can be connected by a path of Farey triangles, and thus any two pairs of disjoint arcs in $T^+$ can be related by a sequence of arc-slides.  We use these ideas in the proof of the next proposition.

\begin{proposition}\label{prop:zeror}
There is a component $V\subset\Ll_0$ such that the link $Q\cup V$ has Property~R.
\end{proposition}

\begin{proof}
By Lemma~\ref{lem:L0}, there exists a component $V$ of $\Ll_0$ and a genus two Heegaard surface with reducing curve $\delta$, where $\delta$ cuts $\Sigma$ into $T^+ \cup T^-$, $Q$ into $\A_0^+ \cup \A_0^-$, and $V$ into $\A_1^+ \cup \A_1^-$, such that $\A_i^-$ is the mirror image of $\A_i^+$ over $\delta$.  Since $\A_0^+$ and $\A_1^+$ are disjoint, non-homotopic arcs in $T^+$, they determine an edge in the Farey graph.  Any handleslide of $Q$ over $V$ along an arc contained in $\delta$ can be realized as a pair of mirrored arc slides in $T^+$ and $T^-$, and vice versa.

Again using Lemma~\ref{lem:L0}, we have that the arcs $\A_0^{\pm}$ and $\A_1^{\pm}$ are identified with the fractions $p/q$ and $r_1/s_1$, respectively.  By the remark preceding the proposition, there is a sequence of (mirrored) arc slides taking $(\A_0^{\pm},\A_1^{\pm})$ to the pair $(\n_0^{\pm},\n_1^{\pm})$ corresponding to the fractions $0/1$ and $1/0$ respectively.  This implies there is a sequence of handleslides taking $Q \cup V$ to $Q_{0,1} \cup Q_{1,0} \subset \Sigma$, which is the 2-component unlink.

\end{proof}

\section{Twisting on vertical tori in $Y_Q$}\label{sec:autom}

The purpose of this section is to define two useful diffeomorphisms, $\Tt_0$ and $\Tt_{\infty}$, of the 3--manifold $Y_Q$, each of which is described as a twist on a vertical torus.  Recall that in Lemma~\ref{twist-lift}, we showed that the Dehn twists $\tau_0$ and $\tau_{\infty}$ of $S$ lift to homeomorphisms $\wt \tau_0$ and $\wt \tau_{\infty}$ of $\wh F$.  The main result in this section is that the twist $\Tt_0$ preserves fibers of $Y_Q$, and the restriction of $\Tt_0$ to $\wh F$ is $\wt \tau_0$.  The same is not quite true for $\Tt_{\infty}$, but we show that $\Tt_{\infty}$ is isotopic to a diffeomorphism $\Tt_{\infty}'$ that preserves fibers of $Y_Q$ and acts on $\wh F$ as $\wt \tau_{\infty}$.  It follows that the links $\Tt_0(\Ll_{c/d})$ and $\Tt_{\infty}(\Ll_{c/d})$ are isotopic in $Y_Q$ to the links $\wt \tau_0(\Ll_{c/d})$ and $\wt \tau_{\infty}(\Ll_{c/d})$, respectively, which will allow us to extend $\Tt_0$ and $\Tt_{\infty}$ over the 4--manifolds determined by these links.

By Lemma~\ref{Seif}, we have that $Y_Q$ is Seifert fibered over the base space $S = S(p,q,p,q)$.  We previously defined $W_{\infty}$ to be the vertical torus that projects to the curve $\lambda_{\infty} \subset S$.   Define $W_0$ to be the vertical torus in $Y_Q$ that projects to the curve $\lambda_0 \subset S$.

\subsection{Twisting the torus $W_0$}
\label{subsec:even_twist}
\ 

Recall that the multi-curve $\Ll_0$ is contained in $\wh F$ as a collection of $pq$ curves that are cyclically permuted by the monodromy $\wh \varphi$.  The torus $W_0$, which is vertical with respect the to Seifert fibration of $Y_Q$, intersects $\wh F$ in the $pq$ curves of $\Ll_0$.  We parameterize $W_0$ with a meridian and longitude.  Let $\ell_0 \subset W_0$ be a curve parallel to components of $\Ll_0$, and let $\mu_0$ be a regular Seifert fiber in $W_0$.  Parameterize $W_0$ as $(\theta,\psi)\in\R^2/\Z^2$, where $\mu_0 = \{(\theta,0)\,:\,\theta\in\R/\Z\}$ and $\ell_0 = \{(0,\psi)\,:\,\psi\in\R/\Z\}$.

We define a automorphism $\Tt_0$ of $Y_Q$ that is given as a Dehn twist along the torus $W_0$.  Let $N_0 = W_0\times[0,1]$ be a regular neighborhood of $W_0$, parameterized by $(\theta,\psi,t)$, and identify $W_0$ with $W_0\times\{0\}$. Define $\Tt_0$ to be the identity on $Y_Q \setminus N_0$.  On $N_0$, define
$$\Tt_0(\theta,\psi,t) = (\theta,\psi-t,t),$$
noting that $\Tt_0|_{W_0 \X \{0,1\}} = \text{Id}$, and thus $\Tt_0\colon Y_Q \rightarrow Y_Q$ is a diffeomorphism.  A diffeomorphism $\Tt\colon Y_Q \rightarrow Y_Q$ is said to be \emph{surface-fiber-preserving} if it maps surface-fibers to surface-fibers.

\begin{lemma}
\label{lem:even_torus}
The torus twist $\Tt_{0}$ is surface-fiber-preserving, and $\Tt_{0}|_{\wh F} = \wt \tau_0$.\end{lemma}

\begin{proof}
First, we note that $\Tt_0(\wh F) = \wh F$, since $\Tt_0$ fixes the $\theta$ parameter of $N_0$ and is the identity away from $N_0$.  The intersection of $\wh F \cap N_0$ is $pq$ disjoint annuli of the form $\{(\theta_0,\psi,t):\psi \in \R/\Z, t \in [0,1]\}$, and the restriction of $\Tt_0$ to each annulus is a Dehn twist about a component of $\Ll_0$.  Note that the regular fiber $\mu_0$ meets $\wh F$ coherently in each point of intersection; thus all of these Dehn twists are coherently oriented.  Thus, following the proof of Lemma~\ref{twist-lift}, we have that $\Tt_0|_{\wh F} = \tilde\tau_0$, as desired.
\end{proof}

\subsection{Twisting the torus $W_\infty$}
\label{subsec:odd_twist}
\ 

In this subsection, we examine the more complicated case of a twist on $W_{\infty}$.  Here the twist does not preserve $\wh F$ but is isotopic to a diffeomorphism that does; hence, we take care to keep track of this isotopy.

Recall that $W_{\infty} \cap \wh F$ is the curve $\Lambda_\infty$, and $\Lambda_{\infty}$ cuts $\wh F$ into $F^{\pm}$.  Moreover, the torus $W_\infty$ cuts $Y_Q$ into $E_{K^+}$ and $E_{K^-}$.  In Subsection~\ref{subsec:cover}, we defined the orientation-reversing reflection $\varrho$ of $Y_Q$ through the torus $W_\infty$ taking $E_{K^+}$ to $E_{K^-}$.  Using $\varrho$, we see that the natural meridian and longitude of $E_{K^{\pm}}$ are identified in $Y_Q$, and thus the torus $W_\infty$ has a natural longitude $\ell_\infty$ and meridian $\mu_\infty$.  The longitude $\ell_{\infty}$ can be viewed as the identified boundary curves of $F^{\pm}$ in $\wh F$, and thus $\ell_{\infty} = \Lambda_{\infty}$.  Parameterize $W_\infty$ as $(\theta,\psi)\in\R^2/\Z^2$, where $\mu_\infty = \{(\theta,0)\,:\,\theta\in\R/\Z\}$ and $\ell_\infty = \{(0,\psi)\,:\,\psi\in\R/\Z\}$.

As in the previous subsection, we define a automorphism $\Tt_\infty$ of $Y_Q$ that is given as a Dehn twist along the torus $W_\infty$.  Let $N_\infty = W_\infty\times[0,1]$ be a regular neighborhood of $W_\infty$, parameterized by $(\theta,\psi,t)$, and identify $W_\infty$ with $W_\infty\times\{0\}$. Define $\Tt_\infty$ to be the identity on $Y_Q \setminus N_\infty$.  On $N_\infty$, define
$$\Tt_\infty(\theta,\psi,t) = (\theta + t,\psi,t),$$
noting that $\Tt_\infty|_{W_\infty \X \{0,1\}} = \text{Id}$, and thus $\Tt_\infty\colon Y_Q \rightarrow Y_Q$ is a diffeomorphism.

\begin{lemma}
\label{lem:odd_twist}
The torus twist $\Tt_{\infty}$ is isotopic to a surface-fiber-preserving diffeomorphism $\Tt_{\infty}'\colon Y_Q \rightarrow Y_Q$ such that $\Tt_{\infty}'|_{\wh F} = \wt \tau_{\infty}$.
\end{lemma}

\begin{proof}
Recall from Lemma~\ref{Seif} that $Y_Q = E_{K^-} \cup N_{\infty} \cup E_{K^+}$.  We define isotopies on each of these components and glue them together to construct the desired isotopy.  Let $H^-\colon E_{K^-} \X I \rightarrow E_{K^-}$ be the trivial isotopy $H^-(x,s) = x$.  Let $H^+\colon E_{K^+} \X I \rightarrow E_{K^+}$ be the isotopy obtained by flowing once around the bundle structure in the negative direction.  On $\pd E_{K^+} = W_{\infty} \X \{1\}$, this isotopy flows points along regular fibers of the Seifert fibered structure.  In $H_1(W_\infty)$, regular fibers are expressed as $pq[\mu_\infty] + [\Lambda_\infty]$, since the boundary slope of the essential annulus in $E(K^+)$ has slope $pq/1$.  Since the isotopy $H^+$ traverses $1/pq^\text{th}$ of each regular fiber, the restriction of $H^+$ to the boundary $W_\infty \X \{1\}$ of $E_{K^+}$, parameterized as $(\theta,\psi,1)$, is the isotopy $(\theta,\psi,1,s) \mapsto (\theta - s, \psi - s/pq,1)$.  

Define an isotopy $H_{\infty}\colon N_{\infty} \X I \rightarrow N_{\infty}$ by $H_{\infty}(\theta,\psi,t,s) = (\theta - st, \psi - st/pq, t)$.  Then the restriction of $H_{\infty}$ to $W_{\infty} \X \{0\}$ sends $(\theta,\psi,0,s)$ to $(\theta,\psi,0)$, which agrees with $H^-$, and the restriction of $H_{\infty}$ to $W_{\infty} \X \{1\}$ sends $(\theta,\psi,1,s)$ to $(\theta-s,\psi-s/pq,1)$, which agrees with $H^+$.  It follows that we can paste the isotopies $H^{\pm}$ and $H_{\infty}$ together to get an isotopy $H\colon ~Y_Q \X I \rightarrow Y_Q$, where $H(x,0) = x$ by construction.

Define $\Tt'_{\infty}\colon ~ Y_Q \rightarrow Y_Q$ by $\Tt_{\infty}'(x) = H(\Tt_{\infty}(x),1)$.  Then $\Tt'_{\infty}$ is isotopic to $\Tt_{\infty}$ via the isotopy $H(\Tt(x),s)$.  We are left to verify that $\Tt_{\infty}'$ is the desired diffeomorphism.  The restriction $\Tt_{\infty}'|_{E_{K^-}}$ is the identity, and the restriction $\Tt_{\infty}'|_{E_{K^+}}$ is the surface-fiber-preserving diffeomorphism that maps each fiber to its image under $(\varphi^+)^{-1}$.  Consider $\Tt_{\infty}'|_{N_{\infty}}$.  We compute
\[ \Tt_{\infty}'(\theta,\psi,t) = H(\Tt_{\infty}(\theta,\psi,t),1) = H(\theta+t,\psi,t,1) = (\theta,\psi-t/pq,t).\]
Since the $\theta$--coordinate is preserved, it follows that $\Tt'_{\infty}$ is surface-fiber-preserving on $N_{\infty}$, and thus $\Tt'_{\infty}$ is surface-fiber-preserving on the entirety of $Y_Q$.  Finally, note that we have already shown that $\Tt'_{\infty}|_{\wh F}$ agrees with $\wt \tau_{\infty}$ outside of $N_{\infty}$.  If we consider $\wh F$ to be the fiber that meets $N_{\infty}$ in those points such that $\theta=0$, we have that
\[ \Tt'_{\infty}|_{\wh F \cap N_{\infty}}(0,\psi,t) = (0,\psi-t/pq,t) = \wt \tau_{\infty}(\psi,t).\]
From the definition of $\wt\tau_{\infty}$, we conclude that $\Tt'_{\infty}|_{\wh F} = \wt\tau_{\infty}$, as desired.
\end{proof}

In the left panel of Figure~\ref{fig:links_isotopy}, we illustrate a collection of arcs $\eta$ contained in $\wh F  \cap N_{\infty}$.  In the middle panel, we see the image $\Tt_{\infty}(\eta)$, and the in right panel, the image $\Tt_{\infty}'(\eta)$ of the arcs under the isotopy $H$.  

\begin{figure}[ht]
	\centering
	\includegraphics[width=\textwidth]{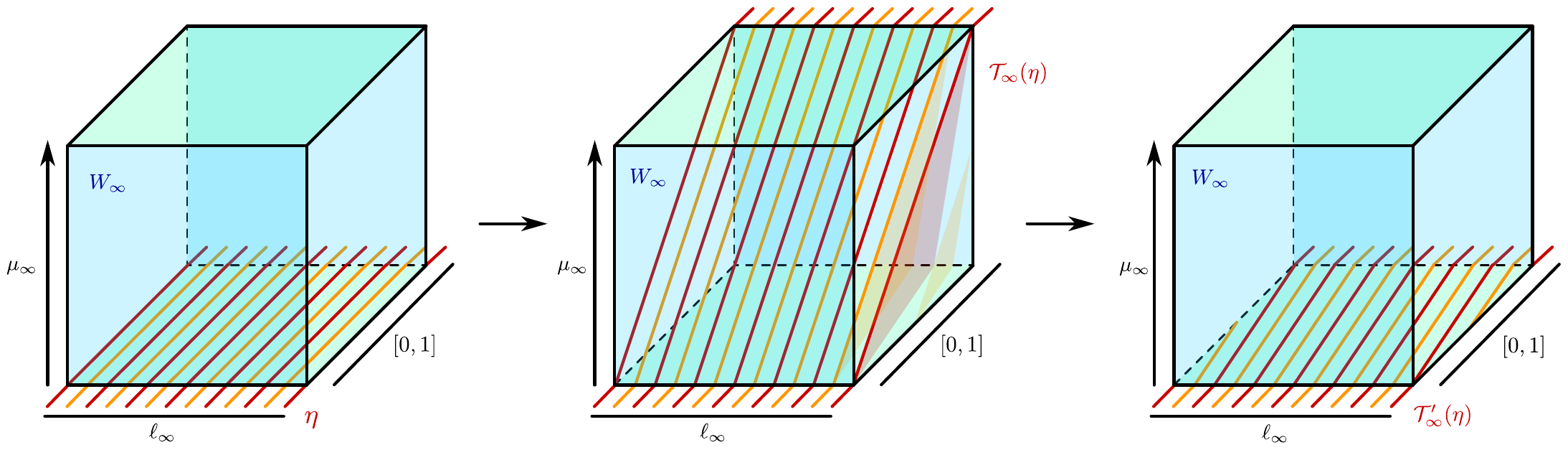}
	\caption{Left: a collection of arcs $\eta$ in $\wh F \cap N_{\infty}$.  Middle: $\Tt_{\infty}(\eta)$.  Right: $\Tt'_{\infty}(\eta)$. Each cube represents $N_\infty$ after the identification of the top side with the bottom side and the right side with the left side.}
	\label{fig:links_isotopy}
\end{figure}

Suppose that $W$ is a torus in a 3--manifold $Y$ and $\Tt\colon Y\to Y$ is a torus twist along $W$ in the direction of a curve $\mu$ on $W$.  If $W$ bounds a solid torus $V$ such that $\mu$ bounds a meridional disk $D$ in $V$, then the twisting can be interpolated to the identity across the solid torus $V$, and thus $\Tt$ is isotopic to the identity.  This property is a higher-dimensional analogue of the fact that a Dehn twist about an inessential curve in a surface is also trivial.

Note that $\Lambda_\infty$ can be isotoped in $\wh F$ to avoid the puncture of $F$ corresponding to $Q$, so that the result is a curve in the punctured surface $F$ preserved set-wise by the (non-closed) monodromy $\varphi$.  (See the right side of Figure~\ref{fig:TK_fib}.)  This isotopy of $\Lambda_{\infty}$ lifts to an isotopy pushing $W_{\infty}$ into $E_Q$.  The two choices for this isotopy correspond the the two distinct swallow-follow tori in $E_Q$, which become isotopic after Dehn filling $E_Q$ to get $Y_Q$, and the choices of $\ell_{\infty}$ and $\mu_{\infty}$ correspond with the natural parameterizations of either torus in $S^3$.  Similarly, we can regard $\widetilde\tau_\infty$ as an automorphism of $F$, as opposed to $\wh F$.

\begin{lemma}
\label{lem:SF_twist}
	The links $Q\cup\Ll_{c/d}$ and $Q\cup\Ll_{(c\pm 2nd)/d}$ are isotopic for any $n\in\N$.  Moreover, the isotopy is the $(\pm n)$--fold meridional Dehn twist about a swallow-follow solid torus for $Q$ in $S^3$.
\end{lemma}

\begin{proof}

Recalling the notation from the previous proof, we may decompose $Y_Q$ into $E_{K^-} \cup N_{\infty} \cup E_{K^+}$, so that $W_{\infty}$ is isotopic to $W_{\infty} \X \{0\}$, and in particular, $W_\infty$ can be made disjoint from the surgery dual knot $Q^*$ to $Q$ in $Y_Q$.  We push $Q^*$ into a parallel copy of $W_{\infty} \X \{0\}$ just outside of $N_{\infty}$ and contained in $E_{K^-}$.  It follows that the torus twist $\Tt_\infty\colon Y_Q \to Y_Q$ and ensuing isotopy are supported away from a neighborhood of $Q^*$.  As such, $\Tt_{\infty}$ and $\Tt_{\infty}'$ can be regarded as (isotopic) diffeomorphisms of either $E_Q$ or of $S^3$.  Viewing $Q \cup \Ll_{c/d}$ as a link in $S^3$, this implies
$$\Tt_{\infty}'(Q \cup \Ll_{c/d}) = Q \cup \wt \tau_{\infty}(\Ll_{c/d}) = Q\cup\Ll_{(c-2d)/d},$$
and so $\Tt_{\infty}(Q \cup \Ll_{c/d})$ is isotopic to $Q \cup \Ll_{(c-2d)/d}$.  (Here, we regard $\widetilde\tau_\infty$ as an automorphism of $F$.)  Since $W_\infty$ is a swallow-follow torus, it bounds a solid torus $V_\infty \subset S^3$ such that $\mu_\infty$ bounds a disk in $V_\infty$, as discussed above.  Thus, $\Tt_{\infty}$, regarded as a diffeomorphism of $S^3$, is isotopic to the identity.  By repeated iterations of $\Tt_{\infty}$ or its inverse, we can conclude that $Q\cup \Ll_{c/d}$ is isotopic to $Q\cup \Ll_{(c\pm 2d)/d}$.
\end{proof}

\section{Standarding the Casson-Gordon Spheres}
\label{sec:std}

In this section, we prove the main theorems, Theorems ~\ref{thmx:PC} and~\ref{thmx:weak}, which assert that any two component R-link of the form $Q_{p,q} \cup J$ has Weak Property R.  As above, fix $Q = Q_{p,q}$.
As in the proof of Proposition~\ref{CGR}, let $B_{c/d}$ denote the Casson-Gordon ball obtained by adding 4--dimensional 2--handles to $S^3 \X I$ along $Q \cup \Ll_{c/d}$, followed by 3--handles and a 4--handle. Let $X_{c/d}$ denote the Casson-Gordon sphere, obtained by capping off $B_{c/d}$ with a standard $B^4$.
We define $Z_{c/d} \subset B_{c/d}$ to be the compact 4--manifold obtained by attaching 2--handles to $Y_Q \X I$ along $\Ll_{c/d}$, followed by 3--handles and a 4--handle.  Let $X_Q$ be the compact 4--manifold obtained by attaching a 2--handle to $B^4$ along the 0--framed knot $Q \subset S^3$, which we refer to as the \emph{trace} of $Q$.

\begin{lemma}\label{lem:deccom}
The Casson-Gordon sphere $X_{c/d}$ decomposes as $X_Q \cup_{Y_Q} Z_{c/d}$.
\end{lemma}

\begin{proof}
Observe that $B_{c/d}$ can be obtained by attaching a 2--handle to $S^3 \X I$ along $Q$ and then capping off the resulting $Y_Q$ boundary component with $Z_{c/d}$.  Thus, we can construct $X_{c/d}$ by attaching a 2--handle to $B^4$ along $Q$ followed by capping off the resulting $Y_Q$ boundary component with $Z_{c/d}$. In other words, $X_{c/d} = X_Q \cup_{Y_Q} Z_{c/d}$.
\end{proof}

Above, $X_Q$ is obtained by attaching a 0--framed 2--handle to $B^4$ along $Q$.  Dually, we obtain a relative handle decomposition of $X_Q$ by starting with its boundary $Y_Q$, attaching a 2--handle to $Y_Q$ along the surgery dual $Q^*$, and capping off the resulting $S^3$ with a 4--ball.  Let $\wh F'$ be a slight pushoff of $\wh F$ in $Y_Q$ in the positive direction.  The surgery dual $Q^*$ decomposes as the union of two arcs, $e \cup f$, where $f$ is a component of a regular Seifert fiber cut along $\wh F'$, and $e$ is an arc connecting $\pd f$ in a parallel copy of $\Lambda_{\infty} \subset \wh F'$.  Observe that if $\pd f = \{x_0,x_1\}$, we have that $\wh\varphi(x_0) = x_1$, and $e$ is the trace of the isotopy dragging $x_1$ back to $x_0$ in the description of the monodromy for $E_Q$ shown in Figure~\ref{fig:mono_drag}.  See Figure~\ref{fig:core} for a depiction of $Q^*$.  The description of $Q^*$ with $e \subset \wh F'$ instead of $\wh F$ is important below, where we consider $Q^*$ and a component $V \subset \Ll_0$ to be disjoint components of a link.

\begin{figure}[h!]
	\centering
	\includegraphics[width=.15\textwidth]{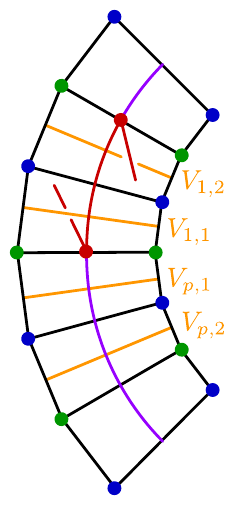}
	\caption{The curve $Q^*\subset Y_Q$.}
	\label{fig:core}
\end{figure}

By Proposition~\ref{prop:zero}, the CG-ball $B_{0}$ is the standard 4--ball; hence, $X_Q$ can alternatively be obtained by attaching a 2--handle to $Q^* \subset Y_Q$, attaching two zero-framed 2--handles to $Q \cup V \subset S^3$, where $V$ is any component of $\Ll_0$, followed by two 3--handles and a 4--handle.  In total, $X_Q$ is determined by the attaching link $L^* \subset Y_Q$ for its three 2--handles, consisting of $Q^*$, a 0--framed meridian $\mu^*$ of $Q^*$, and a curve $V \subset \Ll_0 \subset \wh F$.  The framing on $V$ is the surface-framing induced by $\wh F$.

Recall that $\Tt_0\colon Y_Q \rightarrow Y_Q$ denotes the torus twist along $W_0$ discussed in Subsection~\ref{subsec:odd_twist}.

\begin{lemma}\label{lem:twistq}
The framed link $\Tt_0(L^*)$ is handleslide-equivalent to the framed link $L^*$ in $Y_Q$.
\end{lemma}

\begin{proof}
First, recall that $\Tt_0$ acts on individual fibers as $\wt \tau_0$ by Lemma~\ref{lem:even_torus}, and thus $\Tt_0$ fixes every component of $\Ll_0$, including $V$.  Similarly, $\mu^*$ is isotopic into a ball disjoint from $W_0$, so that $\Tt_0(\mu^*) = \mu^*$.  Using the notation for components of $\Ll_0$ from Subsection~\ref{subsec:square}, let $V_{1,1}',V_{1,2}' \subset \wh F'$ be pushoffs of the corresponding curves in $\Ll_0$.  Then $Q^*$ meets $W_0$ in precisely two points, the points $e \cap V_{1,1}'$ and $e \cap V_{1,2}'$.  It follows that $\Tt_0(f) = f$ and $\Tt_0(e) = e'$, where $e'$ is an arc in $\wh F'$ obtained by a Dehn twist of the arc $e$ about the curves $V_{1,1}'$ and $V_{1,2}'$, so that $\Tt_0(Q^*) = e' \cup f$.

In the manifold $Y_Q$, the curve $V \subset L^*$ is isotopic to $V_{1,1}$.  If this isotopy meets $\Tt_0(Q^*)$, it can be achieved by isotopy and handleslides over the meridian $\mu^*$.  Thus, after isotopy and handleslides, $V$ can be converted to $V_{1,1}$.  Let $Q'$ be the result of a handleslide of $\Tt_0(Q^*)$ over $V_{1,1}$ that undoes the Dehn twist about $V_{1,1}$, although it changes the framing of $\Tt_0(Q^*)$ by $\pm 1$ (see Figure~\ref{fig:twist_slide}).  Similarly, $V_{1,1}$ is isotopic to $V_{1,2}$ in $Y_Q$, and thus after isotopy and handeslides over $\mu^*$, $V_{1,1}$ can be converted to $V_{1,2}$.  Let $Q''$ be the result of a similar handleslide of $Q'$ over $V_{1,2}$ that undoes the other Dehn twist, so that $Q''$ is isotopic to $Q^*$, where the framings differ by $\pm 2$.  Finally, a slide of $Q''$ over its meridian $\mu^*$ preserves the isotopy type of $Q''$ but changes the framing by $\pm 2$, converting the framed component $Q''$ to $Q^*$.  We conclude that $\Tt_0(L^*)$ is handleslide-equivalent to $L^*$ in $Y_Q$.
\end{proof}

\begin{figure}[h!]
	\centering
	\includegraphics[width=\textwidth]{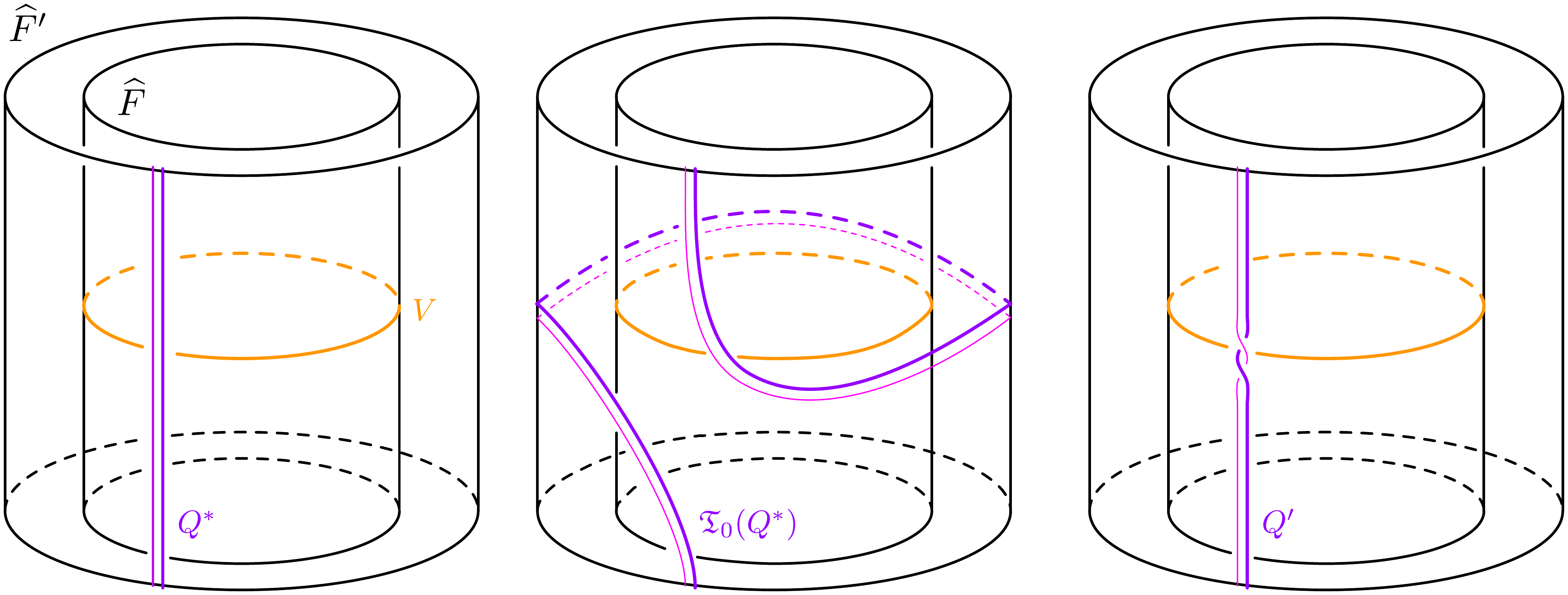}
	\caption{An illustration of the twist $\widetilde\tau_0$ of $Q^*$ about $V'$ producing $\Tt_0(Q^*)$, followed by the slide of $\Tt_0(Q^*)$ over $V$ producing $Q'$, with framings considered throughout.}
	\label{fig:twist_slide}
\end{figure}

Note that Lemma~\ref{lem:twistq} will allow us to extend the automorphism $\Tt_0$ across the trace of $Q$.  The last remaining piece of the puzzle in the proof of the main theorem is the following proposition.

\begin{proposition}
\label{prop:bopit}
	The torus twist $\Tt_0$ extends to a diffeomorphism
	$$\frak T_0: X_{c/d} \rightarrow X_{c/(d-2c)}.$$
	Moreover, the two handle decompositions given by $\frak T_0(Q \cup \Ll_{c/d})$ and $Q \cup \Ll_{c/(d-2c)}$ of $X_{c/(d-2c)}$ become handleslide-equivalent after adding two Hopf pairs to each link.
\end{proposition}

\begin{proof}
	By Lemma~\ref{lem:deccom}, we have that $X_{c/d} = X_Q \cup_{Y_Q} Z_{c/d}$ and $X_{c/(d-2c)} = X_Q \cup_{Y_Q} Z_{c/(d-2c)}$.  By Lemma~\ref{lem:twistq}, we have that $L^*$ and $\Tt_0(L^*)$ are handleslide equivalent, which implies that $\Tt_0$ can be extended to a diffeomorphism from $X_Q$ to $X_Q$.  Similarly, Lemma~\ref{lem:even_torus} asserts that $\Tt_0$ acts on $\wh F$ as $\wt\tau_0$, and thus $\Tt_0(\Ll_{c/d}) = \Ll_{c/(d-2c)}$ by Lemmas~\ref{twister} and~\ref{twist-lift}.  It follows that $\Tt_0$ extends to a diffeomorphism from $Z_{c/d}$ to $Z_{c/(d-2c)}$.  By gluing these diffeomorphisms together, we get a diffeomorphism $\frak T_0$ from $X_{c/d}$ to $X_{c/(d-2c)}$, as desired.

	For the second claim, we again use Lemma~\ref{lem:deccom} to view $X_Q\subset X_{c/d}$.  First, the dual knot $Q^* \subset Y_Q$ determines a relative handle decomposition of $X_Q$ with no 3--handles and a single 4--handle. Consider the split union $L' = Q^*\sqcup U$, where $U$ is a two-component unlink. We consider $U$ to be the attaching circles of the 2--handles of two canceling 2--handle/3--handle pairs.  By Proposition~\ref{prop:zeror}, $L'$ is handleslide-equivalent to $L^*$.  Since $U$ is contained in a ball, it follows that $\frak T_0(L') = \frak T_0(Q^*)\sqcup U$ is handleslide-equivalent to $\frak T_0(L^*)$, which is handleslide-equivalent to $L^*$ by Lemma~\ref{lem:twistq}.  Finally, as noted above, $L^*$ is handleslide equivalent to $L'$.

	In total, $\frak T_0(L')$ is handleslide-equivalent to $L'$.  Noting that $\frak T_0(\Ll_{c/d}) = \Tt_0(\Ll_{c/d}) = \Ll_{c/(d-2c)}$, we can invert the relative handle decompositions of $X_Q$, changing $Q^*$ to $Q$, and the desired statement follows, since the two canceling 2--handle/3--handle pairs described by $U$ invert to become two canceling 1--handle/2--handle pairs -- i.e., two Hopf pairs.
\end{proof}

We note that Theorem~\ref{thmx:PC} follows from the combination of the first statement of Proposition~\ref{prop:bopit} with the previous results in the paper, but we need the second statement to prove the stronger Theorem~\ref{thmx:weak}.

\begin{proof}[Proof of Theorem~\ref{thmx:weak}]
Suppose $Q = Q_{p,q}$ is a generalized square knot, and suppose that $L = Q \cup J$ is a 2R-link.  By Theorem~\ref{thmx:CGequiv}, $L$ is stably equivalent to $Q \cup L^+$, where $L^+$ is a Casson-Gordon derivative for $Q$.  Next, we invoke Proposition~\ref{prop:zero}, which asserts that $Q \cup L^+$ is stably equivalent to $Q \cup \Ll_{c/d}$, where $c$ is even.  By Lemma~\ref{twister}, there is a sequence of Dehn twists $\tau_{\infty}^{\pm 1}$ and $\tau_0^{\pm 1}$ of $S$ taking $\lambda_{c/d}$ to $\lambda_0$, implying that there is a sequence of homeomorphisms $\wt \tau_{\infty}^{\pm 1}$ and $\wt \tau_0^{\pm 1}$ taking $\Ll_{c/d}$ to $\Ll_0$ by Lemma~\ref{twist-lift}.  By Lemma~\ref{lem:SF_twist}, the links $Q \cup \Ll_{c/d}$ and $Q \cup \wt \tau_{\infty}(\Ll_{c/d})$ are isotopic.  By Proposition~\ref{prop:bopit}, we have that the disjoint union of $Q \cup \Ll_{c/d}$ and two Hopf pairs is handleslide-equivalent to an unlink and two Hopf pairs if and only if the same statement is true for $Q \cup \wt\tau_0(\Ll_{c/d})$.  Finally, by Proposition~\ref{prop:zeror}, we have that $Q \cup \Ll_0$ is handleslide-equivalent to an unlink, and thus the disjoint union $Q \cup \Ll_0$ and two Hopf pairs is handleslide-equivalent to an unlink and two Hopf pairs.  We conclude that the same property holds for every link $Q \cup \Ll_{c/d}$, completing the proof.
\end{proof}

\begin{remark}
	One can define torus twists on $Y_Q$ corresponding to a vertical torus lifting any essential curve in $S^*$.  It is possible that a detailed analysis of these twists could yield extra information about the relationships between the the links $Q\cup\Ll_{c/d}$.  In fact, we know this is true in some cases:  A key insight from \cite{GomSchTho_Fibered-knots_10} and~\cite{Sch_Proposed-Property_16} is that the vertical tori $W_{\pm 1}$ sitting above the slopes $\lambda_{\pm1}$ lie in \emph{fishtail neighborhoods} inside the Casson-Gordon 4--sphere in the case of $(p,q) = (3,2)$.  Such neighborhoods have played a central role in the standardization of homotopy 4--spheres.  (See the last paragraph of~\cite{GomSchTho_Fibered-knots_10}.)

	Our techniques, combined with those of~\cite{GomSchTho_Fibered-knots_10}, can be used to show that in the case of $q=2$, the tori $W_{\pm 1}$ lie in fishtail neighborhoods for all odd $p\geq 3$.  In the present development, this is equivalent to showing that the meridian $\mu_{\pm}$, as a curve in $S^3$, has a surface-framing coming from $W_{\pm 1}$ of $\pm 1$.  In light of this, we can conclude not only that $Q_{p,2}$ has Weak Property 2R, but also that only one Hopf pair is required to trivialize any 2R-link $L = Q_{p,2}\cup J$.
\end{remark}

\section{Classifying handlebody-extensions and fibered, homotopy-ribbon disks}
\label{sec:class_CG}

In this section, we show that handlebody-extensions of the closed monodromy $\wh\varphi$ can be understood as deck transformations of branched coverings, and we prove Theorem~\ref{thmx:extensions}.  We also enhance our development of the CG-balls $B_{c/d}$ -- in the vein of Corollary~\ref{coro:CGEX2} -- to take into consideration the fibered, homotopy ribbon disks they contain, and we prove Theorem~\ref{thmx:disks}.

\subsection{Tangles and handlebody extensions}
\label{subsec:tangle}
\ 

Here we discuss yet another perspective on the handlebody extensions of generalized square knots.  Consider the curve $\lambda_{c/d} \subset S$, where $c$ is an even integer.  By Proposition~\ref{prop:lift_parity}, $\lambda_{c/d}$ separates the cone points of order $p$ from those of order $q$.  Thus, there exists an arc $\omega_p$ (resp. $\omega_q$) connecting the cone points of order $p$ (resp. $q$) in $S\setminus\lambda_{c/d}$.  If we consider $S$ as the boundary of a $3$--ball $B^3$, then we can perturb the interiors of the arcs $\omega_p$ and $\omega_q$ into $B^3$ to obtain a rational tangle $T[c/d]$ whose boundary is the orbifold $S$. Since the strands of $T[c/d]$ connect cone points of matching order, we can naturally regard the tangle as a 3--dimensional orbifold.  Let $H_{c/d}$ denote the handlebody with boundary $\wh F$ determined by the curves in $\Ll_{c/d}$, so that $\Ll_{c/d}$ bounds a collection of compressing disks for $H_{c/d}$.

\begin{lemma}
\label{lem:tangle_branch}
	The branched covering $\rho\colon \wh F\to S$ extends to a branched covering
	$$\Rr_{c/d}\colon H_{c/d} \rightarrow T[c/d].$$
\end{lemma}

\begin{proof}	
	The tangle $\Tt[c/d]$ is homeomorphic to $D^2(p,q)\times I$.  Combining Lemmas~\ref{lem:easy_lifts} and~\ref{lem:autom_lift}, we see that the branched cover of $D^2(p,q)$ is a surface $\Sigma'$ of genus $(p-1)(q-1)/2$ with one boundary component.  Taking $\Rr_{c/d}$ to be the product of this covering with $I$, we have $\Rr_{c/d}$ maps $\Sigma' \X I$ to $D^2(p,q) \times I$, where $\Sigma' \X I$ is a handlebody of genus $(p-1)(q-1)$. The curve $\lambda_{c/d}$ bounds a disk in the exterior of the arcs of $\Tt[c/d]$, and this disk lifts to $pq$ disks in $H_{c/d}$, which are bounded by the lifts $\Ll_{c/d}$ of $\lambda_{c/d}$.
\end{proof}

Let $K[c/d]$ denote the rational link obtained as the numerator closure of the rational tangle $\Tt[c/d]$.  Equivalently, $K[c/d]$ is obtained by gluing $\Tt[0]$ to $\Tt[c/d]$ along $S$ via an orientation-reversing homeomorphism identifying the curves $\lambda_0$ on either boundary component.  This link has two components since $c$ is even.   Each component $K_i$ of $K[c/d] = K_1\cup K_2$ admits a 1--bridge splitting, hence is an unknot.  Let $\Sigma_{p,q}(K[c/d])$ denote the $pq$--fold cover of $S^3$ branched along $K[c/d]$, where the component $K_1$ has branching index $p$ and the component $K_2$ has branching index $q$.  One way to construct this cover is to first take the $p$--fold cover of $S^3$ branched along $K_1$, and let $\widetilde K_2$ denote the lift of $K_2$.  Since $K_1$ is unknotted, the result is a new link in $S^3$.  Finish by taking the $q$--fold cover of this $S^3$, branched along the link $\widetilde K_2$.  Alternatively, we could first branch along $K_2$, then over the lift of $K_1$.  For example, if $c/d=2n/1$ for some $n\in\N$, then
$$\Sigma_{p,q}(K[2n/1]) \cong \Sigma_p(T_{n,q})\cong \Sigma_q(T_{n,p})\cong\Sigma(p,q,n),$$
where $\Sigma_m(K)$ denotes the $m$--fold cover of $S^3$ branched over the knot $K$, and $\Sigma(p,q,n)$ denotes the Brieskorn sphere described by Milnor~\cite{Mil_On-the-3-dimensional-Brieskorn_75}. 

Now, we have by definition that $K[c/d] = \Tt[c/d]\cup_S\Tt[0]$.  By taking the union of the branched covering maps, we have $\Sigma_{p,q}(K[c/d]) = H_{c/d}\cup_{\wh F}H_0$.  Henceforth, we let
$$M_{c/d} = \Sigma_{p,q}(K[c/d]) = H_{c/d}\cup_{\wh F} H_0.$$
In addition, since $\wh \varphi$ permutes the disks bounded by $\Ll_{c/d}$, we know that $\wh \varphi$ extends to a homeomorphism $\phi_{c/d}\colon H_{c/d} \rightarrow H_{c/d}$. 

\begin{lemma}\label{lem:auto}
The automorphism $\Phi_{c/d}\colon M_{c/d} \rightarrow M_{c/d}$ defined by $\Phi_{c/d} = \phi_{c/d} \cup_{\wh\varphi} \phi_0$ generates the group of deck transformations for the branched covering $M_{c/d}\to S^3$ with branch locus $K[c/d]$.
\end{lemma}

\begin{proof}
First, recall that $\wh \varphi$ generates the group of deck transformations for the branched covering $\rho\colon \wh F \rightarrow S$, and thus $\phi_{c/d}$ generates the group of deck transformations for the branched covering $\Rr_{c/d}\colon H_{c/d}\to \Tt[c/d]$.  The homeomorphism $\Phi_{c/d}\colon M_{c/d} \rightarrow M_{c/d}$ is defined by taking $\Phi_{c/d}$ on the components of $H_{c/d} \cup_{\wh F} H_0$ to be $\phi_{c/d}\cup_{\wh\varphi}\phi_0$.  Then $\Phi_{c/d}$ is an automorphism of $M_{c/d}$ of order $pq$ that is compatible with the branched covering; hence, $\Phi_{c/d}$ generates the group of deck transformations as desired.
\end{proof}

\begin{lemma}\label{prop:rational_branch}
The 3--manifold $M_{c/d}$ is reducible if and only if $c/d = 0$.  Moreover, the 3--manifold $H_{c'/d'} \cup_{\wh F} H_{c''/d''}$ is reducible if and only if $c'/d' = c''/d''$.
\end{lemma}

\begin{proof}
First, note that $M_0 = H_0 \cup_{\wh F} H_0$ is obtained by gluing two identical genus $(p-1)(q-1)$ handlebodies; thus $M_0 = \#^{(p-1)(q-1)}(S^1\times S^2)$, a reducible 3--manifold.  In the reverse direction, let $c/d\not=0$ and suppose by way of contradiction that $M_{c/d}$ contains an essential 2--sphere $S$.  Let $S_0$ denote the image of $S$ in $S^3$ under the branched covering map.  By Lemma~\ref{lem:auto}, the finite group $G$ generated by $\Phi_{c/d}$ acts on $M_{c/d}$.  Invoking the Equivariant Sphere Theorem (Theorem~\ref{thm:ELT}), we have that $g(S) = S$ or $g(S) = \emp$ for every $g \in G$.
		
	If $g(S)\cap S = \emptyset$, then $S$ is disjoint from the branch locus, so that its image $S_0$ is disjoint from $K[c/d]$.  Since $c/d\not=0$, the link $K[c/d]$ is prime and non-split.  It follows that $S_0$ bounds a three-ball $B$ in $S^3 \setminus K_{c/d}$.  However, $B$ lifts to a three-ball in $M_{c/d}$ bounded by $S$, a contradiction.
	
	If $g(S) = S$, then $S$ intersects the branch locus in a collection of points.  Since $G$ acts cyclically, the induced map on the sphere is a cyclic branched covering, so it must have singular set consisting of two points, by the Riemann-Hurwitz Formula~\cite{Oor_16_The-Riemann-Hurwitz-formula}.  It follows that image $S_0$ is a sphere intersecting $K[c/d]$ in a pair of points.  Since $K[c/d]$ is prime and non-split, $S_0$ must bound a three-ball $B'$ intersecting $K[c/d]$ in a single, unknotted arc.  But then $B'$ lifts to a three-ball bounded by $S$ upstairs, a contradiction. Thus, we conclude that $M_{c/d}$ contains no essential two-spheres, as desired.
	
	For the second statement, note that by Lemma~\ref{lem:autom_lift} there exists a map $\wt f\colon \wh F \rightarrow \wh F$ such that $\wt f(\lambda_{c'/d'}) = \lambda_0$.  Since $\wt f$ is a product of the lifts $\wt \tau_{\infty}^{\pm 1}$ and $\wt \tau_0^{\pm 1}$, there exists $c/d \in \Q_{\infty}$ with $c$ even such that $\wt f(\lambda_{c''/d''}) = \lambda_{c/d}$.  Thus, we can extend $\wt f$ to a diffeomorphism $\wh f\colon H_{c'/d'} \cup H_{c''/d''} \rightarrow M_{c/d}$.  By the first part of the lemma, we have that $M_{c/d}$ is reducible if and only if $c/d = 0$, which is true if and only if $\Ll_{c'/d'} = \wt f^{-1}(\Ll_{0}) = \wt f^{-1}(\Ll_{c/d}) = \Ll_{c''/d''}$, or equivalently, $c'/d' = c''/d''$.
\end{proof}

\subsection{The classification of handlebody extensions}
\label{subsec:h_extensions}
\ 

Recall that $Z_{c/d}$ is defined in Section~\ref{sec:std} as the compact 4--manifold constructed by adding 2--handles to $Y_Q$ along $\Ll_{c/d}$, followed by 3--handles and a 4--handle.  In fact, we can make a stronger assertion following Corollary~\ref{coro:CGEX2} and Proposition~\ref{CGR}.

\begin{lemma}
The 4--manifold $Z_{c/d}$ is diffeomorphic to $H_{c/d} \X_{\phi_{c/d}} S^1$.
\end{lemma}
\begin{proof}
By Proposition~\ref{CGR}, a CG-derivative $L$ may be viewed as a relative handle decomposition for the corresponding bundle $H \X_{\Phi} S^1$.  Thus, $Z_{c/d}$ and $H_{c/d} \X_{\phi_{c/d}} S^1$ have identical relative handle decompositions and as such are diffeomorphic.
\end{proof}

Our next proposition shows that, while on one hand the CG-extensions $\Ll_{c/d}$ and $\Ll_{c'/d'}$ give rise to diffeomorphic handlebody bundles, these bundles are distinct rel-$\partial$ for $c/d\not=c'/d'$.  In other words, these CG-derivatives represent distinct extensions of the closed monodromy $\wh\varphi$.  We say that a diffeomorphism $\Psi_{c/d}\colon Z_{c/d}\to Z_0$ is \emph{handlebody-fiber-preserving} if it sends handlebody-fibers to handlebody-fibers.

\begin{proposition}
\label{prop:HBB_total}
	\ 
	\begin{enumerate}
		\item For any $c/d\in\Q_{\infty}$ with $c$ even, there is a handlebody-fiber-preserving diffeomorphism 
	$$\Psi_{c/d}\colon Z_{c/d}\to Z_{0}.$$
		\item If there is a diffeomorphism from $Z_{c/d}$ to $Z_{c'/d'}$ that restricts to the identity on $Y_Q$, then $c/d=c'/d'$.
	\end{enumerate}
\end{proposition}

\begin{proof}
	By Lemma~\ref{lem:autom_lift}, there is a homeomorphism $\widetilde f \colon \wh F\to \wh F$, obtained as a product of $\wt \tau_{\infty}^{\pm 1}$ and $\wh \tau_0^{\pm 1}$, that covers a homeomorphism $f \colon S\to S$ that preserves cone points and satisfies $f(\lambda_{c/d}) = \lambda_0$.  It follows from Lemmas~\ref{lem:even_torus} and~\ref{lem:odd_twist} that there is a diffeomorphism $\psi_{c/d}\colon Y_Q \rightarrow Y_Q$ obtained as a product of the surface-fiber-preserving maps $(\Tt_{\infty}')^{\pm 1}$ and $\Tt_0^{\pm 1}$ that satisfies $\psi_{c/d}\vert_{\wh F} = \widetilde f$.  By further extending $\psi_{c/d}$ across the cut system for $\Ll_{c/d}$, we get a diffemorphism $\Psi_{c/d}\colon Z_{c/d}\to Z_0$ such that the image of each copy of the handlebody $H_{c/d}$ is a corresponding copy of the handlebody $H_0$, and we conclude that $\Psi_{c/d}$ is handlebody-fiber-preserving.

	Next, suppose that $\Psi\colon Z_{c/d} \to Z_{c'/d'}$ is a diffeomorphism such that $\Psi\vert_{Y_Q} =\id_{Y_Q}$.  Let $\wh Z$ be the closed 4--manifold obtained by gluing $Z_{c/d}$ to $Z_{c'/d'}$ via the identity map on their common boundary $Y_Q$.  Since the handlebody fibers of $Z_{c/d}$ and $Z_{c'/d'}$ have identical boundaries in $Y_Q \subset \wh Z$, it follows that $\wh Z$ fibers over $S^1$, with fibers diffeomorphic to the closed three-manifold
	$$M = H_{c/d}\cup_{\wh F} H_{c'/d'}.$$
	Let $D(Z_{c'/d'})$ be the double of $Z_{c'/d'}$.  Then we may extend the map $\Psi$ to a diffeormorphism $\wh \Psi\colon \wh Z \rightarrow D(Z_{c'/d'})$, by letting $\wh \Psi|_{Z_{c/d}} = \Psi$ and $\wh \Psi|_{Z_{c'/d'}} = \text{Id}$.  Since $D(Z_{c',d'})$ fibers over $S^1$, with fibers the double $Y_g = \#^g(S^1\times S^2)$ of $H_{c'/d'}$ (where $g=(p-1)(q-1)$), the same is true for $\wh Z$.  Note that $Z_{c/d}$ is the complement of a properly embedded disk in a homotopy 4--ball, so that $\Z \cong H_1(Z_{c',d'}) \cong H_1(D(Z_{c',d'})) \cong H_1(\wh Z)$; and thus $\wh Z$ has a unique infinite cyclic cover.

	Since $\wh Z$ fibers over both $M$ and $Y_g$, the infinite cyclic cover of $\wh Z$ must be diffeomorphic to both $M\times\R$ and $Y_g\times \R$.  It follows that $M$ and $Y_g$ are homotopy equivalent.  By the Sphere Theorem~\cite{Pap_On-Dehns-lemma_57}, $M$ is reducible, and thus Proposition~\ref{prop:rational_branch} implies $c/d = c'/d'$.
	\end{proof}

We can now prove Theorem~\ref{thmx:extensions}.

\begin{proof}[Proof of Theorem~\ref{thmx:extensions}]
Suppose $\phi\colon H \rightarrow H$ is a handlebody extension of $\wh \varphi$, and let $L^+$ be a collection of curves bounding disks in $H$ that cut $H$ into a 3-ball.  By Proposition~\ref{prop:zero}, there exists some $c/d \in \Q_{\infty}$ with $c$ even such that after adding some additional curves bounding disks to $L^+$, the collection $L^+$ is handleslide-equivalent in $\wh F$ to $\Ll_{c/d}$.  Thus, $H = H_{c/d}$ and $\phi$ is isotopic to $\phi_{c/d}$.  If there exists some $c'/d' \in \Q_{\infty}$ such that $\Ll_{c/d}$ and $\Ll_{c'/d'}$ determine the same handlebody, then $H_{c/d} \cup H_{c'/d'}$ is a reducible 3--manifold and $c/d = c'/d'$ by Lemma~\ref{prop:rational_branch}.
\end{proof}

\subsection{The classification of fibered, homotopy-ribbon disks}
\label{subsec:ribbon_disks}
\ 

Recall from Section~\ref{sec:std} that the Casson-Gordon ball $B_{c/d}$ is constructed by attaching a 0--framed 2--handle to $S^3 \X I$ along $Q$, followed by gluing in the handlebody bundle $Z_{c/d}$ along the resulting $Y_Q$ boundary component.  Let $R_{c/d} \subset B_{c/d}$ be the core of the 2--handle attached along $Q$, so that $R_{c/d}$ is a disk-knot in $B_{c/d}$, which is diffeomorphic to the standard $B^4$ by Theorem~\ref{thmx:PC}.  By the discussion in Subsection~\ref{subsec:FhRKs}, the disk $R_{c/d}$ is homotopy-ribbon and fibered since $B_{c/d} \setminus R_{c/d} = Z_{c/d}$.

Given any knot $K$, there is a well known ribbon disk $R_K$ for $K\#\overline K$ given as
$$(B^4, R_K) = (S^3,K)^\circ\times I.$$
We refer to $R_K$ as the \emph{product} ribbon disk for $K\#\overline K$.  The following lemma identifies the product ribbon disk for a generalized square knot among the collection $\{(B_{c/d},R_{c/d})\}$ (cf. Section~6 of~\cite{LarMei_Fibered-ribbon_15}).

\begin{lemma}
\label{lem:product}
	The CG-pair $(B_{0},R_{0})$ is the product ribbon disk $(B^4,R_{T_{p,q}})$.
\end{lemma}

\begin{proof}
Let $F^+$ be the genus $(p-1)(q-1)/2$ Seifert surface for $K^+ = T_{p,q}$ discussed in Section~\ref{sec:square}, and let $\Aa$ be the union of the co-cores $\eta_{i,j}^+$ of the bands $B_{i,j}$ on $F^+$, as in Subsection~\ref{subsec:square}.  Puncture the triple $(S^3,F^+,K^+)$ at a point in $K^+$ to get $(B^3,(F^+)^{\circ},(K^+)^\circ)$, and isotope $\partial \Aa$ near $K^+$ in $(F^+)^\circ$ so that $\partial \Aa$ is contained in the puncture; i.e., $\partial \Aa \cap (K^+)^\circ = \emptyset$. Note that
$$(B^3,(F^+)^\circ,(K^+)^\circ)\times I  = (B^4, H, R_{K^+}),$$
where $H=(F^+)^\circ\times I$ is a handlebody of genus $(p-1)(q-1)$ with
	$$\partial H = ((F^+)^\circ \times\{0\})\cup (\pd((F^+)^{\circ}) \X I) \cup ((F^+)^{\circ}\times\{1\}).$$
Furthermore, $\Aa\times I$ is a disk system for $H$.  Let $L = \partial(\Aa \times I)$ be the corresponding cut system of $\partial H$.  By construction, $L$ coincides exactly with the curves $\Ll_0$ (see Figure~\ref{fig:TK_fib}).  Since $K^+$ is fibered, $R_{K^+}$ is fibered as well (via the product fibering) with fiber $H = H_0$.  It follows that $R_{K^+} = R_0$, as desired.
\end{proof}

Recall that a Casson-Gordon sphere is obtained from a Casson-Gordon ball $B_{c/d}$ by capping off with $B^4$.  In what follows, we not only cap off $B_{c/d}$ with $B_0 \cong B^4$, but we also cap off $R_{c/d}$ with $R_0\subset B_0$.  Consider the pair
$$(X_{c/d}, \Kk_{c/d}) = (B_0,R_0)\cup_{(S^3,Q)} (B_{c/d},R_{c/d}),$$
which consists of the Casson-Gordon homotopy 4--sphere $X_{c/d}$ and an embedded 2--sphere $\Kk_{c/d}$ therein. This union respects the fibration of the components, so it follows that $\Kk_{c/d}$ is fibered in $X_{c/d}$.  The fiber is a copy of $H_0$ glued to a copy of $H_{c/d}$ along $F$, which is viewed as $\wh F$ with a disk removed.  Compare this with the 4--manifold $\wh Z$ from Proposition~\ref{prop:HBB_total}, in which these handlebodies are glued along $\wh F$ to obtain the closed fiber $M_{c/d}$ (In fact, $\wh Z$ is obtained from surgery on $X_{c/d}$ along the 2--knot $\Kk_{c/d}$.)  In this context, the fiber of $\Kk_{c/d}$ is $M_{c/d}^\circ$, a punctured version of $M_{c/d}$, and monodromy is $\Phi_{c/d}^\circ$.


\begin{proposition}
\label{prop:disk_rel}
	If there is a diffeomorphism from $(B_{c/d},R_{c/d}
	)$ to $(B_{c'/d'},R_{c'/d'})$ that restricts to the identity on the common boundary $(S^3,Q)$, then $c/d = c'/d'$.
\end{proposition}

\begin{proof}
	Suppose there is such a diffeomorphism. Then 2--knot
	$$(B_{c/d},R_{c/d})\cup_{(S^3,Q)}(B_{c'/d'},R_{c'/d'})$$
is fibered with fiber $M^{\circ} = (H_{c/d} \cup H_{c'/d'})^{\circ}$, and in addition it is diffeomorphic to the fibered 2--knot obtained by doubling $(B_{c'/d'},R_{c'/d'})$.  This double necessarily has fiber $(\#^g(S^1\times S^2))^\circ$. As in the proof of Proposition~\ref{prop:HBB_total}, we can pass to the (unique) infinite cyclic cover of the 2--knot exterior to conclude that $M^\circ$ must be homotopy-equivalent to $Y_g^\circ$.  Again, by Proposition~\ref{prop:rational_branch}, this implies $c/d = c'/d'$.
\end{proof}

On the other hand, if we are allowed to consider diffeomorphisms that act non-trivially on the boundary, many of these CG-pairs become diffeomorphic.

\begin{proposition}
\label{prop:diff_SF}
	For any $n\in\N$, the CG-pairs $(B_{c/d},R_{c/d})$ and $(B_{(c\pm2nd)/d},R_{(c\pm2nd)/d})$ are diffeomorphic.
\end{proposition}

\begin{proof}
	By definition, $B_{c/d}$ is built by attaching 0--framed 2--handles to $Q\cup\Ll_{c/d}$, before capping off with a 4--dimensional 1--handlebody, and $R_{c/d}\subset B_{c/d}$ is the core of the 2--handle attached along $Q$.  By Lemma~\ref{lem:SF_twist}, the links $Q\cup\Ll_{c/d}$ and $Q\cup\Ll_{(c\pm2nd)/d}$ are isotopic in $S^3$.  It follows that $B_{c/d}$ is diffeomorphic to $B_{(c\pm2nd)/d}$, and that this diffeomorphism equates the cores of the two 2--handles attached along the two copies of $Q$.
	\end{proof}

\begin{proof}[Proof of Theorem~\ref{thmx:disks}]
	Part (1) is Lemma~\ref{lem:product}. Part (2) is Proposition~\ref{prop:disk_rel}. Part (3) follows from Proposition~\ref{prop:zero}.  Part (4) is Proposition~\ref{prop:HBB_total}(1).
\end{proof}

\begin{remark}
	The second part of the proof of Proposition~\ref{prop:diff_SF} implies (in particular) that all pairs of the form $(B_{\pm 2n/1},R_{\pm 2n/1})$ are diffeomorphic.  Because this isotopy is given by the torus twist $\Tt_{\infty}$ taking $Q \cup \Ll_{c/d}$ to $Q \cup \Ll_{(c \pm 2nd)/d}$ (as in Lemma~\ref{lem:SF_twist}), we find that $(X_{\pm 2n/1},\Kk_{\pm 2n/1})$ is the $n$--twist spin of the torus knot $T_{p,q}$.  The authors are not aware of a classification of the fibered 2--knots $\Kk_{c/d}$ in the case that $c/d \neq \pm 2n/1$.
\end{remark}

\subsection{Classical examples and results}
\label{subsec:classical}
\ 

In this subsection, we prove Corollary~\ref{corox:classical}, which recovers classical results of Akbulut and Gompf.  Recal the homotopy 4--spheres $\Sigma_m$ and $H(n,k)$ discussed in the introduction; Gompf showed $\Sigma_0$ and the $H(n,k)$ are standard~\cite{Gom_Killing-the-Akbulut-Kirby_91}, while Akbulut showed $\Sigma_m$ is standard for $m>0$~\cite{Akb_Cappell-Shaneson-homotopy_10}.

\begin{proof}[Proof of Corollary~\ref{corox:classical}]
	To apply Theorem~\ref{thmx:PC} to these families of examples, we must show that each admits a handle decomposition with no 1--handles and two 2--handles such that one 2--handle is attached along a generalized square knot.  Such a handle decomposition for $H(n,k)$ is given in Figure~14 of~\cite{GomSchTho_Fibered-knots_10}.  By their discussion in Section 8, page~2334, one of the components of the attaching link $L(n,k)$ for the two 2--handles is the generalized square knot $Q_{n+1,n}$.  This proves the corollary in the case of $H(n,k)$.
	
	Gompf gives a handle decomposition for $\Sigma_m$ in Figure~8 of~\cite{Gom_On-Cappell-Shaneson-4-spheres_91}, and he describes on pages~130--131 how to eliminate the two 1--handles present, as well as one of the 2--handles.  The instructions are to remove the two dotted circles, but add full-twists (one of each sign) to all of the strands passing through them.  Afterwards, the 2--handle given by $xy$ can be cancelled with a 3--handle, so the resulting diagram will have two 2--handles, given by $xz$ and $\alpha$.  We claim that $xz$ is the square knot, $Q_{3,2}$.
	
	To see this, we discard everything from Gompf's Figure~8 except for $xz$ and the two dotted circles.  We remove these dotted circles and add a positive full-twist to the strands passing through the top one and a negative full-twist to the strands passing through the bottom one.  See~Figure~\ref{fig:Gompf}.  That the resulting knot is $Q_{3,2}$ can be verified by simplifying the right frame of Figure~\ref{fig:Gompf}.  This proves the corollary in the case of $\Sigma_m$.
\end{proof}

\begin{figure}[h!]
	\centering
	\includegraphics[width=.9\textwidth]{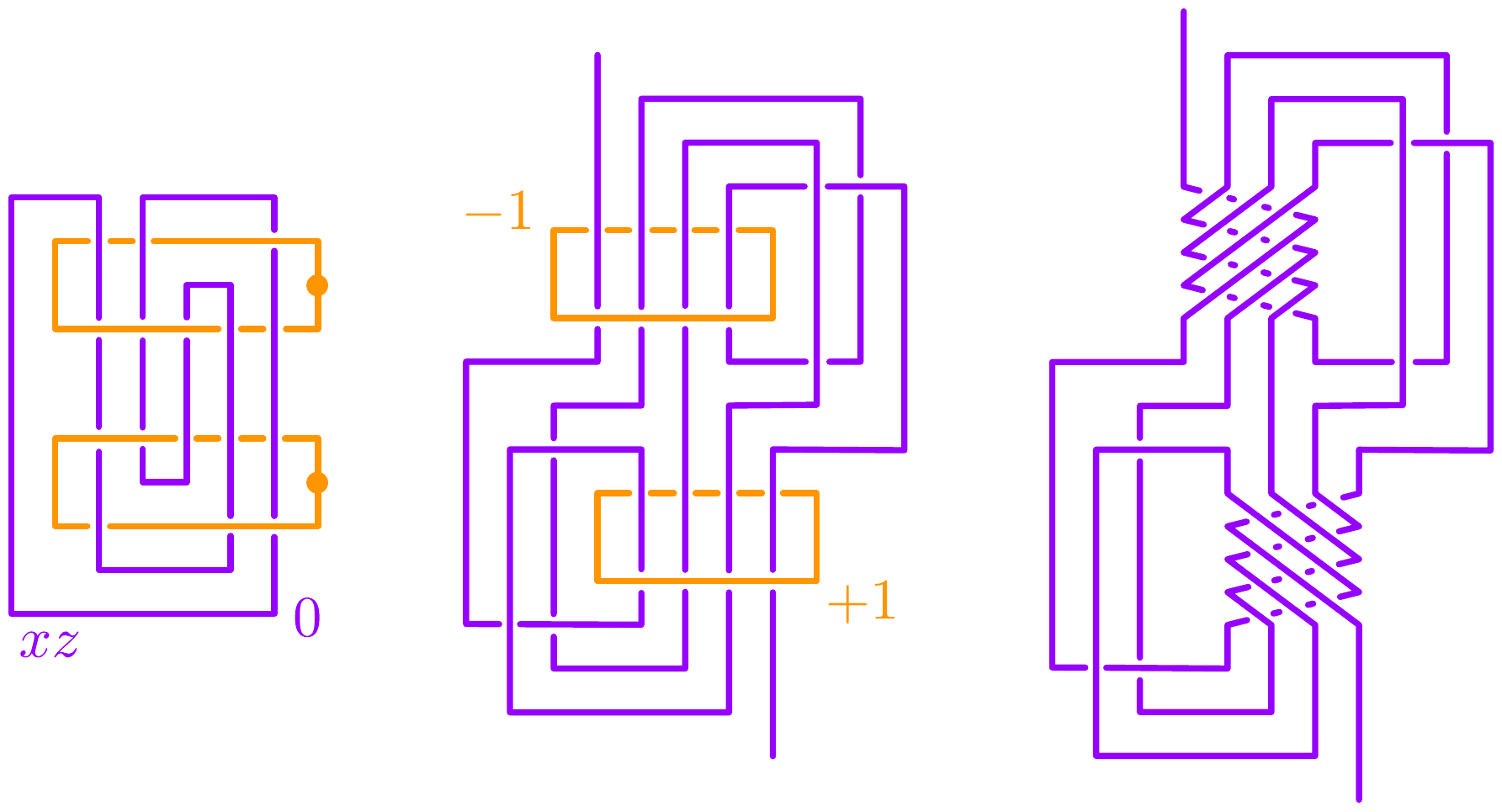}
	\caption{(Left) A sublink from Figure~8 of~\cite{Gom_On-Cappell-Shaneson-4-spheres_91}. (Right) The knot $Q_{3,2} = T_{3,2}\#T_{-3,2}$ in disguise.}
	\label{fig:Gompf}
\end{figure}

We remark that the $\Sigma_m$ are a sub-family of a larger class of homotopy 4--spheres described by Cappell and Shaneson~\cite{CapSha_Some-new-four-manifolds_76}.  Many of these Cappell-Shaneson spheres (beyond the $\Sigma_m$) are known to be standard by work Gompf~\cite{Gom_10_More-Cappell-Shaneson} and Kim and Yamada~\cite{KimYam_17_Ideal-classes}, though handle diagrams have not been given in these cases. General Cappell-Shaneson spheres are not known to be geometrically simply-connected.

\section{Trisecting the Casson-Gordon homotopy four-spheres}
\label{sec:tris}

In this section, we describe a natural trisection of the Casson-Gordon homotopy 4--sphere $X_{\phi}$ corresponding to a handlebody extension $\phi$ of the closed monodromy of a fibered, homotopy-ribbon knot $K \subset S^3$.  We also describe the connections between the R-links arising as Casson-Gordon derivatives, trisections, and the GPRC and Stable GPRC using the authors' framework from~\cite{MeiZup_Characterizing-Dehn_17}.

A Heegaard splitting of a closed 3--manifold $Y$ is a decomposition $Y = H \cup H'$, where $H$ and $H'$ are handlebodies that intersect in their common boundary $\Sigma$, called a \emph{Heegaard surface}.  One dimension higher, a \emph{trisection} $\Tt$ of a closed, smooth 4--manifold $X$ is a decomposition $X = X_1 \cup X_2 \cup X_3$, where $X_i$ is a 4--dimensional handlebody, $H_{ij} = X_i \cap X_j$ is a (3--dimensional) handlebody, and $\Sigma = X_1 \cap X_2 \cap X_3$ is a closed surface.  If $X_i$ has genus equal to $k_i$ and $\Sigma$ has genus $g$, we say that $\Tt$ is a $(g;k_1,k_2,k_3)$--trisection.  A trisection $\Tt$ is uniquely determined by the union $H_{12} \cup H_{23} \cup H_{31}$, called the \emph{spine} of $\Tt$~\cite{GayKir_Trisecting-4-manifolds_16}.

Given a fibered, homotopy-ribbon knot $K \subset S^3$ and a handlebody extension $\phi$ of the closed monodromy of $K$, there is a natural trisection of the CG-sphere $X_{\phi}$, as described in the next proposition.

\begin{proposition}
\label{bundletri}
	Suppose that $K$ is a fibered, homotopy-ribbon knot in $S^3$, with genus $g$ fiber $F$, monodromy $\varphi$, and extension $\phi$ of the closed monodromy $\wh\varphi$.  Then the CG-sphere $X_{\phi}$ admits a $(2g;0,g,g)$--trisection.
\end{proposition}

\begin{proof}
There is a well-known construction of a Heegaard surface for $S^3$ coming from the open book decomposition induced by the fibration of $E_K$:  Let $F_+$ and $F_-$ be two copies of the fiber $F$ in $S^3$, so that $F_+ \cap F_- = K$, and let $\Sigma = F_+ \cup F_-$.  Then each component of $S^3 \setminus \Sigma$ is diffeomorphic to the product $F \X I$ collapsed along $K \X I$; thus, $\Sigma$ cuts $S^3$ into two genus $2g$ handlebodies, which we will call $H_{\A}$ and $H_{\n}$.

Let $R_\phi$ be the CG-disk in $B_{\phi}$ bounded by $K$, and let $\wh F_{\pm}$ be a copy of $F_{\pm}$ capped off in $B_{\phi}$ with the disk $R_\phi$, so that $\wh F_+ \cap \wh F_- = R_\phi$.  Note that $\wh F_\pm$ is not properly embedded: $\wh F_\pm\cap\partial B_{\phi} = F_\pm$ and $\wh F_\pm\cap\Int(B_{\phi}) = \Int(R_\phi)$.  By assumption, $B_{\phi} \setminus R_\phi = H \X_{\phi} S^1$, and as such there is a pair of handlebodies $H_\pm$ in $B_{\phi}$ with $\partial H_\pm = \wh F_\pm$ and $H_+\cap H_- = R_\phi$.  Let $H_{\g} = H_+ \cup H_-$, so that $H_\g$ is the boundary connected sum of $H_+$ and $H_-$, a genus $2g$ handlebody.

We claim that $H_{\A} \cup H_{\n} \cup H_{\g}$ is the spine of a $(2g;0,g,g)$--trisection of $X_{\phi}$.  First, we note that $H_{\A} \cap H_{\n} \cap H_{\g} = \Sigma$, so the triple intersection is as desired.  To complete the proof, it suffices to show that $X \setminus (H_{\A} \cup H_{\n} \cup H_{\g})$ has three components, two of which are genus $g$ 4--dimensional 1--handlebodies and one of which is a 4--ball.  Note that $X_{\phi} \setminus (H_{\A} \cup H_{\n})$ consists of a 4--ball $B^4$ and $B_{\phi}$.  In addition, $B_{\phi}$ cut open along $H_{\g}$ is diffeormorphic to $B_{\phi} \setminus R_\phi = H \X_{\phi} S^1$ cut open along two fibers, say $H \X_{\phi} \{0\}$ and $H \X_{\phi} \{1/2\}$.  Each of the two resulting components is diffeomorphic to $H \X I$, a genus $g$ 4--dimensional 1--handlebody, as desired.
\end{proof}

Note that the above construction depends only on the extension $\phi$ of the closed monodromy of the initial knot $K$; the choices of $F_{\pm}$ are unique up to isotopy.  Therefore, we will let $\Tt_{\phi}$ represent the trisection resulting from Proposition~\ref{bundletri}, without ambiguity.

Every trisection $\Tt$ can be encoded by a \emph{trisection diagram}, which we will define shortly.  A \emph{cut system} in a genus $g$ surface $\Sigma$ is a collection of $g$ pairwise disjoint homotopy classes of curves that cut $\Sigma$ into a connected planar surface.  A cut system $\A$ determines a handlebody $H_{\A}$ by adding 3--dimensional 2--handles to $\Sigma$ along $\A$ and capping off the resulting 2--sphere boundary component with a 3--ball (as in the proof of Lemma \ref{handlebundle}).  A \emph{trisection diagram} for a trisection $\Tt$ is a triple $(\A,\n,\g)$ of cut systems in $\Sigma$ such that $H_{\A} \cup H_{\n} \cup H_{\g}$ is a spine for $\Tt$.  As such, each pair of curves defines a Heegaard diagram for one of the 3--manifold $\pd X_i$.

As an example, there are three genus one trisections of $S^4$, a $(1;1,0,0)$--trisection denoted $\Ss_1$, a $(1;0,1,0)$--trisection denoted $\Ss_2$, and a $(1;0,0,1)$--trisection denoted $\Ss_3$.  Diagrams for $\Ss_1$, $\Ss_2$, and $\Ss_3$ are depicted in Figure~\ref{fig:StdDiag}.

\begin{figure}[h!]
	\centering
	\includegraphics[width=.6\textwidth]{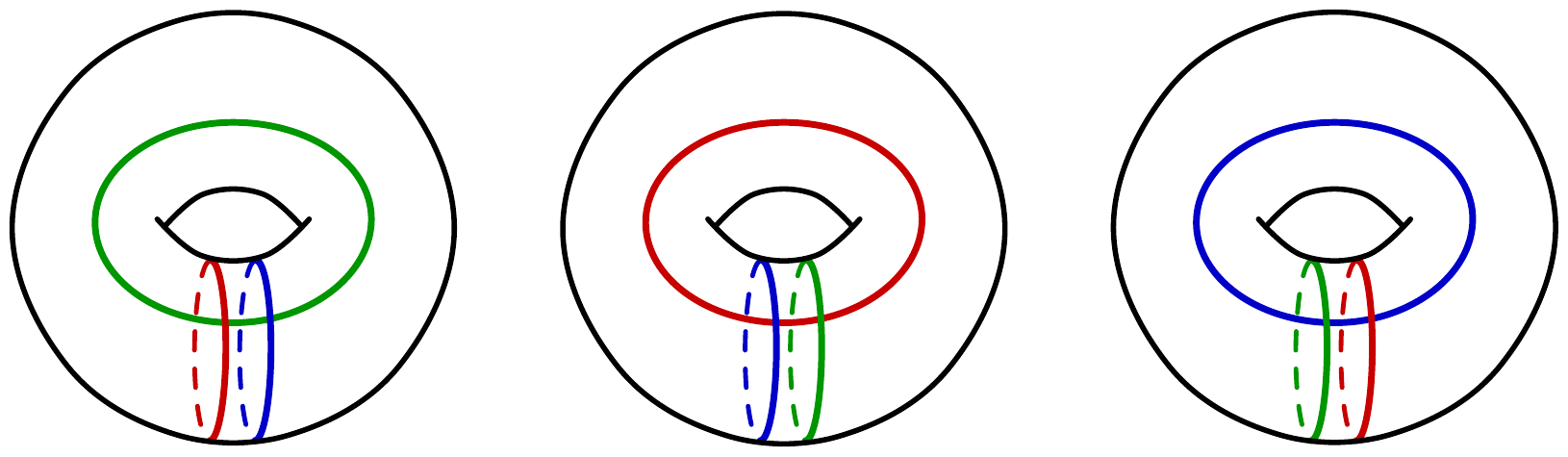}
	\caption{The three unbalanced, genus 1 trisection diagrams $\Ss_1$, $\Ss_2$, and $\Ss_3$ for $S^4$, which are used to perform stabilizations of trisection diagrams.}
	\label{fig:StdDiag}
\end{figure}

To find a trisection diagram for $\Tt_{\phi}$ and to see the connection between R-links and trisections via Heegaard surfaces, we appeal to machinery developed in~\cite{{MeiZup_Characterizing-Dehn_17}}, setting up the next lemma with several more definitions.  Let $L$ be an $n$--component R-link in $S^3$.  A Heegaard surface $\Sigma \subset S^3$ is called \emph{admissible} if $S^3 = H \cup_\Sigma H'$ and $L$ is isotopic to a subset of a core of $H$, so that $H \setminus L$ is a compression body.  Let $H_L$ denote the handlebody resulting from 0--framed surgery on $L$ in $H$.  A genus $g$ Heegaard diagram $(\A,\n)$ for $Y_k = \#^k(S^1 \X S^2)$ is said to be \emph{standard} if $\A$ and $\n$ have $k$ curves in common, and the remaining curves consist of $g-k$ mutually disjoint pairs of curves that intersect each other once.

The next lemma, proved as Lemma 4 in~\cite{{MeiZup_Characterizing-Dehn_17}}, connects R-links to trisections via admissible surfaces:

\begin{lemma}
\label{admissible}
	Let $L$ be an $n$--component R-link with admissible genus $g$ surface $\Sigma$.  Then there is a $(g;0,g-n,n)$--trisection, denoted $\Tt(L,\Sigma)$, of $X_L$ with spine $H' \cup H \cup H_L$.  Moreover, there is a trisection diagram $(\A,\n,\g)$ for $\Tt(L,\Sigma)$ such that
	\begin{enumerate}
		\item $H_{\A} = H'$, $H_{\n} = H$, and $H_{\g} = H_L$;
		\item $L$ is a sublink of $\g$, where $\g$ is viewed as a link framed by $\Sigma$ in $S^3 = H_{\A} \cup H_{\n}$; and
		\item $(\n,\g)$ is a standard diagram for $Y_{g-n}$, where $\n \cap \g = \g - L$.
	\end{enumerate}
\end{lemma}

As an application of Lemma~\ref{admissible}, suppose that $K$ is a fibered, homotopy-ribbon knot with genus $g$ fiber $F$, extension $\phi$, and CG-derivative $L$.  As in the proof of Proposition~\ref{bundletri}, there is a natural Heegaard surface obtained by viewing the fibration of $K$ as an open book decomposition of $S^3$, a fibration $\pi\colon S^3 - K \rightarrow S^1$ so that for each $\theta \in S^1$, $\pi^{-1}(\theta)$ is the interior of a Seifert surface $F_{\theta}$ for $K$.  This Heegaard surface is $\Sigma = F_0 \cup F_{1/2}$, which cuts $S^3$ into two genus $2g$ handlebodies $H$ and $H'$, both viewed as a copy of $F \X I$ with $\pd F \X I$ crushed to $\pd F \X \{pt\}$, such that $F_{1/2} \subset H$ is glued to $F_{1/2} \subset H'$ via the identity map and $F_1 \subset H'$ is glued to $F_0 \subset H$ with the monodromy $\varphi$.  Note that these gluings respect the boundary crushing since $\varphi|_K = \text{id}$.


Given an arc $a \subset F$, let $a_{\theta}$ be the corresponding arc in $F_{\theta}$.  Then every arc $a \subset F$ gives rise to product disks $D(a) \subset H$ with boundary $a_0 \cup a_{1/2}$ and $D'(a) \subset H'$ with boundary $a_{1/2} \cup a_1 = a_{1/2} \cup \varphi(a)_0$.

\begin{lemma}\label{admissCG}
With $K$, $L$, $F$, $\phi$, and $\Sigma$ as above, the surface $\Sigma$ is an admissible surface for~$L$.
\end{lemma}
\begin{proof}
We have established that $S^3 = H \cup H'$, so remains to check that $L$ is isotopic to a core of $H$.  It suffices to find a collection of \emph{dualizing disks} for $L$ in $H$; that is, pairwise disjoint compressing disks $\{D_1,\dots,D_g\}$ such that $|D_i \cap L_j| = \delta_{ij}$.  Let $\{a_1,\dots,a_n\}$ be a collection of arcs in $F$ such that $|a_i \cap L_j| = \delta_{ij}$.  Then the disks $\{D(a_1),\dots,D(a_n)\}$ dualize $L \subset F_0$, and we conclude that $\Sigma$ is an admissible surface for $L$.
\end{proof}

It follows immediately from Lemma \ref{admissible} that $X_L = X_{\phi}$ has a $(2g;0,g,g)$--trisection, and it should come as no surprise that these parameters are the same ones guaranteed by Proposition~\ref{bundletri}.  Indeed, we will see that $\Tt(L,\Sigma)$ and $\Tt_{\phi}$ are two ways of constructing identical trisections.

Let $L$ be a $g$--component derivative of a knot $K$ contained in a Seifert surface $F$.  A collection of \emph{dualizing arcs} for $L$ in $F$ is defined to be a set $\{a_1,b_1,\dots,a_g,b_g\}$ of $2g$ pairwise disjoint and non-isotopic arcs such that $|a_i \cap L_j| = \delta_{ij}$ and $b_i \cap L = \emp$.

\begin{lemma}\label{diagram}
Let $K$, $L$, $F$, $\phi$, and $\Sigma$ be defined as above, and let $\{a_1,b_1,\dots,a_g,b_g\}$ be a collection of dualizing arcs for $L$ in $F$.  Then there is a trisection diagram $(\A,\n,\g)$ for $\Tt(L,\Sigma)$ given by
\begin{eqnarray*}
\A &=& \{\pd D'(a_1), \pd D'(b_1),\dots, \pd D'(a_g),\pd D'(b_g)\}\\
\n &=& \{\pd D(a_1), \pd D(b_1),\dots, \pd D(a_g),\pd D(b_g)\}\\
\g &=& \{L_1,\dots,L_g,\pd D(b_1),\dots,\pd D(b_g)\}.
\end{eqnarray*}
Moreover, $\Tt(L,\Sigma) = \Tt_{\phi}$.
\end{lemma}
\begin{proof}
First, we observe that $\A$ is a cut system for $H$ and $\n$ is a cut system for $H'$.  By definition, the third handlebody $H_L$ in the spine of $\Tt(L,\Sigma)$ is obtained by doing 0--framed Dehn surgery on $L$ in $H$, and thus a set of cut disks for $H_L$ consists of the $g$ disks bounded by $L$ and $g$ disks unaffected by the Dehn surgery on $H$; that is, $\{D(b_1),\dots,D(b_g)\}$.  We conclude that $\g$ bounds a cut system for $H_L$, and $(\A,\n,\g)$ is a trisection diagram for $\Tt(\Sigma,L)$.

For the second claim, let $G_{\A}$, $G_{\n}$, and $G_{\g}$ denote the handlebodies in the spine of $\Tt_{\phi}$ constructed in Proposition~\ref{bundletri}.  By construction, we may suppose that $H_{\A} = G_{\A}$ and $H_{\n} = G_{\n}$.  In addition, $G_{\g}$ is defined by a cut system bounded by curves $L_0 \subset F_0$ and $L_{1/2} \subset F_{1/2}$.  Let $\g' \subset \Sigma$ denote this cut system.  We claim that $\g$ and $\g'$ determine the same handlebody.  Note that $\g$ and $\g'$ contain the $g$ curves in $L_0$ in common, and the remaining $g$ curves $\pd D(b_1),\dots, \pd D(b_g)$ in $\g$ are disjoint from the remaining $g$ curves $L_{1/2}$ in $\g'$.  It follows that $\pd D(b_1),\dots,\pd D(b_g)$ bound disks in $\g'$ and thus $\g$ and $\g'$ determine the same handlebody.  We conclude that $\Tt(\Sigma,L) = \Tt_{\phi}$.
\end{proof}

Next, we connect R-links to the GPRC and Stable GPRC using the theorems of~\cite{{MeiZup_Characterizing-Dehn_17}}.  Trisections $\Tt$ of $X$ and $\Tt'$ of $X'$ can be glued together to get a trisection $\Tt \# \Tt'$ of $X \# X'$ in the obvious way; the connected sum of diagrams for $\Tt$ and $\Tt'$ is a diagram for $\Tt \# \Tt'$.  A trisection $\Tt$ of $S^4$ is called \emph{standard} if $\Tt$ can be expressed as the connected sum of copies of $\Ss_1$, $\Ss_2$, and $\Ss_3$.  Whereas Waldhausen's Theorem~\cite{Wal_Heegaard-Zerlegungen-der-3-Sphare_68} implies that every Heegaard splitting of $S^3$ is standard (i.e. can be expressed as connected sums of standard genus one splittings), the question of whether all trisections of $S^4$ are standard remains open \cite{MeiSchZup_Classification-of-trisections_16}.

It was proved in~\cite{GayKir_Trisecting-4-manifolds_16} that every trisection $\Tt$ of $S^4$ becomes standard after taking the connected sum of $\Tt$ with a standard trisection of $S^4$.  A related notion, defined in~\cite{{MeiZup_Characterizing-Dehn_17}}, is the idea of being $\{i\}$--standard or $\{i,j\}$--standard:  A trisection $\Tt$ is said to be \emph{$\{i\}$--standard} if the connected sum of $\Tt$ with some number of copies of $\Ss_i$ is standard; similarly, $\Tt$ is \emph{$\{i,j\}$--standard} if the connected sum with copies of $\Ss_i$ and $\Ss_j$ is standard.  Note that if $\Tt$ is $\{i\}$--standard or $\{i,j\}$--standard for some $i,j$, the definition implies that $\Tt$ must be a trisection of $S^4$. With this terminology in mind, the uniqueness result of Gay and Kirby implies that every trisection of $S^4$ is $\{1,2,3\}$-standard.

The following is Theorem 3 from~\cite{{MeiZup_Characterizing-Dehn_17}}.

\begin{theorem}\label{equivstab}
	Suppose $L$ is an R-link and $\Sigma$ is an admissible surface for $L$.
	\begin{enumerate}
		\item If $L$ has Property R, then $\Tt(L,\Sigma)$ is $\{2\}$--standard.
		\item The link $L$ has Stable Property R if and only if $\Tt(L,\Sigma)$ is $\{2,3\}$--standard.
	\end{enumerate}
\end{theorem}

Note that the link $L$ has Weak Property R if and only if $\Tt(L,\Sigma)$ is $\{1,2,3\}$--standard; i.e., if and only if $X_L\cong S^4$, by the uniqueness result of~\cite{GayKir_Trisecting-4-manifolds_16}. As a corollary to this theorem, we have the following.

\begin{corollary}\label{prop:23-std}
Let $K$ be a fibered, homotopy-ribbon knot with extension $\phi$.  The following are equivalent.
	\begin{enumerate}
		\item The trisection $\Tt_\phi$ is $\{2,3\}$--standard.
		\item Some CG-derivative corresponding to $\phi$ has Stable Property R.
		\item Every CG-derivative corresponding to $\phi$ has Stable Property R.
	\end{enumerate}
\end{corollary}

\begin{proof}
The proof follows immediately from the observation that if $L$ and $L'$ are distinct CG-derivatives corresponding to $\phi$, Lemmas~\ref{admissCG} and~\ref{diagram} yield that there are admissible surfaces $\Sigma$ and $\Sigma'$ for $L$ and $L'$, respectively, and moreover, $\Tt(L,\Sigma) = \Tt_{\phi} = \Tt(L,\Sigma')$.  Applying Theorem~\ref{equivstab} completes the proof.
\end{proof}

We conclude this section by pointing out a connection between trisections and the Slice-Ribbon Conjecture.  First, we recall a proposition of Abe and Tange (Lemma~5.1 of~\cite{AbeTan_13_A-construction-of-slice}).  For convenience, we present a novel proof here; we acknowledge Christopher Davis, with whom we discovered this simple fact.

\begin{proposition}
\label{prop:ribbon}
	Suppose $L$ is an R-link. If $L$ has Stable Property R, then $L$ is a ribbon link.
\end{proposition}

\begin{proof}
	By hypothesis, $L\sqcup U$ is handleslide-equivalent to $U'$, where $U$ is an unlink of $r$ components and $U'$ is an unlink of $r+n$ components.  Since $U'$ is a ribbon link, our claim will follow if we can show that the result $L''$ of a handleslide on a ribbon link $L'$ is a ribbon link.  Suppose $L''$ is obtained from $L'$ via a slide of component $J'$ of $L'$ over component $J$ of $L'$, producing the new component $J''$ of $L''$.  So, $L'' = (L'\setminus J')\cup J''$.  Let $R'$ be a collection of ribbon disks for $L'$, and let $R_J$ and $R_{J'}$ denote the disks corresponding to $J$ and $J'$.  Let $R_{J''}$ denote the result of taking a push-off of $R_J$ and banding it to $R_{J'}$ along the framed arc corresponding to the handleslide.  It follows that $R'' = (R'\setminus R_{J'})\cup R_{J''}$ is a collection of ribbon disks for $L''$, as desired.  The proof of the proposition by inducting on the number of handleslides necessary to convert $L \sqcup U$ to $U'$.
\end{proof}

It is not known if an R-link with Weak Property R is necessarily ribbon.  Theorem~\ref{equivstab}, Corollary~\ref{prop:23-std}, and Proposition~\ref{prop:ribbon} combine to give the following trisection-theoretic sufficient conditions for a knot or link to be ribbon.

\begin{corollary}
\label{coro:slice-ribbon}
	\ 
	\begin{enumerate}
		\item Let $K$ be a fibered, homotopy-ribbon knot with extension $\phi$.  Then $K$ is ribbon if $\Tt_\phi$ is $\{2,3\}$--standard.
		\item Let $L$ be an R-link and $\Sigma$ an admissible surface for $L$.  Then $L$ is ribbon if $\Tt(L,\Sigma)$ is $\{2,3\}$--standard.
	\end{enumerate}
\end{corollary}

\begin{proof}
	Part (1) follows from Proposition~\ref{prop:ribbon} and Corollary~\ref{prop:23-std}.  Part (2) follows from Proposition~\ref{prop:ribbon} and part (2) of Theorem~\ref{equivstab}.
\end{proof}


\bibliographystyle{amsalpha}
\bibliography{FhRKs}

\end{document}